\documentclass[psamsfonts]{amsart}

\usepackage{setspace}
\usepackage{tipa}
\usepackage{mathrsfs}
\usepackage{amsfonts}
\usepackage{amssymb,amsfonts}
\usepackage[all,arc]{xy}
\usepackage{enumerate}
\usepackage{mathrsfs}
\usepackage{graphicx}
\def\dis{\displaystyle}
\usepackage{geometry}
\newtheorem{thm}{Theorem}[section]
\newtheorem{cor}[thm]{Corollary}
\newtheorem{prop}[thm]{Proposition}
\newtheorem{lem}[thm]{Lemma}

\theoremstyle{definition}
\newtheorem{defn}[thm]{Definition}

\theoremstyle{remark}
\newtheorem{rem}[thm]{Remark}

\numberwithin{equation}{section}

\begin{document}

\date{}
\date{}
\title{On curvature flow with driving force under Neumann boundary conditon in the plane}

\author{Longjie ZHANG}

\date{December, 2015\\Corresponding author University:Graduate School of Mathematical Sciences, The University of Tokyo. Address:3-8-1 Komaba Meguro-ku Tokyo 153-8914, Japan. Email:zhanglj@ms.u-tokyo.ac.jp, zhanglj919@gmail.com}

\maketitle

\begin{minipage}{140mm}

{{\bf Abstract:} We consider a family of axisymmetric curves evolving by its mean curvature with driving force in the half space. We impose a boundary condition that the curves are perpendicular to the boundary for $t>0$, however, the initial curve intersects the boundary tangentially. In other words, the initial curve is oriented singularly. We investigate this problem by level set method and give some criteria to judge whether the interface evolution is fattening or not. In the end, we can classify the solutions into three categories and provide the asymptotic behavior in each category. Our main tools in this paper are level set method and intersection number principle.

 }

{\bf 2010MSC:} 35A01, 35A02, 35K55, 53C44.

\end{minipage}

$$$$

\section{Introduction}\large

 This paper studies the planar curvature flow with driving force and Neumann boundary condition of the form
\begin{equation}\label{eq:cur}
V=-\kappa+A\, \ \textrm{on}\ \Gamma(t)\subset \Omega ,
\end{equation}
\begin{equation}\label{eq:Neum1}
\Gamma(t)\perp\partial \Omega,
\end{equation}
\begin{equation}\label{eq:initial1}
\Gamma(0)=\Lambda_0,
\end{equation}
where $\Omega=\{(x,y)\in \mathbb{R}^2\mid x\geq 0\}$, $V$ is the outer normal velocity of $\Gamma(t)$, $\kappa$ is the curvature of $\Gamma(t)$ and the sign is chosen such that the problem is parabolic. $A$ called driving force is a positive constant.

In this paper, we consider the initial curve $\Lambda_0$ is closed, smooth and given by 
$$
\Lambda_0=\{(x,y)\in \mathbb{R}^2\mid |y|=u_0(x), 0\leq x\leq b_0\},
$$
for $u_0(x)\in C[0,b_0]\cap C^{\infty}(0,b_0)$. By the assumption of $\Lambda_0$, there hold 
$$
u_0(x)>0,\ 0<x<b_0
$$
and 
$$
u_0(0)=u_0(b_0)=0,\ u_{0}^{\prime}(0)=-u_{0}^{\prime}(b_0)=\infty.
$$
\begin{figure}[htbp]
	\begin{center}
            \includegraphics[height=6.0cm]{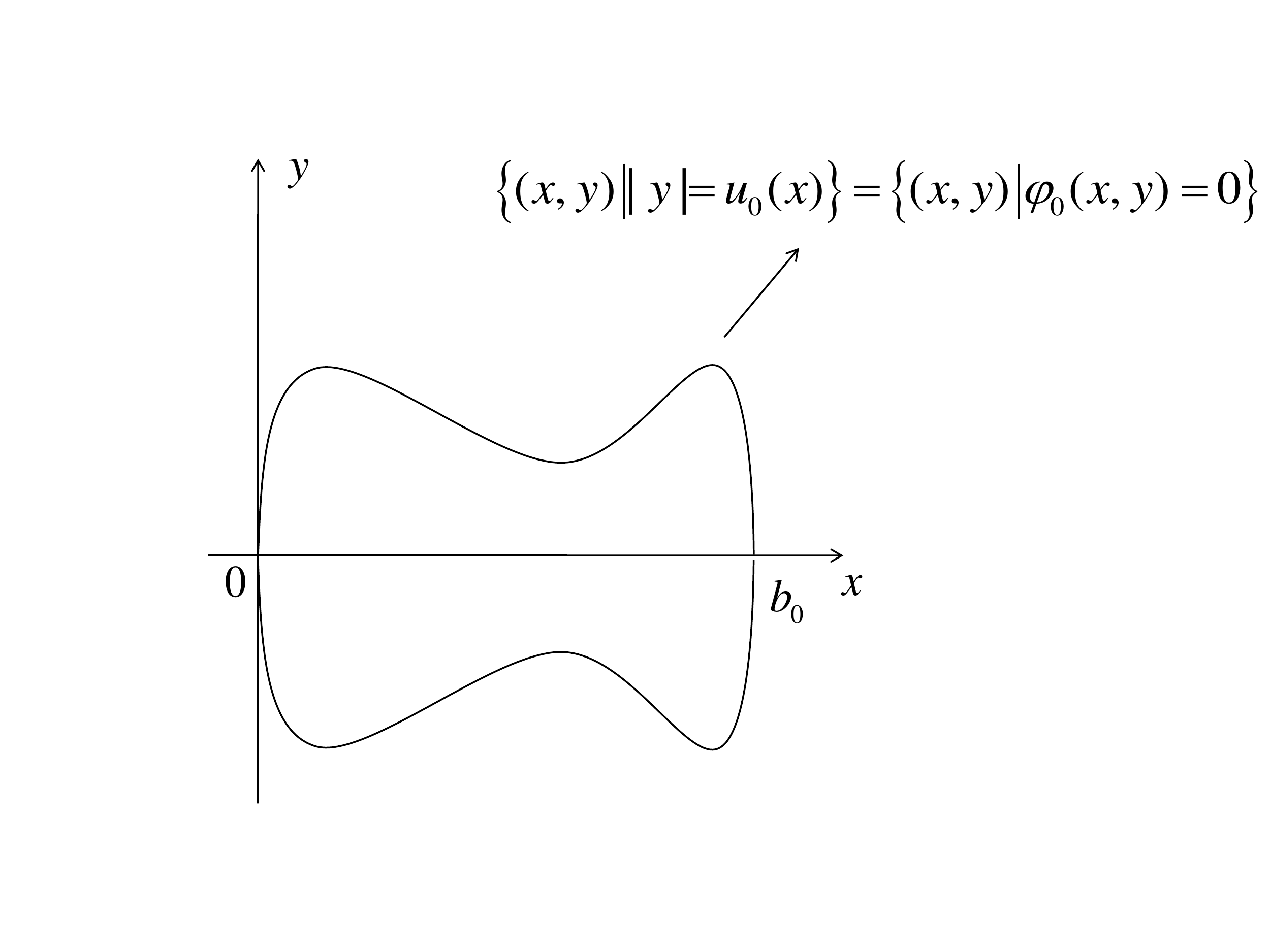}
		\vskip 0pt
		\caption{Initial curve}
        \label{fig:u0}
	\end{center}
\end{figure}
Before giving our main results, we first consider another problem.
\begin{equation}
V=-\kappa+A\, \ \textrm{on}\ \Lambda(t)\subset \mathbb{R}^2,\tag{\ref{eq:cur}*}
\end{equation}
\begin{equation}
\Lambda(0)=\Lambda_0,\tag{\ref{eq:initial1}*}
\end{equation}
We consider this problem by level set method. Seeing the theory in \cite{G}, there exists unique viscosity solution $\phi$ of the following level set equation
$$
\left\{
\begin{array}{lcl}
\dis{\phi_t=|\nabla \phi|\textmd{div}(\frac{\nabla \phi}{|\nabla \phi|})+A|\nabla \phi|}\ \textrm{in}\  \mathbb{R}^2\times(0,T),\\
\phi(x,y,0)=a_1(x,y),
\end{array}
\right.
$$
where $a_1(x,y)$ satisfies $\Lambda_0=\{(x,y)\mid a_1(x,y)=0\}$. The results in appendix show that the zero set of $\phi$ is not fattening. Indeed, thanks to Theorem \ref{thm:gu}, the zero set of $\phi$ can be written into
$$
\Lambda(t)=\{(x,y)\mid \phi(x,y,t)=0\}=\{(x,y)\in\mathbb{R}^2\mid |y|=v(x,t), a_*(t)\leq x\leq b_*(t)\},\ 0<t<T.
$$ 
Moreover, $(v,a_*,b_*)$ is the solution of the following free boundary problem
\begin{equation}
\left\{
\begin{array}{lcl}
\dis{u_t=\frac{u_{xx}}{1+u_x^2}+A\sqrt{1+u_x^2}},\ x\in(a_*(t),b_*(t)),\ 0<t< \delta,\\
u(a_*(t),t)=0,\ u(b_*(t),t)=0,\ 0\leq t< \delta,\\
u_x(a_*(t),t)=\infty,\ u_x(b_*(t),t)=-\infty,\ 0\leq t<\delta,\\
u(x,0)=u_0(x),\ 0\leq x\leq b_0.
\end{array}
\right.\tag{*}
\end{equation}
And $a_*$ and $b_*$ are called the end points of $\Lambda(t)$.

Here we give our main results.

\begin{thm}\label{thm:exist}
If there exists $\delta$ such that $a_*(t)<0$, for $0<t<\delta$, then there exist $T_1>0$ and a unique smooth family of smooth curves $\Gamma(t)$ satisfying (\ref{eq:cur}), (\ref{eq:Neum1}), $0\leq t<T_1$, and satisfying (\ref{eq:initial1}) in the sense that $\lim\limits_{t\rightarrow0^+}d_H(\Gamma(t),\Lambda_0)=0$. Moreover, $\Gamma(t)$ can be written into $\Gamma(t)=\{(x,y)\in\Omega\mid |y|=u(x,t),\ 0\leq x\leq b(t)\}$. And $(u,b)$ is the unique solution of the following free boundary problem
\begin{equation}\label{eq:1graph}
u_t=\frac{u_{xx}}{1+u_x^2}+A\sqrt{1+u_x^2},,\ 0<t<T_1,\ 0<x<b(t),
\end{equation}
\begin{equation}\label{eq:1bounday}
u(b(t),t)=0,\ u_x(b(t),t)=-\infty,\ u_x(0,t)=0,\ 0<t<T_1,
\end{equation}
\begin{equation}\label{eq:1initial}
u(x,0)=u_0(x),\ 0\leq x\leq b_0.
\end{equation}
\end{thm}
Here $d_H(A,B)$ denotes the Hausdorff distance defined as
$$
d_H(A,B)=\max\{\sup\limits_{x\in A}\inf\limits_{y\in B}d(x,y),\sup\limits_{y\in B}\inf\limits_{x\in A}d(x,y)\}.
$$

Let $T$ be the maximal smooth time given by
$$
T=\sup\{t\mid \Gamma(s)\ \text{is}\ \text{smooth}, 0<s<t\}.
$$

\begin{thm}\label{thm:threecondition}
 (Classification) Under the same assumptions of Theorem \ref{thm:exist}, denote

$h(t)=\max\limits_{0\leq x\leq b(t)}u(x,t)$.Then $\Gamma(t)$ must fulfill one of the following situations.

(1). (Expanding) The existence time $T=\infty$ and both $h(t)$ and $b(t)$ tend to $\infty$, as $t\rightarrow\infty$.

(2). (Bounded) The existence time $T=\infty$ and both $h(t)$ and $b(t)$ are bounded from above and below by two positive constants, as $t\rightarrow\infty$.

(3). (Shrinking) The existence time $T<\infty$ and both $h(t)$ and $b(t)$ tend to 0, as $t\rightarrow T$.

\end{thm}

\begin{thm}\label{thm:asym}
 (Asymptotic behavior) Under the same assumptions of Theorem \ref{thm:exist}. Then $\Gamma(t)$ must fulfill one of following three conditions.

(1). (Expanding) Assume that $T=\infty$ and that both $h(t)$ and $b(t)$ tend to $\infty$ as $t\rightarrow\infty$, there exist $t_0>0$, $R_1(t)$, $R_2(t)$ such that
$$B_{R_1(t)}((0,0))\cap\{x>0\}\subset U(t) \subset B_{R_2(t)}((0,0))\cap\{x>0\},\ t>t_0$$
where $U(t)=\{(x,y)\in\mathbb{R}^2\mid |y|<u(x,t),\ x>0\}$. Moreover $\lim\limits_{t\rightarrow\infty}R_1(t)/t=\lim\limits_{t\rightarrow\infty}R_2(t)/t=A$.

(2). (Bounded) Assume that $T=\infty$ and that both $h(t)$ and $b(t)$ are bounded from above and below by two positive constants for $t>0$. Then $\lim\limits_{t\rightarrow\infty}d_H(\Gamma(t),\partial B_{1/A}((0,0))\cap\{x\geq0\})=0$.

(3). (Shrinking) Assume that $T<\infty$ and that both $h(t)$ and $b(t)$ tend to 0 as $t\rightarrow T$. Then the flow $\Gamma(t)$ shrinks to a point at $t=T$.
\end{thm}
We extend $\Gamma(t)$ by even and still denote the extended curve by $\Gamma(t)$. Then problem (\ref{eq:cur}), (\ref{eq:Neum1}), (\ref{eq:initial1}) is equivalent to the following problem in whole space 
\begin{equation}\label{eq:cureven}
V=-\kappa+A,\ \text{on}\ \Gamma(t)\subset \mathbb{R}^2,
\end{equation}
\begin{equation}\label{eq:initialeven}
\Gamma(0)=\Gamma_0=\Lambda_0\cup\{(-x,y)\in\mathbb{R}^2\mid (x,y)\in\Lambda_0\}.
\end{equation}
If we extends $u_0$ by even(still denoted by $u_0$), then obviously
\begin{equation}\label{eq:ineven}
\Gamma_0=\{(x,y)\in\mathbb{R}^2\mid |y|=u_0(x)\}.
\end{equation}
In this paper we consider the problem (\ref{eq:cureven}), (\ref{eq:initialeven}) instead of the problem (\ref{eq:cur}), (\ref{eq:Neum1}), (\ref{eq:initial1}).

We next give a sufficient result to have a fattening phenomenon. The definition of interface evolution and fattening are given in section 2.

\begin{thm}\label{thm:fattening1} (Fattening)

 If there exists $\delta$ such that $a_*(t)\geq0$, for $0<t<\delta$, the interface evolution $\Gamma(t)$ for (\ref{eq:cureven}) with initial data $\Gamma_0$ is fattening.
\end{thm}

Theorem \ref{thm:exist} and Theorem \ref{thm:fattening1} can be explained by Figure \ref{fig:exist} and \ref{fig:fattening1}. $\varphi$ in Figure \ref{fig:exist} and \ref{fig:fattening1} is given by the unique viscosity solution of 
$$
\left\{
\begin{array}{lcl}
\dis{\varphi_t=|\nabla \varphi|\textmd{div}(\frac{\nabla \varphi}{|\nabla \varphi|})+A|\nabla \varphi|}\ \textrm{in}\  \mathbb{R}^2\times(0,T),\\
\varphi(x,y,0)=a_2(x,y),
\end{array}
\right.
$$
where $a_2(x,y)$ satisfies $\Gamma_0=\{(x,y)\mid a_2(x,y)=0\}$. Let $\Gamma(t)=\{(x,y)\mid \varphi(x,y,t)=0\}$.

\begin{figure}[htbp]
	\begin{center}
            \includegraphics[height=7.0cm]{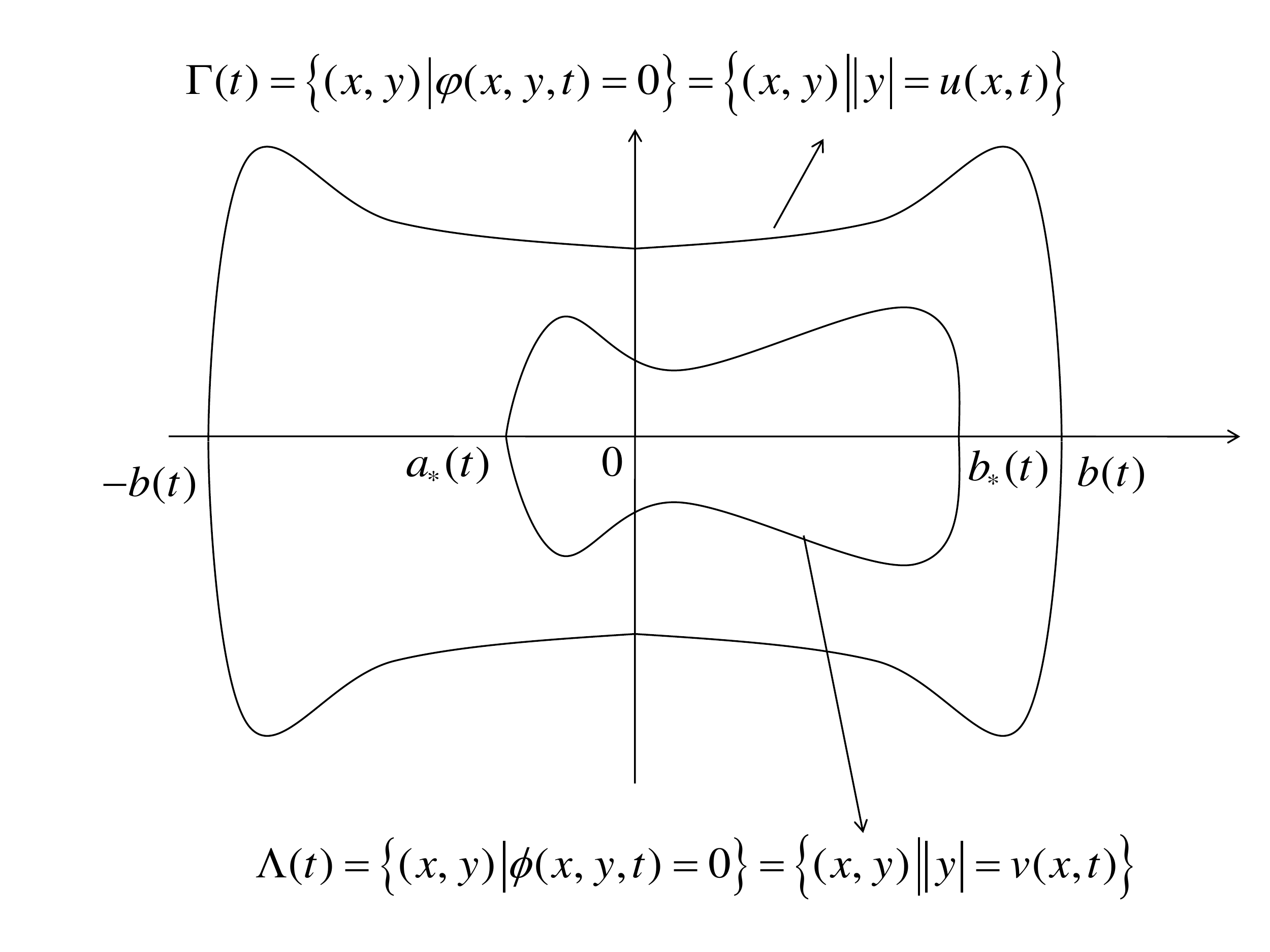}
		\vskip 0pt
		\caption{$a_*(t)<0$ in Theorem \ref{thm:exist}}
        \label{fig:exist}
	\end{center}
\end{figure}

\begin{figure}[htbp]
	\begin{center}
            \includegraphics[height=7.0cm]{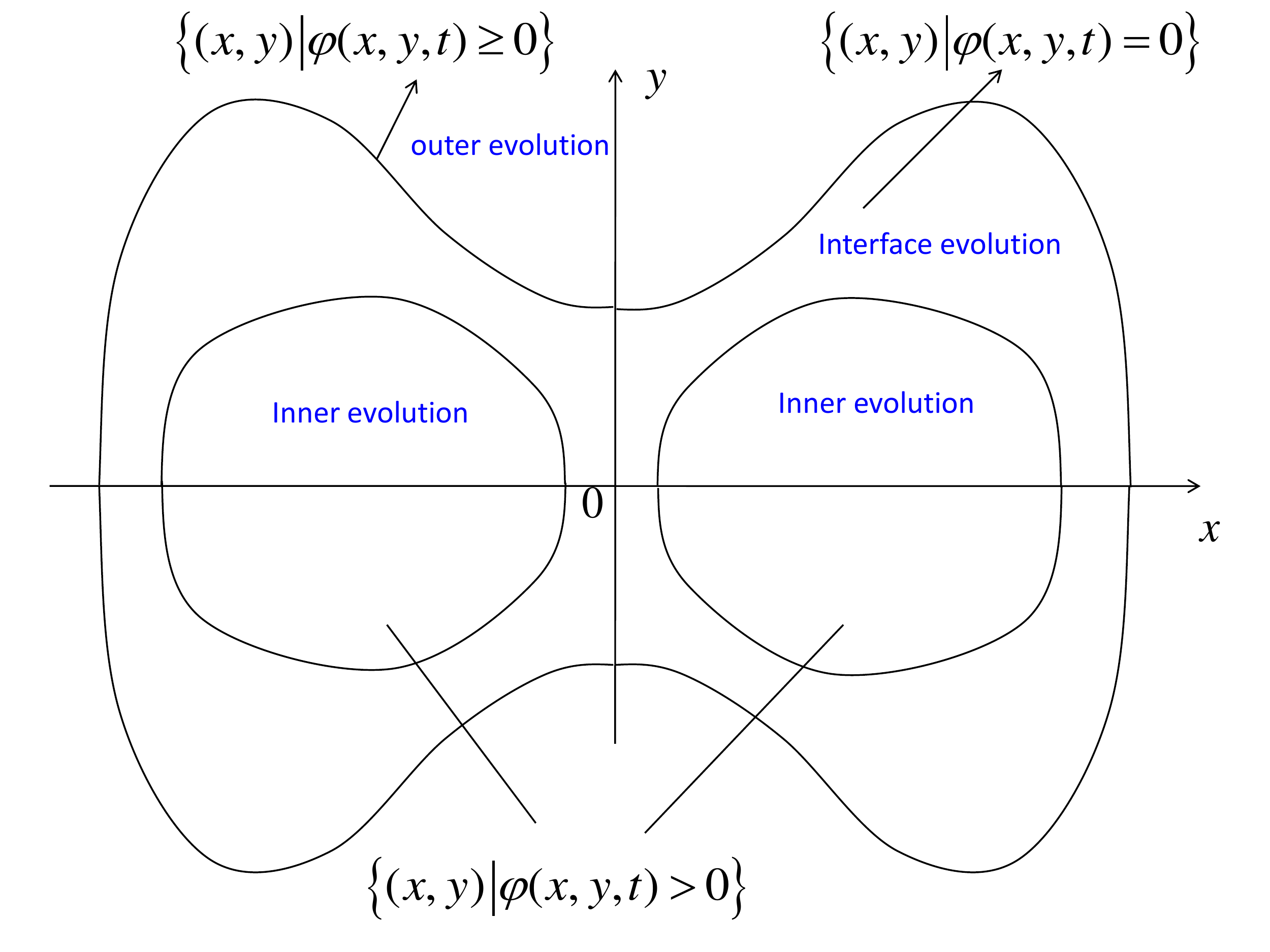}
		\vskip 0pt
		\caption{$a_*(t)\geq0$ in Theorem \ref{thm:fattening1}}
        \label{fig:fattening1}
	\end{center}
\end{figure}

\begin{rem}
The assumptions for $a_*(t)$ in Theorem \ref{thm:exist} and \ref{thm:fattening1} seem not to be understood easily. Here we explain the assumptions by giving some sufficient conditions. 

Denoting $\kappa(O)$ as the curvature of $\Lambda_0$ at origin, it is easy to see that
$$
\kappa(O)=-\lim\limits_{x\rightarrow0^{+}}u_0^{\prime\prime}/(1+(u_0^{\prime})^2)^{3/2}.
$$
Then we can prove that if 

(a). $\kappa(O)<A$, there holds $a_*(t)<0$, for $t$ small;

(b). $\kappa(O)>A$, there holds $a_*(t)>0$, for $t$ small.
\\Since
$$
a_*^{\prime}(0)=\kappa(O)-A.
$$
\end{rem}
{\bf The role of $a_*(t)$ in main theorem.} Let $(u,b)$ is the unique solution of the free boundary problem (\ref{eq:1graph}), (\ref{eq:1bounday}), (\ref{eq:1initial}). Obviously, the flow $\Gamma^*(t)=\{(x,y)\mid |y|=u(x,t),\ 0\leq x\leq b(t)\}$ satisfies (\ref{eq:cur}), (\ref{eq:Neum1}), (\ref{eq:initial1}) naturally.

Let $(v,a_*,b_*)$ be the solution of the problem (*). If $a_*(t)\geq0$, $0<t<\delta$, the curves 
$$
\Lambda(t)=\{(x,y)\mid|y|=v(x,t),\ a_*(t)\leq x\leq b_*(t)\}
$$ 
are located in $\{x\geq0\}$.  

Note that $\Gamma^*(0)=\Lambda(0)=\Lambda_0$. However, $\Gamma^*(t)$ are perpendicular to $y$-axis and the family $\{\Lambda(t)\}$ evolves freely. This means that there exist two types of flows $\Gamma^*(t)$ and $\Lambda(t)$ evolving by $V=-\kappa+A$ with the same initial curve $\Lambda_0$. This can be considered as non-uniqueness. Indeed, seeing the proof of Theorem \ref{thm:fattening1}, the flow given by extending $\Gamma^*(t)$ evenly is the boundary of closed evolution and  the flow given by extending $\Lambda(t)$ evenly is the boundary of open evolution.

If $a_*(t)<0$, $0<t<\delta$, the problem (*) will not make sense in the half space $\{x\ge0\}$. But the solution given by (*) plays the role of a sub-solution (in the proof of Lemma \ref{lem:closebou}). Using this sub-solution, the boundaries of the open evolution and closed evolution are away from the $x$-axis. By the uniqueness result(Proposition 5.4), we can prove they are the same. 

 In the curve shortening flow------$A=0$, since $a_*(t)\geq0$ always holds, the interface evolution is fattening.  

{\bf Motivation.} This research is motivated by \cite{MNL}, the mean curvature flow with driving force under the Neumann boundary condition in a two-dimensional cylinder with periodically undulating boundary. 
\begin{figure}[htbp]
 \begin{minipage}{0.49\hsize}
  \begin{center}
   \includegraphics[width=10cm]{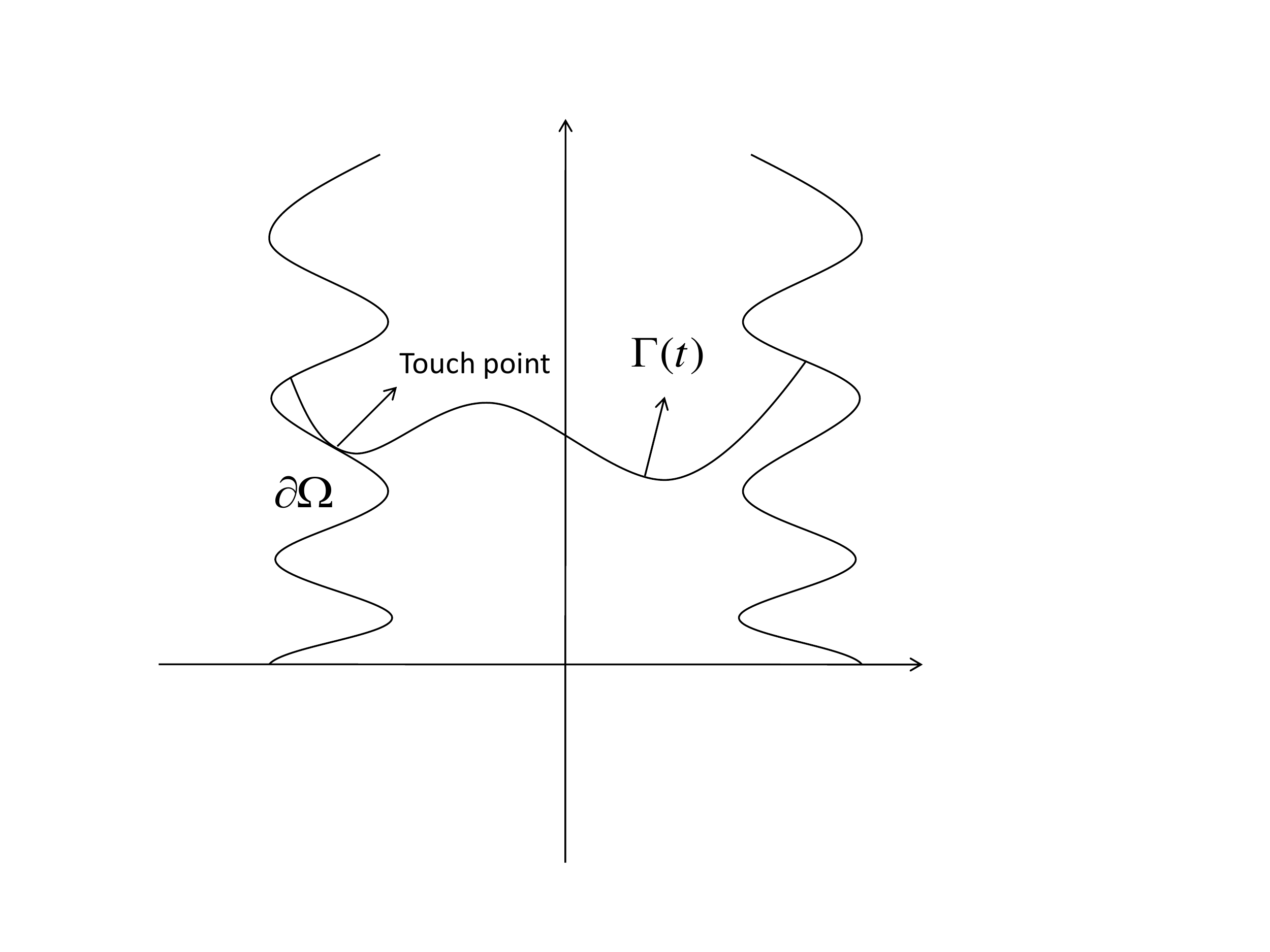}
  \end{center}
  \caption{Curve touching }
  \label{fig:intro1}
 \end{minipage}
 \begin{minipage}{0.49\hsize}
  \begin{center}
   \includegraphics[width=10cm]{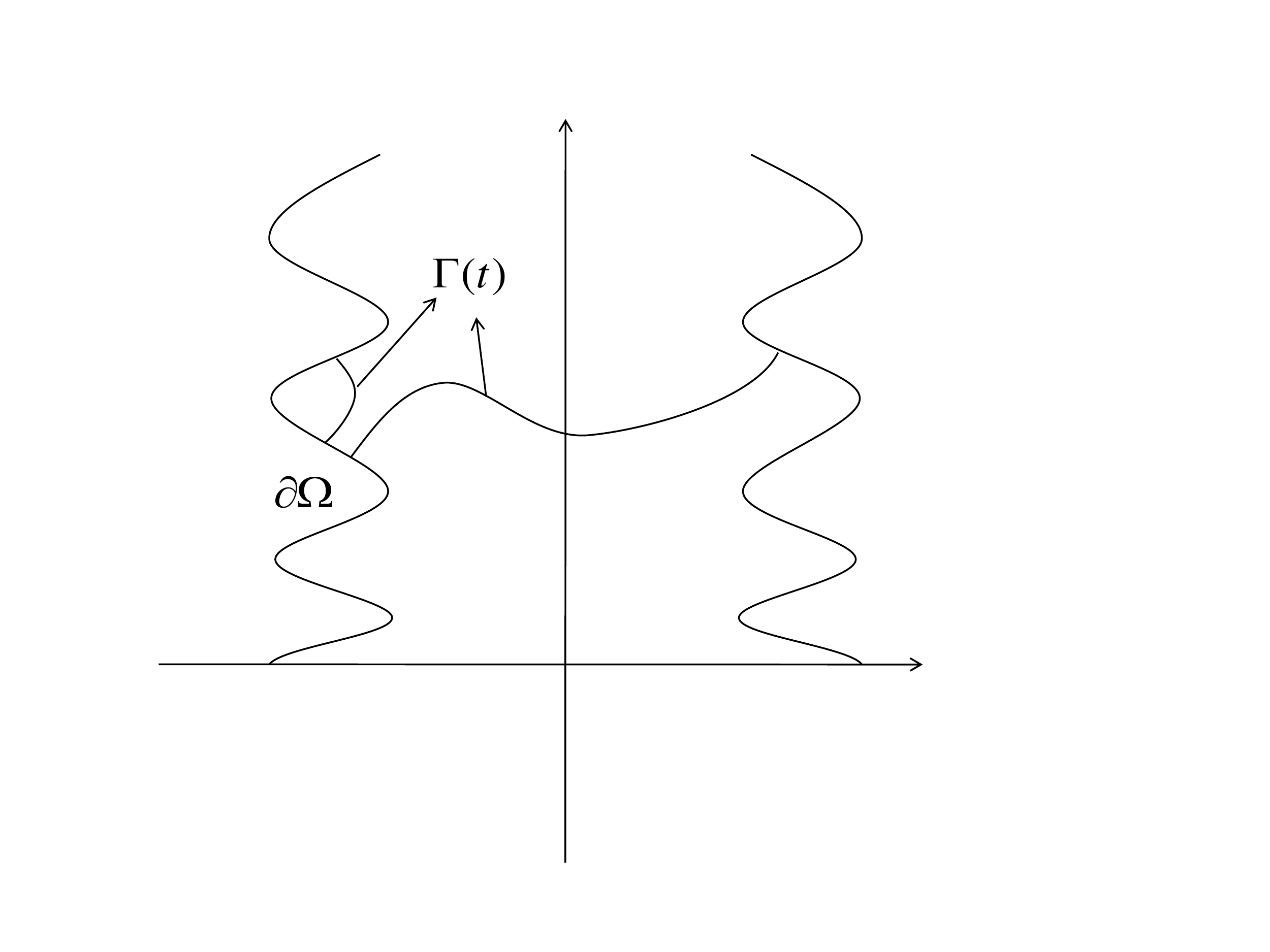}
  \end{center}
  \caption{After touching}
  \label{fig:intro2}
 \end{minipage}
\end{figure}
 In \cite{MNL}, they only consider the condition that for initial curve $\{(x,y)\in\mathbb{R}^2\mid y=u_0(x)\}$ with $|u_0^{\prime}(x)|<M$ for some $M$. They show that the interior point of $\Gamma(t)=\{(x,y)\in\mathbb{R}^2\mid y=u(x,t)\}$ never touches the boundary and $\Gamma(t)$ remains graph. Therefore, the problem can be studied by the classical parabolic theory. If removing the assumption $|u_0^{\prime}(x)|<M$, when $u(x,t)$ touches the boundary, the singularity will develop(Figure \ref{fig:intro1}). Noting Figure \ref{fig:intro2}, after touching, $\Gamma(t)$ will possibly separate into two parts and become non-graph($\Gamma(t)$ can't be represented by $y=u(x,t)$). This makes us analyze what will happen after touching boundary. Since $\Gamma(t)$ may become non-graph, we want to use the level set method established by \cite{CGG}; see also Evans and Spruck \cite{ES2} for the mean curvature flow, where fattening phenomenon is first observed. Therefore the main task in this paper is to study whether the interface evolution is fattening or not. The notion of the level set solution is introduced in Section 2.

{\bf A short review for mean curvature flow.} For the classical mean curvature flow: $A=0$ in (\ref{eq:cur}), there are many results. Concerning this problem, Huisken \cite{H} shows that any solution that starts out as a convex, smooth, compact surface remains so until it shrinks to a "round point" and its asymptotic shape is a sphere just before it disappears. He proves this result for hypersurfaces of $\mathbb{R}^{n+1}$ with $n\geq2$, but Gage and Hamilton \cite{GH} show that it still holds when $n=1$, the curves in the plane. Gage and Hamilton also show that embedded curve remains embedded, i.e. the curve will not intersect itself. Grayson \cite{Gr} proves the remarkable fact that such family must become convex eventually. Thus, any embedded curve in the plane will shrink to "round point" under curve shortening flow. But in higher dimensions it is not true. Grayson \cite{Gr2} also shows that there exists a smooth flow that becomes singular before shrinking to a point. His example consisted of a barbell: two spherical surfaces connected by a sufficiently thin "neck". In this example, the inward curvature of the neck is so large that it will force the neck to pinch before shrinking. In \cite{AAG}, A. Altschuler, S. B. Angenent and Y. Giga study the flow whose initial hypersurface is a compact, rotationally symmetric hypersurface but pinching on $x$-axis by level set method. They proved the hypersurface will separate into two smooth hypersurfaces after pinching.

{\bf Main method.} In this paper, one of the most important tools is the intersection number principle. It was also used in \cite{AAG} that the intersection number between two families evolving by mean curvature flow is non-increasing. But for the problem with driving force, their intersection number may increase. In \cite{GMSW}, they give the extended intersection number principle for the following free boundary problem called (Q)
\begin{equation}
\left\{
\begin{array}{lcl}
\dis{u_t=\frac{u_{xx}}{1+u_x^2}+A\sqrt{1+u_x^2}},\ x\in(a(t),b(t)),\ 0<t< T,\\
u(a(t),t)=0,\ u(b(t),t)=0,\ 0\leq t< T,\\
u_x(a(t),t)=\tan\theta_-(t),\ u_x(b(t),t)=-\tan\theta_+(t),\ 0\leq t< T,\\
u(x,0)=u_0(x),\ a(0)\leq x\leq b(0),
\end{array}
\right.\tag{Q}
\end{equation}
where $0<\theta_{\pm}<\pi/2$.

If $(u_1,a_1,b_1)$, $(u_2,a_2,b_2)$ are the solutions of (Q) with $\theta_{\pm}^1$, $\theta_{\pm}^2$, $u_0^1$ and $u_0^2$, the intersection number between
$$
\{(x,y)\mid y=u_1(x,t),a_1(t)\leq x\leq b_1(t)\}
$$
and
$$
\{(x,y)\mid y=u_2(x,t),a_2(t)\leq x\leq b_2(t)\}
$$
will increase possibly, by some simple examples(here we omit it). In \cite{GMSW}, they find a non-increasing quantity. If $u_1$ and $u_2$ are extended by straight line, such that the extended functions $u_1^{*}$, $u_2^{*}$ are in $C^1(\mathbb{R})$, then the intersection number between $u_1^{*}$ and $u_2^{*}$ is non-increasing provided that $\theta_{\pm}^1 \neq \theta_{\pm}^2$. If $\theta_{+}^1= \theta_{+}^2$, the intersection number will not increase provided that $b_1(t)\neq b_2(t)$ and decrease at $t_0$ satisfying $b_1(t_0)=b_2(t_0)$. Similarly for $a(t)$. These results are called ``extended intersection number principle''. 

As we observing before, in this paper, the curve symmetric around $x$-axis is considered. Under this condition, $\theta_{\pm}$ in problem (Q) satisfy $\theta_{\pm}=\pi/2$. If we extend $u_1$ and $u_2$ by vertical straight line, it will be seen that the extended $C^1$ curve $\gamma_1(t)$ and $\gamma_2(t)$(Since $\theta_{\pm}=\pi/2$, the extended curve will not be graph) will intersect with each other even if the intersection number between $\gamma_1(0)$ and $\gamma_2(0)$ is zero. We will investigate the intersection number in Section 4.

The rest of this paper is organized as follows. In Section 2, we introduce the level set method established by \cite{CGG}. The definition of open evolution, closed evolution, fattening and the basic knowledge including comparison principle, monotone convergence theorem and so on are given in this section. In Section 3, we prove the Evans-Spruck estimate(also called gradient interior estimate). In Section 4, the results of intersection number are investigated and their application are given. In Section 5, we give the proof of Theorem \ref{thm:exist} and Theorem \ref{thm:fattening1}. In Section 6, we study the possible formations of singularity. In Section 7, we classify the solution given by Theorem \ref{thm:exist} and prove the asymptotic behavior in each category(Theorem \ref{thm:threecondition} and \ref{thm:asym}). In Section 8, we give another non-fattening result in $(n+1)$-dimension with a type of initial hypersurfaces.


\section{Level set method}

Since the initial curve $\Gamma_0$ given in (\ref{eq:initialeven}) has singularity at $(0,0)$, the equation $V=-\kappa+A$ does not make sense at $t=0$. Therefore, we want to apply the level set method to our problem. In this section, we introduce the level set method in $\mathbb{R}^N$. For $\Gamma(t)$ being a smooth family of smooth, closed, compact hypersurfaces in $\mathbb{R}^{N}$. Assume there exists $\psi(x,t)$ such that $\Gamma(t)=\{x|\psi(x,t)=0,x\in\mathbb{R}^N\}$. If $\Gamma(t)$ evolves by (\ref{eq:cureven}), we can derive $\psi(x,t)$ satisfying 
\begin{equation}\label{eq:level}
\dis{\psi_t=|\nabla \psi |\textmd{div}(\frac{\nabla \psi}{|\nabla \psi|})+A|\nabla \psi|\ \textrm{in}\  \mathbb{R}^N\times(0,T)}.
\end{equation}
Equation (\ref{eq:level}) is called the level set equation of (\ref{eq:cureven}). Theorem 4.3.1 in \cite{G} gives the existence and uniqueness of the viscosity solution for (\ref{eq:level}) with $\psi(x,0)=\psi_0(x)$. Where $\psi_0(x)$ is a bounded and uniform continuous function.
\begin{figure}[htbp]
	\begin{center}
            \includegraphics[height=5cm]{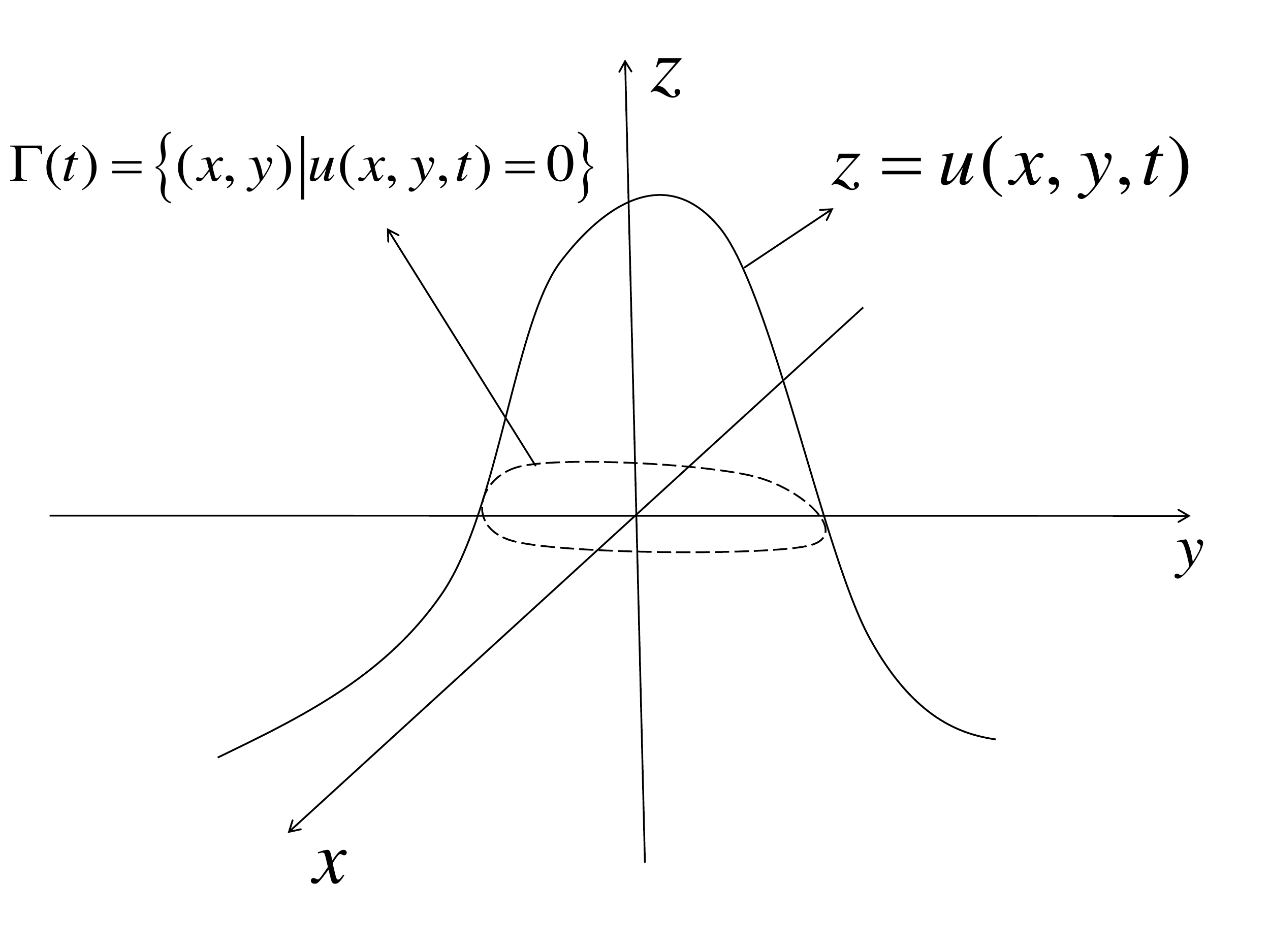}
		\vskip 0pt
		\caption{Level set method in $\mathbb{R}^2$}
        \label{fig:Levelsetmethod}
	\end{center}
\end{figure}

{\bf Level set method} Using the solution of level set equation, we introduce the level set method. 
\begin{defn}\label{def:evo} (1) Let $D_0$ be a bounded open set in $\mathbb{R}^N$. A family of  open sets $\{D(t)\mid D(t)\subset \mathbb{R}^N\}_{0<t<T}$ is called an $(generlized)$ $open$ $evolution$ of (\ref{eq:cureven}) with initial data $D_0$ if there exists a viscosity solution $\psi$ of (\ref{eq:level}) that satisfies
$$
D(t)=\{x\in\mathbb{R}^N \mid \psi(x,t)>0\},\ D_0=\{x\in\mathbb{R}^N\mid \psi(x,0)>0\}.
$$

(2) Let $E_0$ be a bounded closed set in $\mathbb{R}^N$. A family of closed sets $\{E(t)\mid E(t)\subset \mathbb{R}^N\}_{0<t<T}$ is called a $(generlized)$ $closed$ $evolution$ of (\ref{eq:cureven}) with initial data $E_0$ if there exists a viscosity solution $\psi$ of  (\ref{eq:level}) that satisfies
$$
E(t)=\{x\in\mathbb{R}^N \mid \psi(x,t)\geq0\},\ E_0=\{x\in\mathbb{R}^N\mid \psi(x,0)\geq0\}.
$$

The set $\Gamma(t)=E(t)\setminus D(t)$ is called an \textit{(generalized) interface evolution} of  (\ref{eq:cureven}) with initial data $\Gamma_0=E_0\setminus D_0$.
\end{defn}

\begin{rem}\label{rem:lev} (1) For open set $D_0$, we often choose
$$
\psi(x,0)=\max\{\textrm{sd}(x,\partial D_0),-1\}
$$
where
$$
\textrm{sd}(x,\partial D_0)=\left\{
\begin{array}{lcl}
\textrm{dist}(x,\partial D_0), \ x\in D_0,\\
-\textrm{dist}(x,\partial D_0), \ x\notin D_0.
\end{array}\right.
$$

(2) Seeing  that the choice of $\psi(x,0)$ isn't unique, but by the Theorem 4.2.8 in \cite{G}, the open evolution $D(t)$ and closed evolution $E(t)$ are both independent on the choice of $\psi(x,0)$.

(3) In generally, even if $E_0=\overline{D_0}$, we can not guarantee $E(t)=\overline{D(t)}$. If $E(t)\setminus D(t)$ has interior points for some $t$, we call the interface evolution is fattening. Respectively, if $E(t)=\overline{D(t)}$, for all $0<t<T$, we say the interface evolution is regular. Therefore, in the proof of Theorem \ref{thm:exist}, it is sufficient to prove $\partial D(t)=\partial E(t)$. In the proof of Theorem \ref{thm:fattening1}, it is sufficient to prove that there exists a ball $B$ such that $B\subset E(t)\setminus D(t)$.
\end{rem}

We now list some fundamental properties of open evolution and closed evolution of  (\ref{eq:cureven})(Chapter 4 in \cite{G}).

\begin{thm}\label{thm:semi}
(Semigroups)\cite{G}. Denote $U(t)$ and $M(t)$ being the operators such that
$U(t)D_0=D(t)$ and $M(t)E_0=E(t)$, for $t>0$. Then we have $U(t)D(s)=D(t+s)$ and $M(t)E(s)=E(t+s)$, for any $t>0$, $s>0$. 
\end{thm}

\begin{thm}\label{thm:order}(Order preserving property)\cite{G}. Let $D_0$, $D_0^{\prime}$ be two open sets in $\mathbb{R}^N$ and let $E_0$, $E_0^{\prime}$ be two closed sets in $\mathbb{R}^N$. Then

(1) $U(t)D_0\subset U(t)D_0^{\prime}$, if $D_0\subset D_0^{\prime}$;

(2) $M(t)E_0\subset M(t)E_0^{\prime}$, if $E_0\subset E_0^{\prime}$;

(3) $U(t)D_0\subset M(t)E_0^{\prime}$, if $D_0\subset E_0^{\prime}$;

(4) $E_0\subset D_0$ and $\textrm{dist}(E_0,\partial D_0)>0$, then $M(t)E_0\subset U(t)D_0$.
\end{thm}

\begin{thm}\label{thm:mon}(Monotone convergence)\cite{G}.

(1) Let $D(t)$ and $\{D_j(t)\}$ be open evolutions with initial data $D_0$ and $D_{j0}$ respectively. If $D_{j0}\uparrow D_0$, then $D_j(t)\uparrow D(t)$, $t>0$, i.e., $\bigcup\limits_{j\geq1}D_j(t)=D(t)$;

(2) Let $E(t)$ and $\{E_j(t)\}$ be closed evolutions with initial data $E_0$ and $E_{j0}$ respectively. If $E_{j0}\downarrow E_0$, then $E_j(t)\downarrow E(t)$, $t>0$, i.e., $\bigcap\limits_{j\geq1}E_j(t)=E(t)$.
\end{thm}
\begin{thm}\label{thm:conti}(Continuity in time)\cite{G}. Let $D(t)$ and $E(t)$ be open and closed evolutions, respectively.

(1a) $D(t)$ is a lower semicontinuous function of $t\in[0,T)$, in the sense that for any $t_0\geq 0$, and sequence $x_n\in (D(t_n))^c$ with $x_n\rightarrow x_0$, $t_n\rightarrow t_0$, the limit $x_0\in (D(t_0))^c$. If $D(0)$ is bounded so that $\mathcal{C}_{\epsilon}(D(t_0))$ is compact, this implies that for any $t_0\geq0$, $\epsilon>0$ there is a $\delta>0$ such that $|t-t_0|<\delta$ implies $D(t)\supset\mathcal{C}_{\epsilon}(D(t_0))$.

(1b) $E(t)$ is an upper semicontinuous function of $t\in[0,T)$, in the sense that for any $t_0\geq 0$, and sequence $x_n\in E(t_n)$ with $x_n\rightarrow x_0$, $t_n\rightarrow t_0$, the limit $x_0\in E(t_0)$. If $E(0)$ is bounded so that $\mathcal{N}_{\epsilon}(E(t_0))$ is compact, this implies that for any $t_0\geq0$, $\epsilon>0$ there is a $\delta>0$ such that $|t-t_0|<\delta$ implies $E(t)\subset\mathcal{N}_{\epsilon}(E(t_0))$.

(2a) $D(t)$ is a left upper semicontinuous in $t$ in the sense that for any $t_0\in(0,T)$, $x_0\in (D(t_0))^c$ there is a sequence $x_n\rightarrow x_0$ and $t_n\uparrow t_0$ with $x_n\in (D(t_n))^c$. Moreover, for any $t_0\in(0,T)$, $\epsilon>0$ there exists a $\delta>0$ such that $t_0-\delta<t<t_0$ implies $\mathcal{C}_{\epsilon}(D(t))\subset D(t_0)$.

(2b) $E(t)$ is a left lower semicontinuous in $t$ in the sense that for any $t_0\in(0,T)$, $x_0\in E(t_0)$ there is a sequence $x_n\rightarrow x_0$ and $t_n\uparrow t_0$ with $x_n\in E(t_0)$. Moreover, for any $t_0\in(0,T)$, $\epsilon>0$ there exists a $\delta>0$ such that $t_0-\delta<t<t_0$ implies $\mathcal{N}_{\epsilon}(E(t))\supset E(t_0)$.
\end{thm}

Where $N_{\epsilon}(A)=\{x\in \mathbb{R}^N\mid d(x,A)<\epsilon\}$, for $A$ is a closed subset in $\mathbb{R}^N$ and $C_{\epsilon}(A)=N_{\epsilon}(A^c)^c$, for $A$ is an open subset in $\mathbb{R}^N$.

\begin{rem}\label{rem:ori} For $A>0$, even if $D_1(0)$ and $D_2(0)$ are disjoint, $D_1(t)$ and $D_2(t)$ may intersect. The basic reason is that the level set equation (\ref{eq:level}) is not orientation free(If $u$ is a solution, there does not hold that $-u$ is also a solution for  (\ref{eq:level})). 
\end{rem}
 In order to prove Theorem \ref{thm:fattening1}, we need the following lemma. This lemma gives the construction of an open evolution containing two disjoint components. 

\begin{lem}\label{lem:sep} Assume $D_1(t)$ and $D_2(t)$ being the open evolution of (\ref{eq:level}) with $D_1(0)=U_1$ and $D_2(0)=U_2$. And $D(t)$ is denoted as the open evolution of (\ref{eq:level}) with $D(0)=U_1\cup U_2$. If $D_1(t)\cap D_2(t)=\emptyset$ for $0\leq t\leq T$, then  $D(t)=D_1(t)\cup D_2(t)$, $0\leq t\leq T$.
\end{lem}

Under the condition $A=0$, $D_1(t)\cap D_2(t)=\emptyset$ holds automatically provided that $D_1(0)\cap D_2(0)=\emptyset$. But for $A>0$, it is not true. Therefore, we give the assumption $D_1(t)\cap D_2(t)=\emptyset$ for $0\leq t\leq T$.

\begin{proof} First, we assume $D_1(t)\cap D_2(t)=\emptyset$ for $0\leq t\leq T$ and $\delta=:\min\limits_{0\leq t\leq T}\textrm{dist}(D_1(t),D_2(t))>0$. We define
$$
a_i(x)=\max\{\textrm{sd}(x,\partial D_i(0)),0\}, \ x\in\mathbb{R}^n,\ i=1,2.
$$

\begin{figure}[htbp]
	\begin{center}
            \includegraphics[height=5cm]{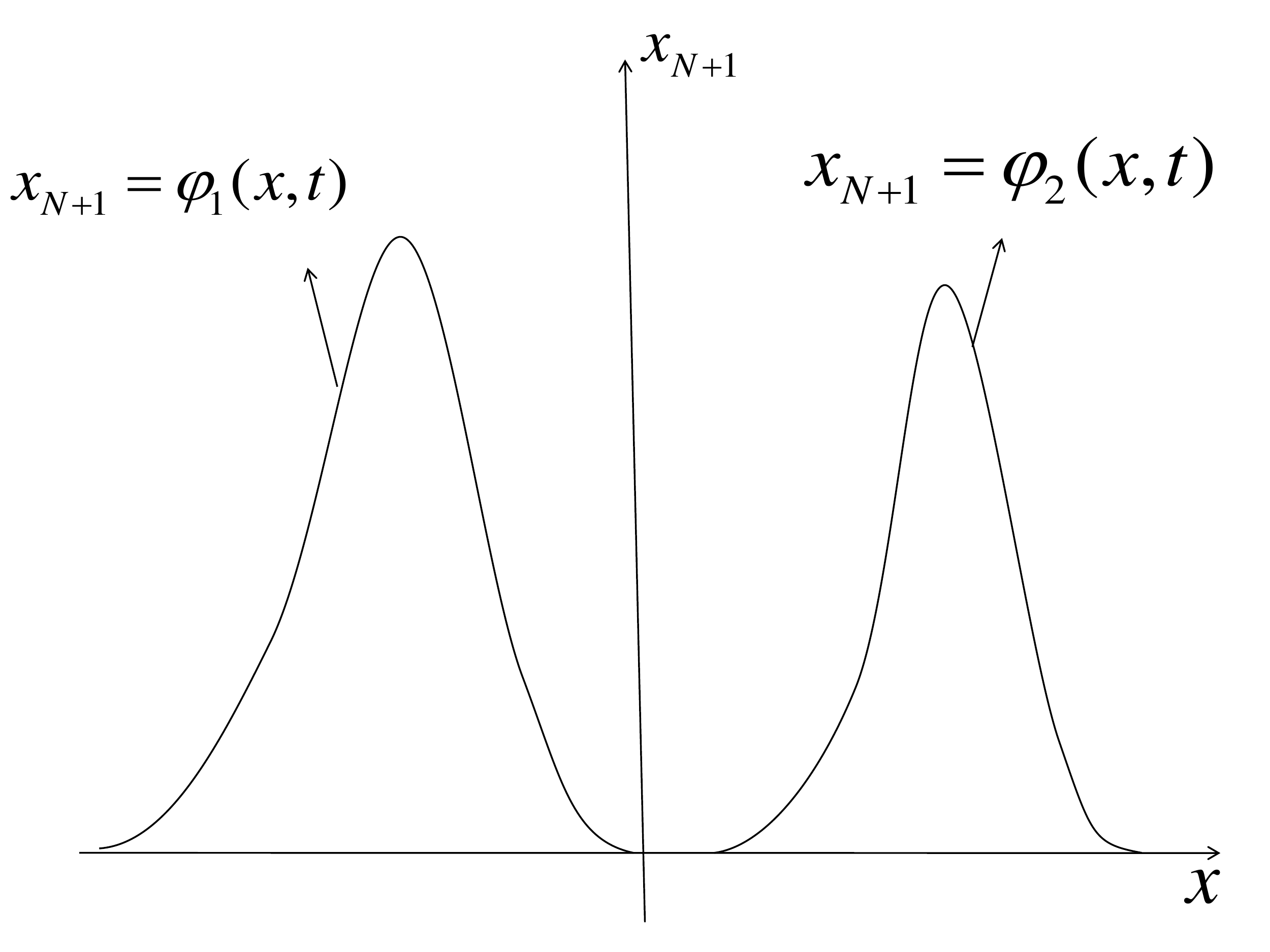}
		\vskip 0pt
		\caption{Proof of Lemma \ref{lem:sep}}
        \label{fig:lemsep}
	\end{center}
\end{figure}

In the theory in \cite{G}, there exist non-negative $\varphi_i(x,t)\in C_{c}(\mathbb{R}^N\times[0,T])$ being the solution of (\ref{eq:level}) with $\varphi_i(x,0)=a_i(x)$. Then $D_i(t)=\{x\in\mathbb{R}^N\mid \varphi_i(x,t)>0\}$ and $\varphi_i=0$ hold outside of $D_i(t)$, $i=1,2$. Since $\textrm{supp}\varphi_1$ and $\textrm{supp}\varphi_2$ are seperated by $\delta$, it is easy to show that $\varphi(x,t)=:\max\{\varphi_1,\varphi_2\}(x,t)$ is also a viscosity solution of (\ref{eq:level}). Then $D(t)=\{x\mid\varphi(x,t)>0\}=D_1(t)\cup D_2(t)$ with initial data $D(0)=U_1\cup U_2$, for $0\leq t\leq T$.

Next we prove the result only under the assumption $D_1(t)\cap D_2(t)=\emptyset$, $0\leq t\leq T$. Consider $\dis{D_i^j(t)=\{x\mid \varphi_i(x,t)>\frac{1}{j}\}}$. 

We claim that $\min\limits_{0\leq t\leq T}\textrm{dist}(D_1^j(t),D_2^j(t))>0$, for all $j$. If $\min\limits_{0\leq t\leq T}\textrm{dist}(D_1^j(t),D_2^j(t))=0$, for some $j$, then there exist $t_0\in[0,T]$ and sequences $\{x_m\}\subset D_1^j(t_0)$, $\{y_m\}\subset D_2^j(t_0)$ such that
$$
|x_m-y_m|\rightarrow0,\ \ \varphi_1(x_m,t_0)>\frac{1}{j},\ \varphi_2(y_m,t_0)>\frac{1}{j}.
$$
Then there exists $x$, such that $\lim\limits_{m\rightarrow\infty}x_m=\lim\limits_{m\rightarrow\infty}y_m=x$. Then
$$
\varphi_1(x,t_0)\geq\frac{1}{j}>0,\ \varphi_2(x,t_0)\geq\frac{1}{j}>0.
$$
Consequently, $x\in D_1(t_0)\cap D_2(t_0)\neq\emptyset$, contradiction. Then we have $\min\limits_{0\leq t\leq T}\textrm{dist}(D_1^j(t),D_2^j(t))>0$, for all $j$. By the argument in the first step, there holds $D^j(t)=D_1^j(t)\cup D_2^j(t)$ is the open evolution with initial openset $\dis{\{x\mid \varphi_1(x,0)>\frac{1}{j}\}\cup\{x\mid \varphi_2(x,0)>\frac{1}{j}\}}$, for $0\leq t\leq T$.

Noting $\bigcup\limits_{j=1}^{\infty}D_1^j(0)\cup D_2^j(0)=U_1\cup U_2$ and using Theorem \ref{thm:mon}, $D(t)=\bigcup\limits_{j=1}^{\infty}D^j(t)=\bigcup\limits_{j=1}^{\infty}D_1^j(t)\cup D_2^j(t)=D_1(t)\cup D_2(t)$, for $0\leq t\leq T$.
\end{proof}

\begin{thm}\label{thm:openevolutionmeancurvature}(Relation between evolution and mean curvature flow) Let $D(t)$, $E(t)$ be the open evolution and closed evolution, respectively. Assume that in an open region $U\times(t_1,t_2)\subset\mathbb{R}^N\times(0,T)$, $\partial D(t)$ and $\partial E(t)$ are the graph of continuous functions $v_1$, $v_2$. Precisely,
$$
\partial D(t)\cap U=\{x\in \mathbb{R}^N\mid x_N=v_1(x^{\prime},t), x^{\prime}\in U^{\prime}\}
$$
and 
$$
\partial E(t)\cap U=\{x\in \mathbb{R}^N\mid x_N=v_2(x^{\prime},t), x^{\prime}\in U^{\prime}\},
$$
where $x^{\prime}=(x_1,\cdots,x_{N-1})$, $U^{\prime}=U\cap\{x_N=0\}$ and $v_1$, $v_2$ are continuous in $U^{\prime}\times(t_1,t_2)$. Then the function $v_1$($v_2$) is a viscosity supersolution(subsolution) of 
$$
v_t=\left(\delta_{ij}-\frac{v_{x_i}v_{x_j}}{1+|\nabla v|^2}\right)v_{x_ix_j}+ A\sqrt{1+|\nabla v|^2}
$$
or is a viscosity subsolution(supersolution) of
$$
v_t=\left(\delta_{ij}-\frac{v_{x_i}v_{x_j}}{1+|\nabla v|^2}\right)v_{x_ix_j}-A\sqrt{1+|\nabla v|^2}
$$
where the signs of the last terms are determined by the direction of the normal velocity of $\partial D(t)\cap U$ and $\partial E(t)\cap U$.
\end{thm}
We can use the similar method in \cite{ES} to prove this theorem. Here we omit it.

\section{A Priori estimates }

In this section, we give an interior gradient estimate.  

{\bf Graph equation} Let $u(x,t)$ be some function on an open subset of $\mathbb{R}^n\times \mathbb{R}$, then the graph of $u(x,t)$ is a family of hypersurfaces in $\mathbb{R}^{n+1}$. If the family of hypersurfaces moves by $V=-\kappa+A$ if and only if 
$$
u_t=\left(\delta_{ij}-\frac{u_{x_i}u_{x_j}}{1+|\nabla u|^2}\right)u_{x_ix_j}\pm A\sqrt{1+|\nabla u|^2},
$$
where the signs of the last terms are determined by direction of the normal velocity $V$.

Under the case $A=0$,
$$
u_t=\dis{\left(\delta_{ij}-\frac{u_{x_i}u_{x_j}}{1+|\nabla u|^2}\right)u_{x_ix_j}},
$$
The estimate for $|\nabla u|$ in entire space $\mathbb{R}^n$ is given by \cite{EH}. The interior estimate of gradient is first given by \cite{ES}. Here we give the estimate under the condition $A>0$. Since the proof for $n=1$ is not easier than $n\geq 2$, we prove it in $n$-dimensional setting.

\begin{thm}\label{thm:es} For $u\in C^3(\Omega_{T})\cap C^0(\overline{\Omega}_T)$, $u$ satisfies
\begin{equation}\label{eq:graph}
u_t=\dis{\left(\delta_{ij}-\frac{u_{x_i}u_{x_j}}{1+|\nabla u|^2}\right)u_{x_ix_j}\pm A\sqrt{1+|\nabla u|^2}},
\end{equation}
For the condition ``$+$''(``$-$''), we assume $u<0$($u>0$) in $\Omega_T$, $u(0,T)=-v_0$($u(0,T)=v_0$). Then
$$
|\nabla u(0,T)|\leq (3+16v_0)e^{2K},
$$
where $\dis{K=20v_0^2(4n+\frac{1}{T}+4A+\frac{A}{2v_0})}+2$, $\Omega_T=B_1(0)\times (0, 2T)$ and
$$
\delta_{ij}=\left\{
\begin{array}{lcl}
1,\ i=j\\
0,\ i\neq j
\end{array}
\right..
$$
\end{thm}

\begin{proof} We only prove the condition "$+$". For the condition, we can consider ``$-u$'' to get the result. Denote $w=\sqrt{1+|\nabla u|^2}$, $\nu^i=u_{x_i}/\sqrt{1+|\nabla u|^2}$, $g^{ij}=\delta_{ij}-\nu^i\nu^j$. We define the operator $L$ as
$$
Lh=g^{ij}h_{x_ix_j}-h_t+A\nu^kh_{x_k}.
$$

We let $h=\eta(x,t,u(x,t))w$, where $\eta$ is a non-negative function and will be identified in future. By calculation,
\begin{eqnarray*}
Lh&=&g^{ij}(w_{x_ix_j}\eta+w_{x_i}(\eta)_{x_j}+(\eta)_{x_i}w_{x_j}+w(\eta)_{x_ix_j})\\
&-&(\eta)_tw-\eta w_t+A\nu^k(w_{x_k}\eta+(\eta)_{x_k}w)\\
&=&\eta Lw+wL\eta+2g^{ij}w_{x_i}(\eta)_{x_j}\\
&=&\eta Lw+wL\eta+2g^{ij}w_{x_i}\left(\frac{h_{x_j}-w_{x_j}\eta}{w}\right).
\end{eqnarray*}
Then  
$$
Lh-2g^{ij}\frac{w_{x_i}}{w}h_{x_j}=\eta\left(Lw-2g^{ij}\frac{w_{x_i}w_{x_j}}{w}\right)+wL\eta.
$$

We claim that 
$$
Lw-2g^{ij}\frac{w_{x_i}w_{x_j}}{w}\geq 0.
$$
Therefore, there holds 
\begin{equation}\label{eq:grainter}
Lh-2g^{ij}\frac{w_{x_i}}{w}h_{x_j}\geq wL\eta.
\end{equation}

We begin to prove the claim. Seeing
$$
w_{x_ix_j}=\nu^ku_{x_kx_ix_j}+\frac{1}{w}(u_{x_kx_i}u_{x_kx_j}-\nu^k\nu^lu_{x_kx_i}u_{x_lx_j}),
$$
we have
$$
g^{ij}w_{x_ix_j}\geq\nu^kg^{ij}u_{x_kx_ix_j}=\nu^k((g^{ij}u_{x_ix_j})_{x_k}-g^{ij}_{x_k}u_{x_ix_j})
$$
$$
=\nu^k\left(u_{tx_k}-A\frac{u_{x_l}u_{x_lx_k}}{\sqrt{1+|\nabla u|^2}}\right)-\nu^kg_{x_k}^{ij}u_{x_ix_j}.
$$
Combining 
$$
g_{x_k}^{ij}=-\frac{1}{w}(\nu^ju_{x_ix_k}+\nu^iu_{x_jx_k})+\frac{2u_{x_i}u_{x_j}}{w^3}w_{x_k},
$$
\begin{eqnarray*}
\nu^kg^{ij}u_{x_kx_ix_j}
&=&w_t-A\nu^kw_{x_k}+\frac{\nu^ku_{x_ix_j}}{w}\left(\nu^ju_{x_ix_k}+\nu^iu_{x_jx_k}-\frac{2u_{x_i}u_{x_j}}{w^2}w_{x_k}\right)\\
&=&w_t-A\nu^kw_{x_k}+\frac{2}{w}g^{ij}w_{x_i}w_{x_j}.
\end{eqnarray*}
Therefore
$$
g^{ij}w_{x_ix_j}\geq w_t-A\nu^kw_{x_k}+\frac{2}{w}g^{ij}w_{x_i}w_{x_j}.
$$
Then 
$$
Lw\geq\frac{2}{w}g^{ij}w_{x_i}w_{x_j}.
$$
We complete the proof of the claim.

Next we choose $\eta=f\circ\phi(x,t,u(x,t))$,
$$\dis{\phi(x,t,z)=\left(\frac{z}{2v_0}+\frac{t}{T}(1-|x|^2)\right)^+}$$
and
$$f(\phi)=e^{K\phi}-1.$$

When $\phi>0$, there holds
$$
\phi_z=\frac{1}{2v_0},\ \phi_t=\frac{1-|x|^2}{T},\ \phi_{x_i}=-\frac{2t}{T}x_i,\ \phi_{x_ix_j}=-\frac{2t}{T}\delta_{ij}.
$$
Consequently, when $\phi>0$, $z<0$, $0<t<2T$,
$$
0\leq\phi\leq2,\ \sum\phi_{x_i}^2\leq\frac{4t^2}{T^2}\leq16.
$$

By calculation,
\begin{eqnarray*}
L\eta&=&g^{ij}f^{\prime\prime}(\phi_{x_i}+\phi_{z}u_{x_i})(\phi_{x_j}+
\phi_{z}u_{x_j})+g^{ij}f^{\prime}(\phi_{x_ix_j}+\phi_zu_{x_ix_j})\\
&-&f^{\prime}(\phi_t+\phi_zu_t)+A\nu^kf^{\prime}(\phi_{x_k}+\phi_zu_{x_k})\\
&\geq&\frac{f^{\prime\prime}}{w^2}(\phi_{x_i}+\phi_zu_{x_i})^2+f^{\prime}(g^{ij}
\phi_{x_ix_j}-\phi_t+A\nu^k\phi_{x_k})+f^{\prime}\phi_zLu
\end{eqnarray*}
\begin{eqnarray*}
&=&\frac{f^{\prime\prime}}{w^2}\left(-\frac{2t}{T}x_i+\frac{1}{2v_0}u_{x_i}\right)^2+f^{\prime}\left(-\frac{2t}{T}(n-\frac{|\nabla u|^2}{1+|\nabla u|^2})-\frac{1-|x|^2}{T}\right.\\
&-&\left. Ax_k\frac{2t}{T}\frac{u_{x_k}}{\sqrt{1+|\nabla u|^2}}\right)+f^{\prime}\phi_zLu.
\end{eqnarray*}
Combining
$$
Lu=g^{ij}u_{x_ix_j}-u_t+A\nu^ku_{x_k}=-\frac{A}{\sqrt{1+|\nabla u|^2}},
$$
there holds
\begin{eqnarray*}
L\eta&\geq& \frac{f^{\prime\prime}}{w^2}\left(\frac{|\nabla u|^2}{8v_0^2}-8\right)+f^{\prime}\left(-4n-\frac{1}{T}-4A\right)-f^{\prime}\phi_z\frac{A}{\sqrt{1+|\nabla u|^2}}\\
&\geq& \frac{f^{\prime\prime}}{w^2}\left(\frac{|\nabla u|^2}{8v_0^2}-8\right)+
f^{\prime}\left(-4n-\frac{1}{T}-4A-\frac{A}{2v_0}\right)\\
&=&\frac{K^2e^{K\phi}}{w^2}\left(\frac{|\nabla u|^2}{8v_0^2}-8\right)+
Ke^{K\phi}\left(-4n-\frac{1}{T}-4A-\frac{A}{2v_0}\right).
\end{eqnarray*}

When $|\nabla u|\geq \max\{16v_0,2\}$, we have 
$$
\frac{|\nabla u|^2}{16v^2_0}\geq8,\ \frac{|\nabla u|^2}{16}\geq\frac{1+|\nabla u|^2}{20}.
$$

Then
\begin{eqnarray*}
L\eta&\geq&\frac{K^2e^{K\phi}}{w^2}\frac{|\nabla u|^2}{16v_0^2}+
Ke^{K\phi}\left(-4n-\frac{1}{T}-4A-\frac{A}{2v_0}\right)\\
&\geq&\frac{K^2e^{K\phi}}{20v_0^2}+Ke^{K\phi}(-4n-\frac{1}{T}-4A-\frac{A}{2v_0})\\
&=&Ke^{K\phi}\left(\frac{K}{20v_0^2}-4n-\frac{1}{T}-4A-\frac{A}{2v_0}\right)>0,
\end{eqnarray*}
when we choose $\dis{K=20v_0^2(4n+\frac{1}{T}+4A+\frac{A}{2v_0})+2}$, $\Omega_T=B_1(0)\times (0, 2T)$.

Therefore by (\ref{eq:grainter}), there holds 
$$
Lh-2g^{ij}\frac{w_{x_i}}{w}h_{x_j}\geq0\ \text{on}\ \{h>0\ \textrm{or}\ |\nabla u|>\max\{16v_0,2\}\}.
$$
 By maximum principle,
\begin{eqnarray*}
(e^{\frac{K}{2}}-1)w(0,T)&=&h(0,T)\leq\max\limits_{h=0\ \textrm{and} \ |\nabla u|=\max\{16v_0,2\}}h\\
&\leq&(e^{2K}-1)\max\{\sqrt{1+(16v_0)^2},\sqrt{5}\}.
\end{eqnarray*}
Consequently, $w(0,T)\leq e^{2K}(3+16v_0)$.
\end{proof}

\begin{rem}\label{rem:es} (1) In Theorem \ref{thm:es}, $\Omega_T$ can be replaced by $\Omega_T=B_R(x_0)\times(0,2T)$ and $v_0=u(x_0,T)$. Then the conclusion becomes
$$
|\nabla u(x_0,T)|\leq e^{2K}(3+16\frac{v_0}{R}),
$$
where $\dis{K=20\frac{v_0^2}{R^2}\left(4n+\frac{R^2}{T}+\frac{4A}{R}+\frac{A}{2v_0}\right)}+2$. We can set $v(x,t)=\dis{\frac{u(Rx+x_0,R^2t)}{R}}$, then we can use Theorem \ref{thm:es} for $v(x,t)$.

(2) When $u$ is the solution of (\ref{eq:graph}) for ``+'' without $u<0$, then we can set
$$v=u-M-\epsilon$$
where $M=\sup\limits_{\overline{\Omega}_T} |u|$ and $\epsilon>0$. Using (1) in Remark \ref{rem:es} to $v$, we can deduce
$$|\nabla u(0,T)|\leq\left(3+16\frac{M-u(0,T)+\epsilon}{R}\right)e^{2\widetilde{K}_{\epsilon}},$$
where $\dis{\widetilde{K}_{\epsilon}=\frac{20(M-u(0,T)+\epsilon)^2}{R^2}\left(4n+\frac{R^2}{T}+\frac{4A}{R}
+\frac{A}{2(M+\epsilon-u(0,T))}\right)}$+2.

Tending $\epsilon\rightarrow0$, we have
$$
|\nabla u(0,T)|\leq \left(3+32\frac{M}{R}\right)e^{2\widetilde{K}},
$$
where $\dis{\widetilde{K}=\frac{80M^2}{R^2}\left(4n+\frac{R^2}{T}+\frac{4A}{R}\right)
+\frac{20AM}{R^2}}+2$.
\end{rem}

Then we use the (2) in Remark \ref{rem:es} and the same method as \cite{AAG} to prove the next corollary.

\begin{cor}\label{cor:es} For $s_1<s_2$, $\rho>0$ and $q\in\mathbb{R}^n$ we set
$$
\Omega=B_{\rho}(x_0)\times(s_1,s_2).
$$
Suppose that $u\in C^3(\Omega)$ solves the equation (\ref{eq:graph}) in $\Omega$ with $M=\sup\limits_{\overline{\Omega}}|u|<\infty$. For any $\epsilon>0$ there is a constant $C=C(M,\epsilon,n)$ such that
$$|\nabla u|\leq C \ \textrm{on}\ \Omega_{\epsilon}=B_{\rho-\epsilon}(x_0)\times(s_1+\epsilon^2,s_2).$$
\end{cor}
\begin{rem}\label{rem:hes}
(1) From Corollary \ref{cor:es} and \cite{LSU}, there exist $C_k(M,\epsilon,n)$ such that 
$$
|\nabla^k u|\leq C_k,\ (x,t)\in B_(\rho-2\epsilon)\times(s_1+2\epsilon^2,s_2).
$$
(2) Noting $C$ and $C_k$ are all independent on $s_2$, if the solution $u$ exists for all $t>s_1$, $s_2$ can be chosen as $\infty$.
\end{rem}

\section{Intersections and the Sturmian theorem}

 In this section, we want to introduce the intersection number argument. There are many applications for the intersection number principle. For example, in this paper:

(1). A type of derivative estimate in Theorem \ref{thm:grad}.

(2). The flow $\Gamma(t)$ evolving by $V=-\kappa+A$ with some initial curve does not intersect itself.(Proposition \ref{lem:alphad2}) 

(3). The asymptotic behavior in Section 7. 

Since the proof in $\mathbb{R}^2$ is not difficult than in higher dimension,  some theorems and lemmas are proved in $\mathbb{R}^{n+1}$, $n\geq1$. 

\textbf{Sturm's classical result} The Sturmian theorem states that the number of zeros(counted with multiplicity) of a solution of linear parabolic equation of the type
$$
u_t=a(x,t)u_{xx}+b(x,t)u_x+c(x,t)u
$$
doesn't increase with time, provided that $u$ is defined on a rectangle $x_0\leq x\leq x_1$, $0<t<T$ and $u(x_j,t)\neq0$ for $j=0,1$, for all $t\in(0,T)$. This result also holds for the number of sign changing rather than the number of zeros of $u(\cdot,t).$ 

It is well known that the intersection number between two families of rotationally symmetric hypersurfaces $\Gamma_1(t)$ and $\Gamma_2(t)$ evolving by $V=-\kappa$ is non-increasing(\cite{A2}). The definition of intersection number between two families of rotationally symmetric hypersurfaces is given following. But this result is not true under the condition $V=-\kappa+A$. Indeed seeing future, the intersection number between two families of rotationally symmetric hypersurfaces evolving by $V=-\kappa+A$ may increase. In this section we give some results about this.

{\bf Horizontal and vertical graph equation } If $\Gamma(t)$ is a family of rotationally symmetric hypersurfaces in $\mathbb{R}^{n+1}$, then parts of $\Gamma(t)$ may be represented either as horizontal graph, $r=u(x,t)$, or vertical graph, $x=v(r,t)$, where $(x,y_1,\cdots,y_n)\in\mathbb{R}^{n+1}$ and $r=\sqrt{y_1^2+y_2^2+\cdots+y_n^2}$.

If $\Gamma(t)$ is given as a horizontal graph, then $\Gamma(t)$ evolves by $V=-\kappa+A$ in $\mathbb{R}^{n+1}$ and the direction of the normal velocity $V$ is chosen outward iff $u$ satisfies the horizontal graph equation 
\begin{equation}\label{eq:1horizontal}
\frac{\partial u}{\partial t}=\frac{u_{xx}}{1+u_x^2}-\frac{n-1}{u}+A\sqrt{1+u_x^2}.
\end{equation}
If $\Gamma(t)$ is given as a vertical graph, then $\Gamma(t)$ evolves by $V=-\kappa+A$ in $\mathbb{R}^{n+1}$ iff $v$ satisfies the vertical graph equation
\begin{equation}\label{eq:1vertical+}
\frac{\partial v}{\partial t}=\frac{v_{rr}}{1+v_r^2}+\frac{n-1}{r}v_r+ A\sqrt{1+v_r^2},
\end{equation}
or
\begin{equation}\label{eq:1vertical-}
\frac{\partial v}{\partial t}=\frac{v_{rr}}{1+v_r^2}+\frac{n-1}{r}v_r-A\sqrt{1+v_r^2},
\end{equation}
where the signs of the last terms are determined by the direction of the normal velocity $V$(We choose ``$+$($-$)'' when the direction of $V$ is rightward(leftward)).

{\bf Intersection number for rotationally symmetric hypersurfaces} For two rotationally symmetric hypersurfaces $\Gamma_1(t)$ and $\Gamma_2(t)$ are given by $\Gamma_1(t)=\{(x,y)\in \mathbb{R}\times\mathbb{R}^{n}\mid r=u_1(x,t)\}$ and $\Gamma_2(t)=\{(x,y)\in \mathbb{R}\times\mathbb{R}^{n}\mid r=u_2(x,t)\}$. The intersection number between $\Gamma_1(t)$ and $\Gamma_2(t)$ denoted by $\mathcal{Z}[\Gamma_1(t),\Gamma_2(t)]$ is defined by the number of intersections between $u_1(\cdot,t)$ and $u_2(\cdot,t)$.

\begin{thm}\label{thm:sl}Two smooth families of smooth, closed, hypersurfaces given by $\Gamma_1(t)=\{(x,y)\in \mathbb{R}\times\mathbb{R}^n\mid r=u_1(x,t),a_1(t)\leq x\leq b_1(t)\}$, $\Gamma_2(t)=\{(x,y)\in \mathbb{R}\times\mathbb{R}^n\mid r=u_2(x,t),a_2(t)\leq x\leq b_2(t)\}$ evolve by $V=-\kappa+A$ in $\mathbb{R}^{n+1}$, $0<t<T$. Then either $\Gamma_1\equiv\Gamma_2$ for all $t\in(0,T)$, or the number of intersections of $\Gamma_1(t)$ and $\Gamma_2(t)$ is finite for all $t\in(0,T)$. In the second case, if $a_1(t)$, $b_1(t)$, $a_2(t)$ and $b_2(t)$ are all different and their order remains unchanged for all $t\in(0,T)$, this number is nonincreasing in time, and decreases whenever $\Gamma_1(t)$ and $\Gamma_2(t)$ have a tangential intersection.
\end{thm}
We only give the sketch of the proof. For example, if the order of $a_1$, $b_1$, $a_2$, $b_2$ is given by $a_1(t)<a_2(t)<b_1(t)<b_2(t)$, $0<t<T$, the intersections are only in the interval $[a_2(t),b_1(t)]$. Since $u_1(a_2(t),t)-u_2(a_2(t),t)\neq0$ and $u_1(b_1(t),t)-u_2(b_1(t),t)\neq0$, $0<t<T$, using Theorem D in \cite{A1}, the intersection number between $u_1$ and $u_2$ is not increasing and decreases when tangentially intersecting in $[a_2(t),b_1(t)]$. Consequently, the intersection number between $\Gamma_1(t)$ and $\Gamma_2(t)$ is not increasing. We can prove the result for the other conditions with the same method.

\begin{thm}\label{thm:grad} $\Gamma(t)=\{(x,y)\in\mathbb{R}^{n+1}\mid r=u(x,t),a_2(t)\leq x\leq b_2(t)\}$ is a smooth family of closed, smooth hypersurfaces in $\mathbb{R}^{n+1}$, $0<t<T$. If $\Gamma(t)$ evolves by $V=-\kappa+A$ in $\mathbb{R}^{n+1}$,  there is a function $\sigma$: $\mathbb{R}_+\times\mathbb{R}_+\rightarrow\mathbb{R}$ such that
$$
|u_x(x,t)|\leq \sigma(t,u(x,t))
$$
holds for $0<t<T$, $a_2(t)<x<b_2(t)$. The function $\sigma$ only depends on $M=\max\limits_{a_2(0)<x<b_2(0)} u(x,0)$ and $T$.
\end{thm}
\begin{proof} Let
$w_{0}(r)\in C^{\infty}((0,+\infty))$, $w^{\prime}_0(r)\geq0$ and
$$
x=w_0(r)=\left\{
\begin{array}{lcl}
0, \ 0\leq r<M+1\\
1, \ r>M+2
\end{array}\right..
$$

\begin{figure}[htbp]
	\begin{center}
            \includegraphics[height=7.0cm]{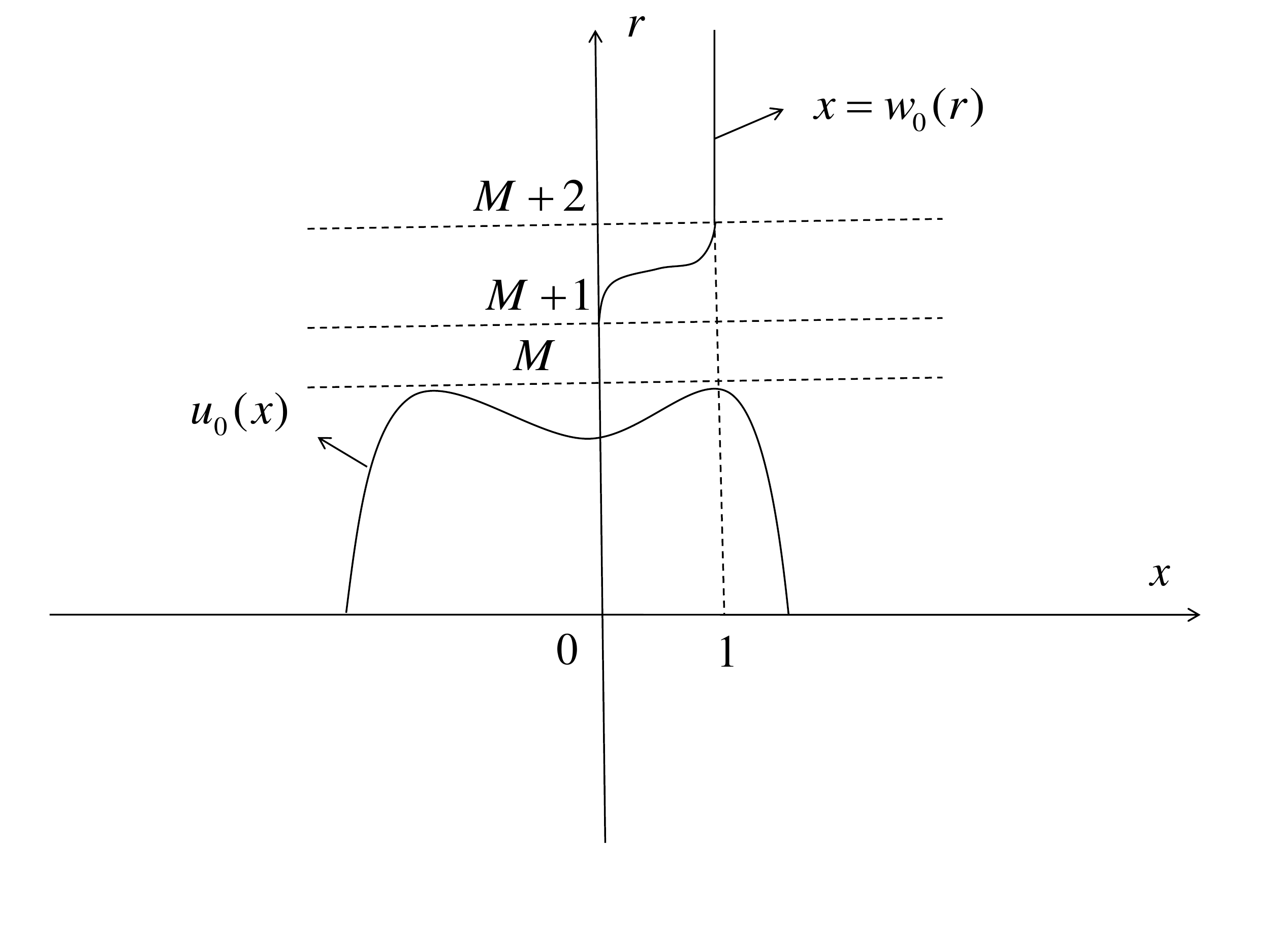}
		\vskip 0pt
		\caption{Proof of Theorem \ref{thm:grad}}
        \label{fig:grad}
	\end{center}
\end{figure}

We let $w$ be the unique solution of the vertical equation (\ref{eq:1vertical-}) with the boundary condition
$$
w_r(0,t)=0,\  t\geq0
$$
and initial condition
$$
w(r,0)=w_0(r),\ r\geq0.
$$
Differentiating (\ref{eq:1vertical-}) in $r$, 
\begin{equation}\label{eq:deriveeq}
p_t=a(r,t)p_{rr}+b(r,t)p_r+c(r,t)p,
\end{equation}
where $p=w_r$, $a(r,t)=1/(1+w_r^2)$, $b(r,t)=-2w_rw_{rr}/(1+w_r^2)^2+(n-1)/r-Aw_r/\sqrt{1+w_r^2}$, $c(r,t)=-(n-1)/r^2$.

By the maximum principle, we have for all $r,t>0$, $w_r\geq 0$ and $\sup\limits_{r\geq0}w(r,t)$ is nonincreasing in time. It follows from classical estimate for parabolic equation that all derivative of $w$ are uniformly bounded for $r,t\geq0$.

We note that $w_r(r,0)>0$ for $M+1<r<M+2$. Using the property of Green's function, for any $\delta$ satisfying $0<\delta<M+AT$, there exists $A_{\delta,T}>0$ such that
$A_{\delta,T}$ decreases with respect to $\delta$ and
\begin{equation}\label{eq:derivativeblew}
p(r,t)\geq e^{-\frac{A_{\delta,T}}{t}},
\end{equation}
for $\delta\leq r\leq M+AT$, $0<t<T$.

Since the strong maximum principle implies that $p(r,t)>0$ for $r>0$, the inverse of $x=w(r,t)$ exists, denoted by $r=v(x,t)$. Seeing the normal velocity of $x=w(r,t)$ is leftward, the normal velocity of $r=v(x,t)$ is upward. Then $v(x,t)$ satisfies the horizontal graph equation (\ref{eq:1horizontal}) with the free boundary condition
$$
v(a(t),t)=0,\ v_x(a(t),t)=\infty,\ \lim\limits_{x\rightarrow b(t)}v(x,t)=\infty,\ \lim\limits_{x\rightarrow b(t)}v_x(x,t)=\infty,\ t>0.
$$
Let $\Sigma(t)=\{(x,y)\in \mathbb{R}^{n+1}\mid r=v(x,t),a(t)\leq x< b(t)\}$ and $\Sigma_{\xi}(t)$ denote the translation of $\Sigma(t)$ given by
$$
x=w(r,t)+\xi.
$$
$\Sigma_{\xi}(t)$ can be also represented by $r=v(x-\xi,t)$, $a(t)+\xi\leq x<b(t)+\xi$. Let $a_1(t)$ and $b_1(t)$ be the end point of $\Sigma_{\xi}(t)$, then $a_1(t)=\xi+a(t)$, $b_1(t)=\xi+b(t)$. Obviously, for $(x_0,t_0)\in (a_2(t_0),b_2(t_0))\times(0,T)$, there exists $\xi\in \mathbb{R}$ such that
$$
v(x_0-\xi,t_0)=u(x_0,t_0).
$$
By the following Lemma \ref{lem:inters}, we can deduce that the graph of $u(x,t_0)$ intersects $v(x-\xi,t_0)$ only once.
\begin{figure}[htbp]
	\begin{center}
            \includegraphics[height=7.0cm]{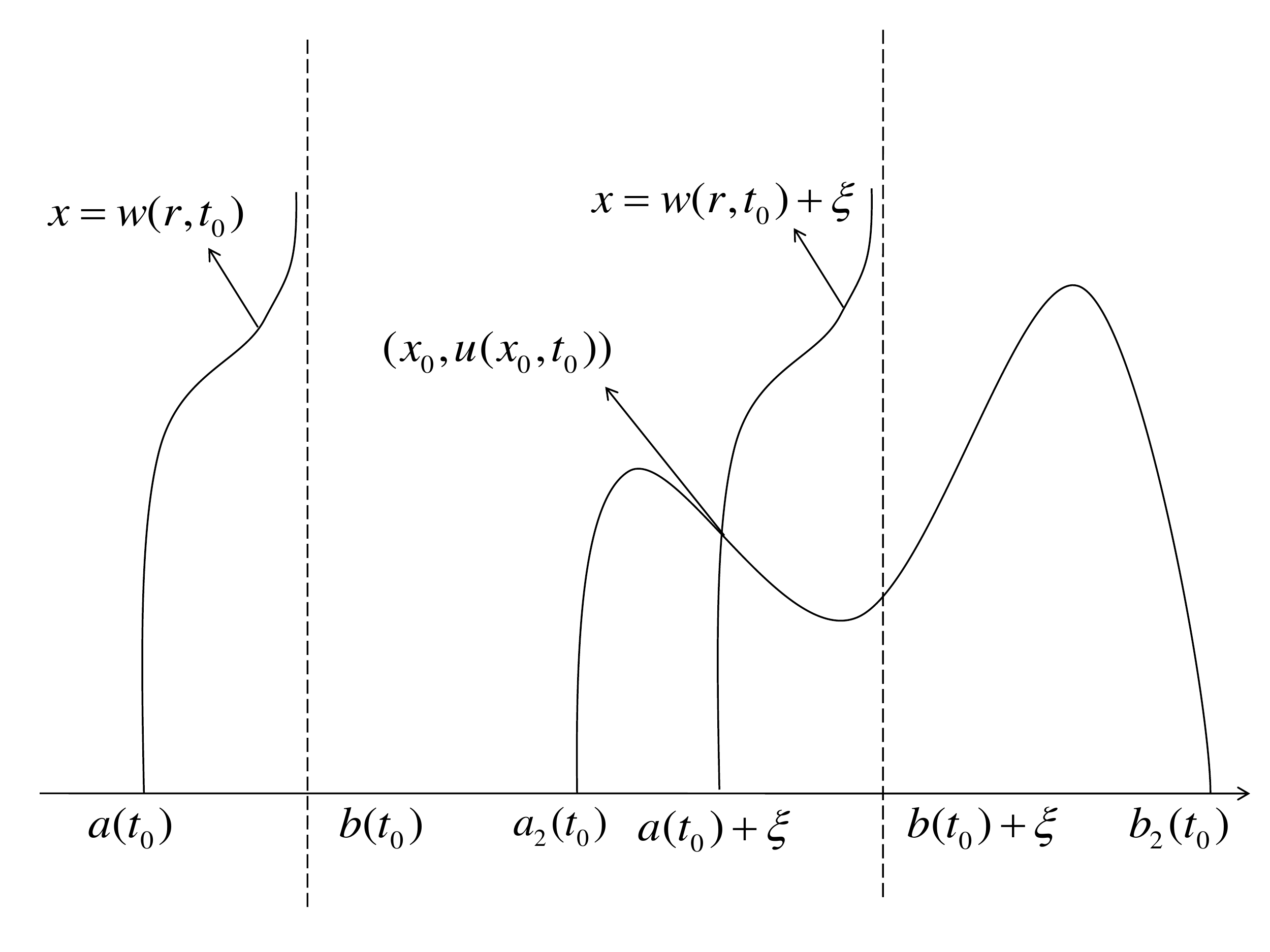}
		\vskip 0pt
		\caption{Proof of Theorem \ref{thm:grad}}
        \label{fig:grad2}
	\end{center}
\end{figure}
Next we claim 
$$v_x(x_0-\xi,t_0)\geq u_x(x_0,t_0).$$ 
If not, $v_x(x_0-\xi,t_0)< u_x(x_0,t_0)$, then there exists $\delta>0$, such that
$$
u(x,t_0)>v(x-\xi,t_0),
$$
for all $x\in(x_0,x_0+\delta)$. Since $\lim\limits_{x\rightarrow b_1(t)}v(x,t_0)=+\infty$, $\Sigma_{\xi}(t_0)$ intersects $\Gamma(t_0)$ at least twice. This yields a contradiction. 

By maximum principle, it is easy to see $r=u(x,t)<M+At<M+AT$, $a_2(t)\leq x\leq b_2(t)$, $0<t<T$. Combining (\ref{eq:derivativeblew}), there holds
$$u_x(x_0,t_0)\leq\frac{1}{w_r(v(x_0-\xi,t_0),t_0)}\leq e^{\frac{A_{v(x_0-\xi,t_0),T}}{t_0}}=e^{\frac{A_{u(x_0,t_0),T}}{t_0}}:=\sigma(t_0,u(x_0,t_0)).$$

By considering the reflection $\widetilde{\Sigma}(0)=\{(x,y)\mid x=-w_0(r)\}$ and the equation (\ref{eq:1vertical+}) with $w_r(0,t)=0,\  t\geq0$ and $w(r,0)=w_0(r),\ r\geq0$, the bound for $-u_x(x_0,t_0)$ can be got similarly.
\end{proof}

\begin{lem}\label{lem:inters} $\Sigma_{\xi}(t)$ and $\Gamma(t)$ is given by Theorem \ref{thm:grad}, then $\Sigma_{\xi}(t)$ intersects $\Gamma(t)$ at most once.
\end{lem}
\begin{proof} By the same argument as Theorem \ref{thm:sl}, the intersection number between $\Sigma_{\xi}(t)$ and $\Gamma(t)$ is not increasing provided that $a_1(t)$, $b_1(t)$, $a_2(t)$ and $b_2(t)$ are all different and that their order remains unchanged. So we only prove this result when the order of $a_1(t)$, $b_1(t)$, $a_2(t)$ and $b_2(t)$ changes.

{\bf Case 1.} Assume $a_1(t)<a_2(t)<b_1(t)< b_2(t)$, $t<t_2$ and $a_1(t)<a_2(t)<b_2(t)< b_1(t)$, $t>t_2$. And for $t<t_2$, $\Sigma_{\xi}(t)$ does not intersect $\Gamma(t)$. Then $\Sigma_{\xi}(t)$ does not intersect $\Gamma(t)$, for $t>t_2$.
\begin{figure}[htbp]
 \begin{minipage}{0.49\hsize}
  \begin{center}
   \includegraphics[width=5cm]{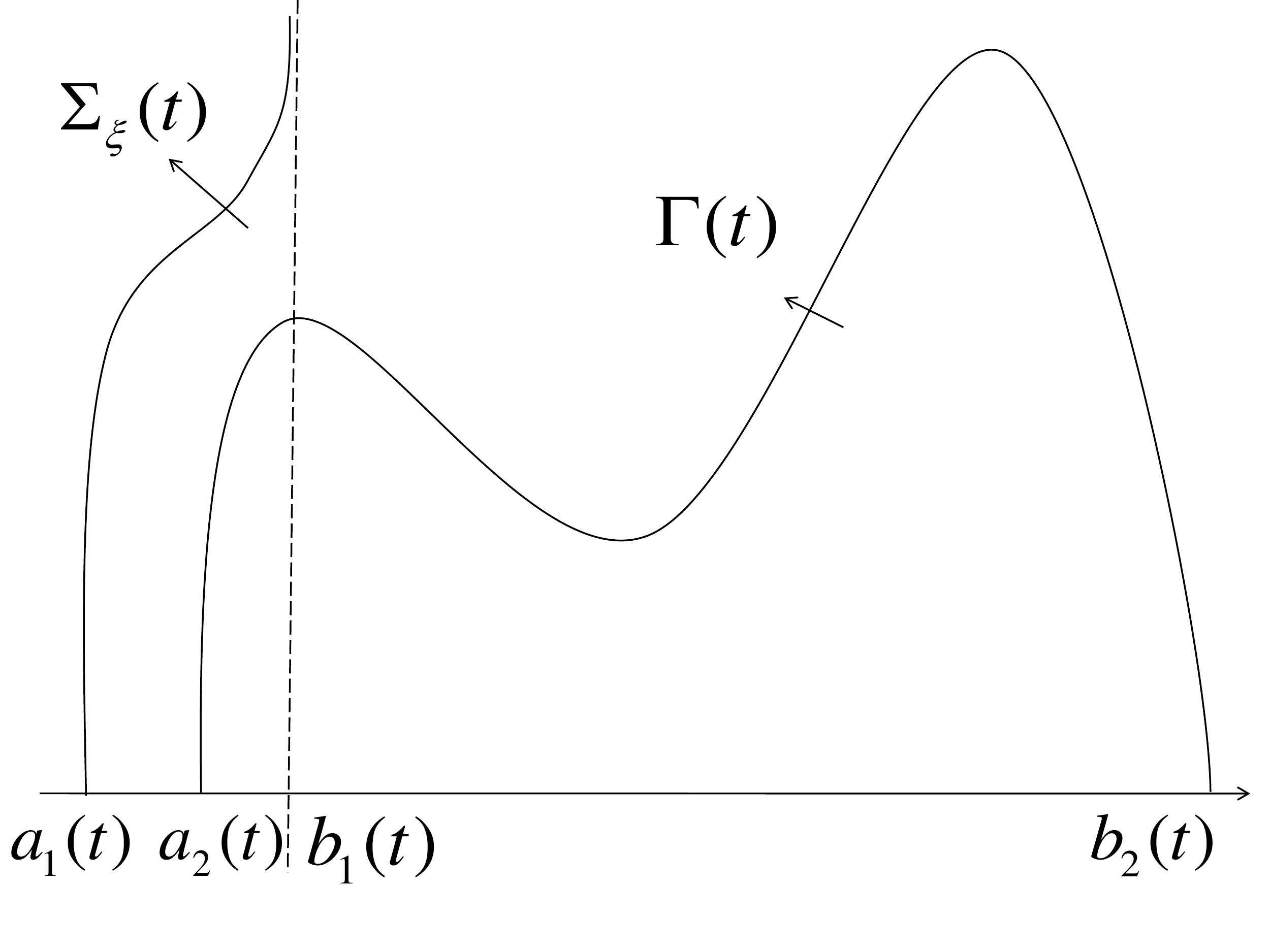}
  \end{center}
  \caption{Case 2 }
  \label{fig:21}
 \end{minipage}
 \begin{minipage}{0.49\hsize}
  \begin{center}
   \includegraphics[width=5cm]{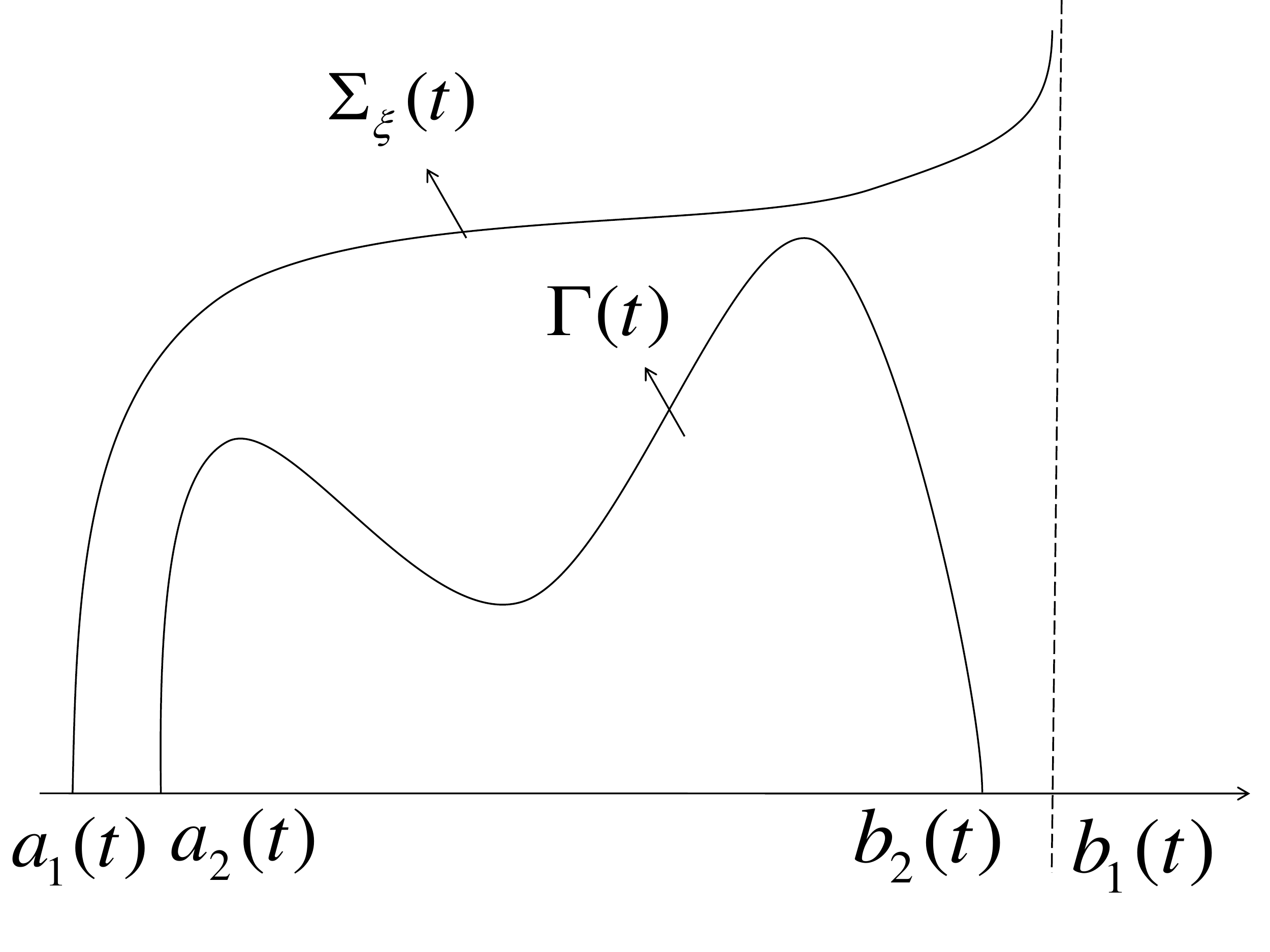}
  \end{center}
  \caption{Case 2 }
  \label{fig:22}
 \end{minipage}
\end{figure}

Since$\lim\limits_{x\rightarrow b_1(t_2)}v(x-\xi,t_2)=+\infty$ and $u(b_2(t_2),t_2)=0$, there exists a positive $\delta$ independent on $t$, such that $v(x-\xi,t_2)>u(x,t_2)$, $b_1(t_2)-\delta<x<b_1(t_2)$. By continuity, there exists $\epsilon$ such that 
\begin{equation}\label{eq:1interlemma}
v(b_1(t_2)-\delta-\xi,t)>u(b_1(t_2)-\delta,t),\ t_2-\epsilon\leq t<t_2+\epsilon
\end{equation}
and
\begin{equation}\label{eq:2interlemma}
v(x-\xi,t)>u(x,t),\ b_1(t_2)-\delta<x\leq b_2(t),\ t_2\leq t<t_2+\epsilon.
\end{equation}
The assumptions in this case imply boundary condition
$$
u(a_2(t),t)-v(a_2(t)-\xi,t)< 0,\ t_2-\epsilon\leq t<t_2+\epsilon
$$ 
and initial condition 
$$
u(x,t_2-\epsilon)<v(x-\xi,t_2-\epsilon),\ a_2(t_2-\epsilon)\leq x\leq b_1(t_2)-\delta.
$$ 
Combining the other boundary condition (\ref{eq:1interlemma}), using maximum principle in domain 
$$
\cup_ {t_2-\epsilon\leq t<t_2+\epsilon}\left(\left[a_2(t),b_1(t_2)-\delta\right]\times\{t\}\right),
$$ 
there holds
$$
u(x,t)<v(x-\xi,t),\ a_2(t)\leq x\leq b_1(t_2)-\delta,\ t_2-\epsilon\leq t<t_2+\epsilon.
$$
Seeing (\ref{eq:2interlemma}), $u(x,t)<v(x-\xi,t)$, $a_2(t)\leq x\leq b_2(t)$, $t_2\leq t<t_2+\epsilon$. It means that $\Sigma_{\xi}(t)$ does not intersect $\Gamma(t)$, for $t_2\leq t<t_2+\epsilon$. So by Theorem \ref{thm:sl}, $\Sigma_{\xi}(t)$ does not intersect $\Gamma(t)$, for $t>t_2$.

{\bf Case 2.} Assume $a_1(t)<a_2(t)$, $\Sigma_{\xi}(t)$ does not intersect $\Gamma(t)$, $t<t_3$ and $a_1(t_3)=a_2(t_3)$.

\begin{figure}[htbp]
  \begin{center}
   \includegraphics[width=9cm]{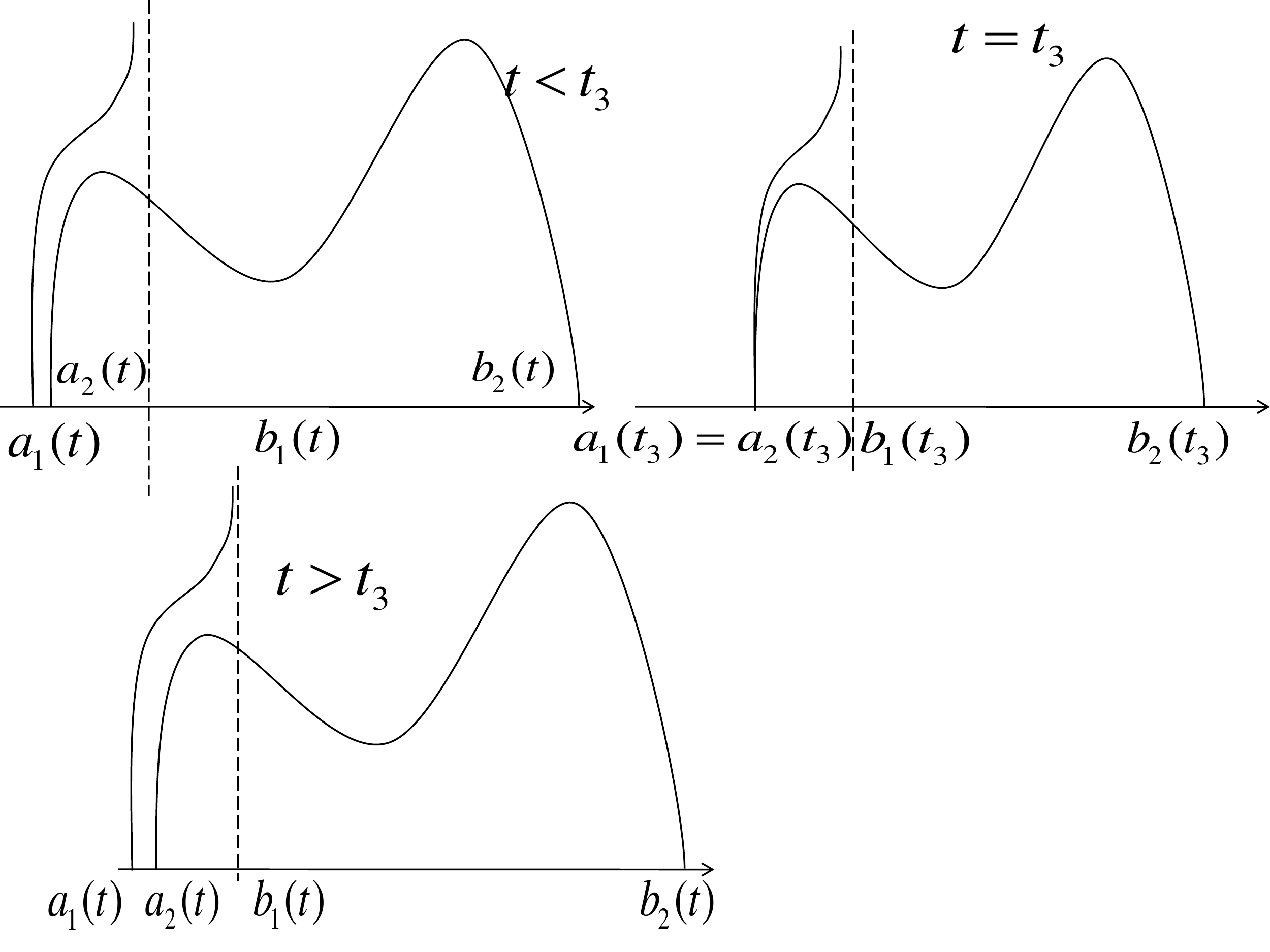}
  \end{center}
  \caption{Case 3}
  \label{fig:9}
 \end{figure}

Since $\lim\limits_{x\rightarrow a_2(t)}u_x(x,t)=\infty$, there exist $\delta_1$ and $\epsilon$ such that $r=u(x,t)$ can be expressed as $x=h(r,t)$, $0\leq r\leq \delta_1$, $t_3-\epsilon<t<t_3+\epsilon$. The assumptions in this case imply that
 $$
w(\delta_1,t)+\xi<h(\delta_1,t),\ t_3-\epsilon<t<t_3+\epsilon.
$$
 It is easy to see $w(r,t)+\xi$ and $h(r,t)$ satisfy the vertical graph equation
$$
\left\{
\begin{array}{lcl}
\dis{w_t=\frac{w_{rr}}{1+w_r^2}+\frac{n-1}{r}w_r-A\sqrt{1+w_r^2}, \ 0\leq r\leq\delta_1,\ t_3-\epsilon<t<t_3+\epsilon},\\
w_r(0,t)=0,\  t\geq0,\\
\end{array}
\right.
$$
and $w(r,t_3-\epsilon)+\xi< h(r,t_3-\epsilon)$. By strong maximum principle, $w(r,t)+\xi<h(r,t)$, for $0\leq r<\delta_1,\ t_3-\epsilon<t<t_3+\epsilon$. Contradiction to $a_1(t_3)=a_2(t_3)$. It means that this case does not happen.

{\bf Case 3.} Assume $a_2(t)<b_2(t)<a_1(t)<b_1(t)$, $t<t_6$ and $a_2(t)<a_1(t)<b_2(t)<b_1(t)$, $t>t_6$.

\begin{figure}[h]
 \begin{minipage}{0.49\hsize}
  \begin{center}
   \includegraphics[width=5cm]{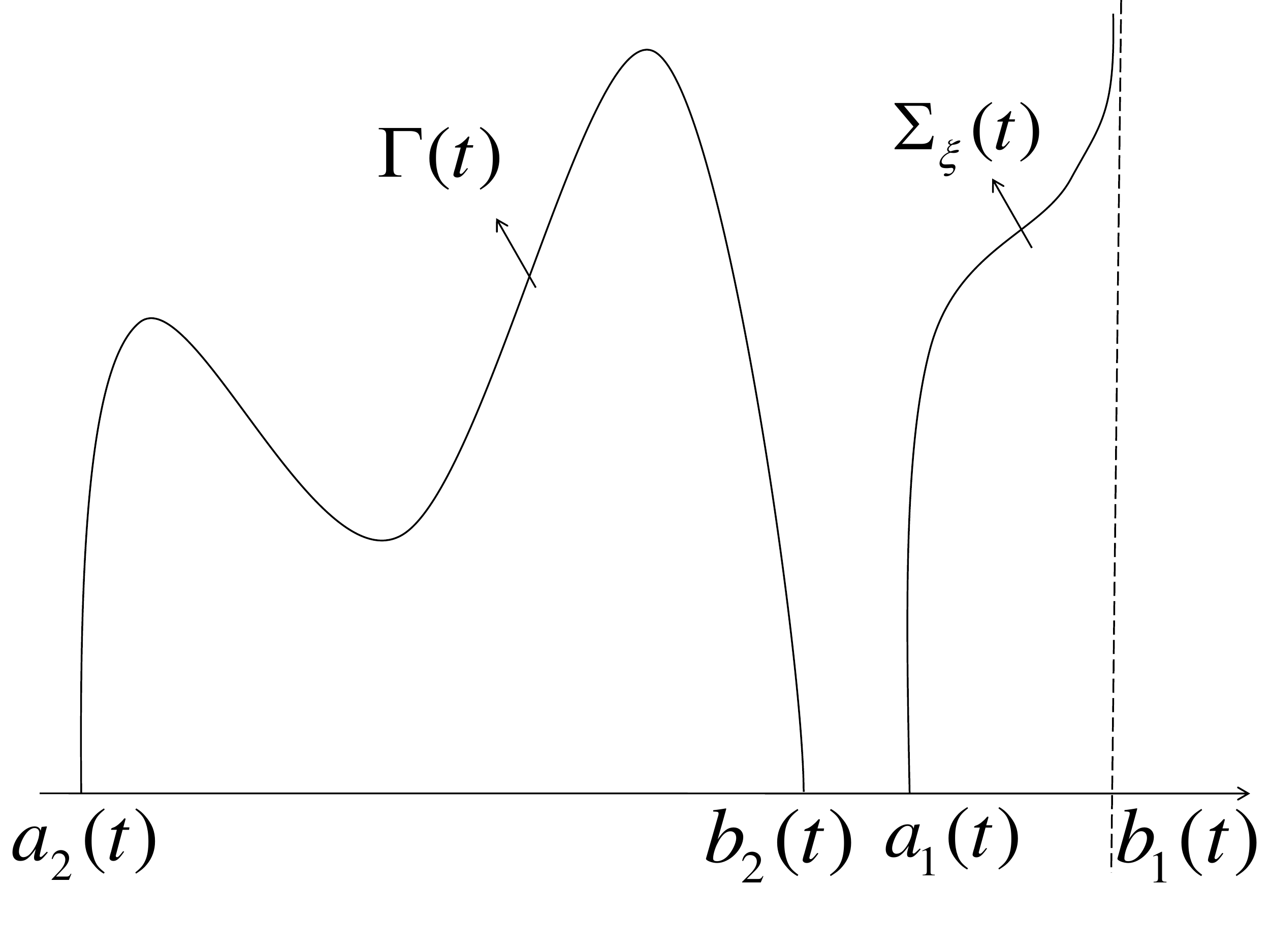}
  \end{center}
  \caption{Case 4}
  \label{fig:61}
 \end{minipage}
 \begin{minipage}{0.49\hsize}
  \begin{center}
   \includegraphics[width=5cm]{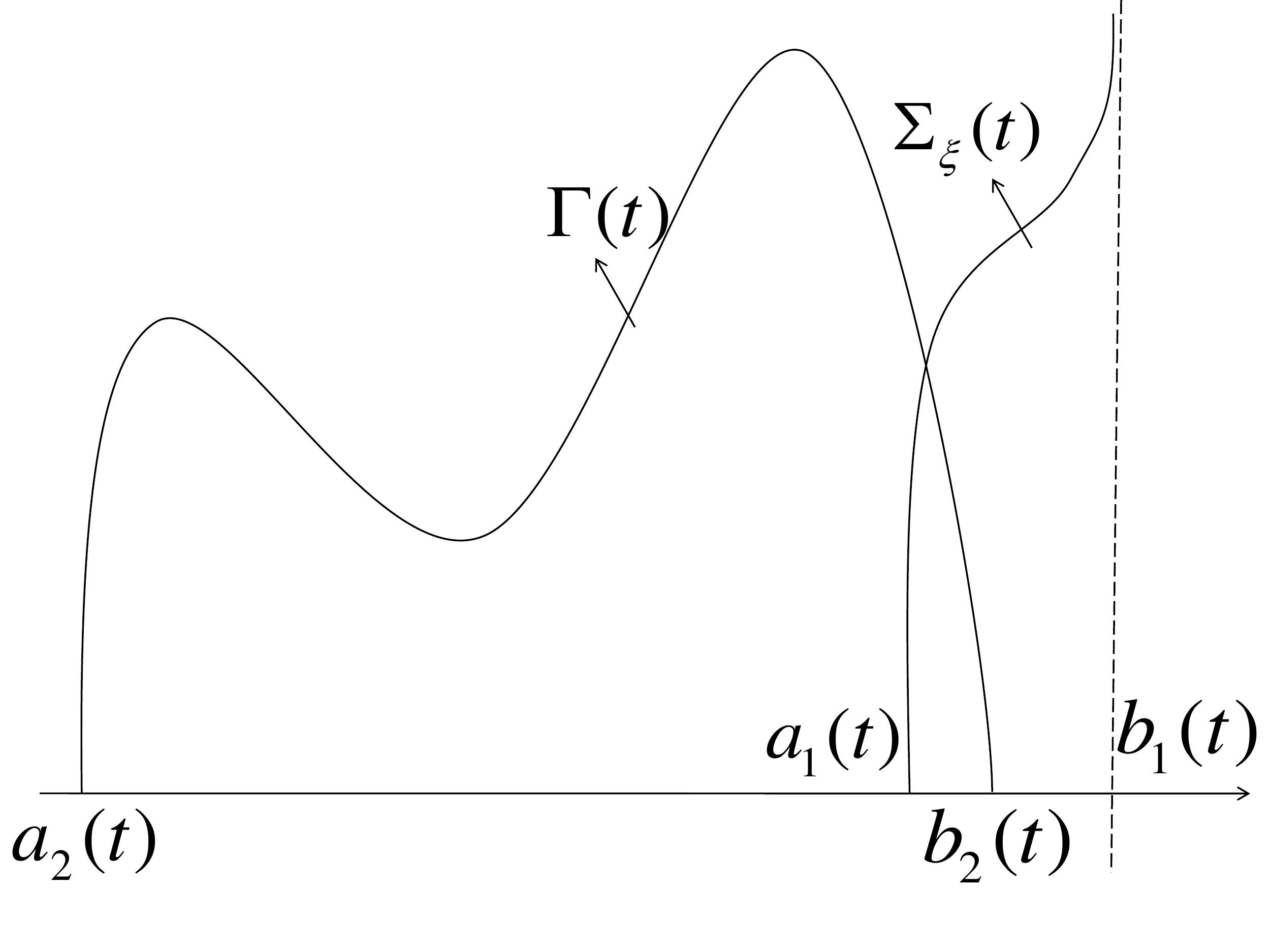}
  \end{center}
  \caption{Case 4}
  \label{fig:62}
 \end{minipage}
\end{figure}

Obviously, $\Sigma_{\xi}(t)$ dosen't intersect $\Gamma(t)$, $t<t_6$. For $\lim\limits_{x\rightarrow b_2(t)}u_x(x,t)=-\infty$ and $\lim\limits_{x\rightarrow a_1(t)}v_x(x-\xi,t)=\infty$, there exists $\epsilon$ such that 
$$
u_x(x,t)-v_x(x-\xi,t)<0,\ a_1(t)\leq x\leq b_2(t),\ t_6<t<t_6+\epsilon.
$$
Seeing $u(a_1(t),t)-v(a_1(t)-\xi,t)>0$ and $u(b_2(t),t)-v(b_2(t)-\xi,t)<0$, $u(x,t)$ intersects $v(x-\xi,t)$ only once in $[a_1(t),b_2(t)]$, $t_6<t<t_6+\epsilon$. Consequently, $\Sigma_{\xi}(t)$ intersects $\Gamma(t)$ only once, $t_6<t<t_6+\epsilon$. So by Theorem \ref{thm:sl} we have $\Sigma_{\xi}(t)$ intersects $\Gamma(t)$ only once, $t>t_6$. 

The other conditions can be proved similarly as the three cases above. We see that the intersection number increases only in Case 3.

Then we can conclude that

1. if $a_1(0)<a_2(0)$, $\Sigma_{\xi}(t)$ does not intersect $\Gamma(t)$.

2. if $a_2(0)<a_1(0)<b_2(0)$, $\Sigma_{\xi}(t)$ intersects $\Gamma(t)$ at most once.

3. if $b_2(0)<a_1(0)$, $\Sigma_{\xi}(t)$ intersects $\Gamma(t)$ at most once.(Only in this case, the intersection number may increase)

We complete the proof.
\end{proof}

\begin{rem}\label{rem:intersection1}
The intersection number between two closed, compact, rotationally symmetric hypersurfaces $\Gamma_1(t)=\{(x,y)\in \mathbb{R}\times\mathbb{R}^n\mid r=u_1(x,t),a_1(t)\leq x\leq b_1(t)\}$, $\Gamma_2(t)=\{(x,y)\in \mathbb{R}\times\mathbb{R}^n\mid r=u_2(x,t),a_2(t)\leq x\leq b_2(t)\}$ is denoted by $\mathcal{Z}(t):=\mathcal{Z}[\Gamma_1(t),\Gamma_2(t)]$. If $\Gamma_i(t)$ evolve by $V=-\kappa+A$ in $\mathbb{R}^{n+1}$, seeing Theorem \ref{thm:sl} and the proof in Lemma \ref{lem:inters}, we can similarly prove 
\\
(a) $\mathcal{Z}(t)$ does not increase when $t$ satisfies $\mathcal{Z}(\Gamma_1(t),\Gamma_2(t))>0$.
\\
(b) If $\mathcal{Z}(t_0)=0$, then $\mathcal{Z}(t)\leq 1$, $t_0<t<T$. 

Observing the proof of Case 3 in Lemma \ref{lem:inters}, it also holds that the intersection number will possibly increase once only for $a_1(0)>b_2(0)$ or $a_2(0)>b_1(0)$ in this remark. The results in this remark can be proved similarly as Lemma \ref{lem:inters}.

Observing the opinion in this remark similarly, since $\mathcal{Z}(0)\leq1$ in Lemma \ref{lem:inters}, there holds $\mathcal{Z}(t)\leq 1$ for $0<t<T$.
\end{rem}

Using the intersection argument, we can prove the following theorem.

\begin{thm}\label{thm:gu} Let $\Gamma(t)$, $t\in [0,T)$, be a family of smooth hypersurfaces evolving by $V=-\kappa+A$ in $\mathbb{R}^{n+1}$. If $\Gamma(0)$ is obtained by rotating the graph of a function around the $x$-axis, then so are the $\Gamma(t)$ for $t\in[0,T).$
\end{thm}

For the proof of Theorem \ref{thm:gu}, we see that $\Gamma(t)$ is also rotationally symmetric because the equation is rotationally invariance. Since $\Gamma(0)$ is obtained by rotating the graph of a function around the $x$-axis, $\Gamma(0)$ can be written into $\Gamma(0)=\{(x,y)\in \mathbb{R}\times\mathbb{R}^n\mid r=v_0(x)\}$ for some function $v_0(x)$. It means that all straight vertical line $x=c$ intersects $\Gamma(0)$ at most once. Using the same argument in Lemma \ref{lem:inters},  all $x=c$ intersects $\Gamma(t)$ at most once. Then $\Gamma(t)$ can be written into $\{(x,y)\in \mathbb{R}\times\mathbb{R}^n\mid r=u(x,t)\}$. We omit the details.

For the following argument, we only consider the results in $\mathbb{R}^2$. In our problem, the curve evolving by $V=-\kappa+A$ maybe intersect itself at $r=0$. To conquer this difficulty, we give the definition of the $\alpha$-domain first used by \cite{AAG}. 

\begin{defn}\label{def:alphad} We say a domain is an $\alpha$-domain if

(1). Let $U\subset \mathbb{R}^{2}$ be an open set of the form
$$
U=\{(x,y)\in\mathbb{R}^2\mid r<u(x)\}.
$$

(2). $I=\{x\in\mathbb{R}\mid u(x)>0 \}$ is a bounded, connected interval. Then there exist $a_1<a_2$ such that $\partial I=\{a_1,a_2\}$.

(3). $u$ is smooth on $I$;

(4). $\partial U$ intersects each cylinder $\partial C_{\rho}$ with $0<\rho\leq\alpha$ twice and these intersections are transverse, where $C_{\rho}=\{(x,y)\in\mathbb{R}^2\mid r<\rho\}$.
\end{defn}

\begin{figure}[htbp]
	\begin{center}
            \includegraphics[height=5cm]{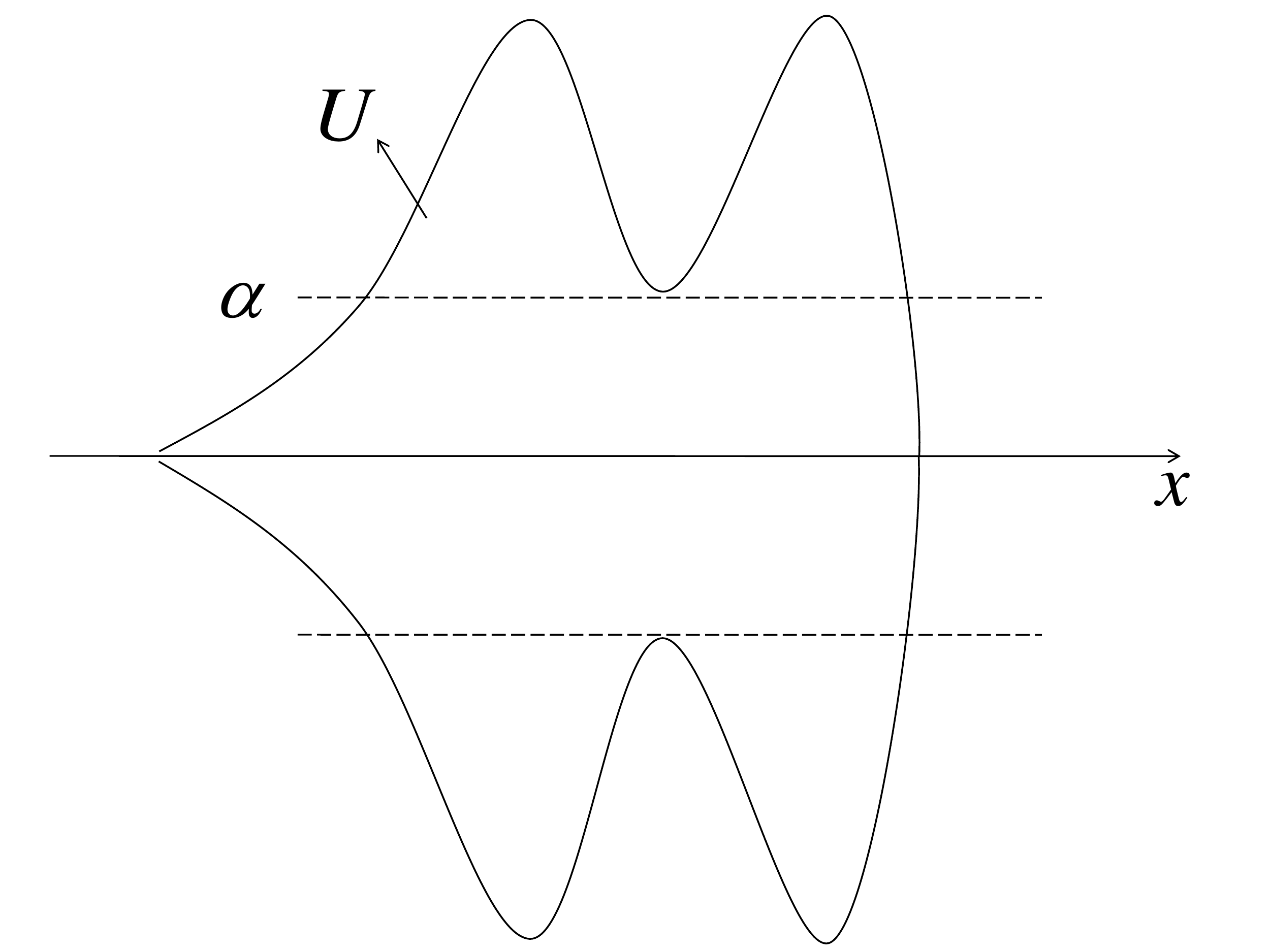}
		\vskip 0pt
		\caption{$\alpha$-domain}
        \label{fig:alphadom}
	\end{center}
\end{figure}
From Figure \ref{fig:alphadom}, we observe that the boundary $\partial U$ of an $\alpha$-domain $U$ does not intersect itself at $y=0$. The condition (3) implies $\partial U$ is a smooth curve, except possibly at its endpoints $(a_1,0),(a_2,0)$. The condition (4) implies that there exist $\delta_1,\delta_2>0$ such that
$$
u(a_1+\delta_1)=u(a_2-\delta_2)=\alpha,
$$
and
$$
u^{\prime}(x)=\left\{
\begin{array}{lcl}
>0,\ x\in(a_1,a_1+\delta_1],\\
<0,\ x\in[a_2-\delta_2,a_2).
\end{array}
\right.
$$
Therefore, the inverse of $u|_{[a_1,a_1+\delta_1]}$ and $u|_{[a_2-\delta_2,a_2]}$ exist, denoted by $v_1$, $v_2:[0,\alpha]\rightarrow\mathbb{R}$. By the implicity theorem, they are smooth in $(0,\alpha]$. Moreover, $v_1^{\prime}(r)>0$, $v_2^{\prime}(r)<0$, $(0<r\leq\alpha)$ and
$$
\partial U\cap C_{\alpha}=\{(x,y)\in\mathbb{R}^2\mid0\leq r\leq\alpha,\ x=v_i(r),\ i=1,2\}.
$$
The two components of $\partial U\cap C_{\alpha}$ are called the left and right caps of $\partial U$.

\begin{lem}\label{lem:alphad2} Let $U$ be an $\alpha$-domain. Then there exists $t_U>0$ such that $D(t)$ denoted the open evolution with $D(0)=U$ is an $(\alpha+At)$-domain, $0<t<t_U$.
\end{lem}

For proving the Lemma \ref{lem:alphad2}, we need the following lemma. 
\begin{lem}\label{lem:in} Let $(u,a,b)$ be the solution of
\begin{equation}\label{eq:eq1}
u_t=\frac{u_{xx}}{1+u_x^2}+A\sqrt{1+u_x^2}, \ x\in(a(t),b(t)), \ 0<t< T,
\end{equation}
\begin{equation}\label{eq:eq2}
u(a(t),t)=0,\ u(b(t),t)=0,\ 0\leq t< T,
\end{equation}
\begin{equation}\label{eq:eq3}
u_x(a(t),t)=+\infty,\ u_x(b(t),t)=-\infty,\ 0\leq t< T,
\end{equation}
\begin{equation}\label{eq:eq4}
u(x,0)=u_0(x),\ a(0)\leq x\leq b(0),
\end{equation}
where $u_0\in C[a(0),b(0)]\cap C^1(a(0),b(0))$.

We denote $\gamma_1(t)$ consisting of the following three parts, $\gamma_{11}(t)=\{(x,y)\in \mathbb{R}^2\mid x=a(t), y<0\}$, $\gamma_{12}(t)=\{(x,y)\in \mathbb{R}^2\mid x=b(t), y<0\}$ and $\gamma_{13}(t)=\{(x,y)\in \mathbb{R}^2\mid y=u(x,t), a(t)\leq x\leq b(t)\}$. For all $C\in \mathbb{R}$,  denote $\gamma_2(t)=\{(x,y)\in \mathbb{R}^2\mid y=C+At\}$. 

Then the intersection number $\mathcal{Z}[\gamma_1(t),\gamma_2(t)]$ is not increasing in $t\in [0,T)$.
\end{lem}
\begin{proof} It is sufficient to show for $t_1\in(0,T)$, there exists $\epsilon>0$ such that $\mathcal{Z}[\gamma_1(t),\gamma_2(t)]$ is not increasing on $(t_1-\epsilon,t_1+\epsilon)$. For convenience, we denote $a(t_1)=x_1$ and $b(t_1)=x_2$. 

Next we will prove this result by three cases separately.

First, if $C+At_1\leq0$. Since $u_x(x_1,t_1)=+\infty$ ,$u_x(x_2,t_1)=-\infty$ and $u(x_1,t_1)=u(x_2,t_1)=0$, there exist $\epsilon$, $\delta>0$ such that
\begin{equation}\label{eq:481interlemma}
u(x_1+\delta,t)>C+At,\ u(x_2-\delta,t)>C+At,\ t_1-\epsilon\leq t< t_1+\epsilon,
\end{equation}
\begin{equation}\label{eq:482interlemma}
u_x(x,t)>0,\ x\in(a(t),x_1+\delta),\ u_x(x,t)<0,\ x\in(x_2-\delta,b(t)),\ t_1-\epsilon\leq t< t_1+\epsilon,
\end{equation}
and
$$
(x_1+\delta,x_2-\delta)\subset(a(t),b(t)),\ t_1-\epsilon\leq t< t_1+\epsilon.
$$

\textbf{Case 1.} $C+At_1<0$.

Let $\mathcal{Z}[\gamma_1(t),\gamma_2(t)]=h(t)+\mathcal{Z}[\gamma_{13}(t),\gamma_2(t)]$, where $h(t)$ is denoted as the intersection number between $\gamma_{2}(t)$ and half lines $\gamma_{11}(t)$, $\gamma_{12}(t)$.

In this condition, there exists $\epsilon$ such that $C+A(t_1+\epsilon)<0$. Then there holds 
$h(t)\equiv2$, $t_1-\epsilon\leq t< t_1+\epsilon$. (\ref{eq:481interlemma}) and (\ref{eq:482interlemma}) imply $\mathcal{Z}[\gamma_{13}(t),\gamma_2(t)]=\mathcal{Z}_{[x_1+\delta,x_2-\delta]}[\gamma_{13}(t),\gamma_2(t)]$, for $t\in(t_1-\epsilon,t_1+\epsilon)$. Where $\mathcal{Z}_{I}[\Gamma_{1}(t),\Gamma_2(t)]$ denotes the intersection number between $\Gamma_1(t)$ and $\Gamma_2(t)$ in set $I$. By (\ref{eq:481interlemma}) again and Theorem D in \cite{A1}, $\mathcal{Z}_{[x_1+\delta,x_2-\delta]}[\gamma_{13}(t),\gamma_2(t)]$ does not increase for $t_1-\epsilon\leq t< t_1+\epsilon$.

Therefore $\mathcal{Z}[\gamma_1(t),\gamma_2(t)]$ is non-increasing for $t\in(t_1-\epsilon,t_1+\epsilon)$. 

\textbf{Case 2.} $C+At_1=0$.

For $t_1\leq t< t_1+\epsilon$, obviously, $h(t)=0$. And seeing $u(a(t),t)=u(b(t),t)=0$, (\ref{eq:481interlemma}) and (\ref{eq:482interlemma}), $C+At$ intersects $u(x,t)$ exactly twice in $[a(t),x_1+\delta)\cup(x_2-\delta,b(t)]$. Then $\mathcal{Z}_{[a(t),x_1+\delta)\cup(x_2-\delta,b(t)]}[\gamma_{13}(t),\gamma_2(t)]=2$, $t_1\leq t< t_1+\epsilon$. Therefore,
$$
h(t)+\mathcal{Z}_{[a(t),x_1+\delta)\cup(x_2-\delta,b(t)]}[\gamma_{13}(t),\gamma_2(t)]=2,\ t_1\leq t< t_1+\epsilon.
$$

For $t_1-\epsilon< t\leq t_1$, obviously, $C+At<0$, then there holds $h(t)\equiv2$. (\ref{eq:481interlemma}) and (\ref{eq:482interlemma}) imply that  $\mathcal{Z}_{[a(t),x_1+\delta)\cup(x_2-\delta,b(t)]}[\gamma_{13}(t),\gamma_2(t)]=0$, $t_1-\epsilon< t\leq t_1$. Then we have
$$
h(t)+\mathcal{Z}_{[a(t),x_1+\delta)\cup(x_2-\delta,b(t)]}[\gamma_{13}(t),\gamma_2(t)]=2,\ t_1-\epsilon< t< t_1+\epsilon.
$$

On the other hand, by (\ref{eq:481interlemma}) and Theorem D in \cite{A1}, there holds $\mathcal{Z}_{[x_1+\delta, x_2-\delta]}[\gamma_{13}(t),\gamma_2(t)]$ is not increasing, $t_1-\epsilon<t< t_1+\epsilon$.

Therefore $\mathcal{Z}[\gamma_1(t),\gamma_2(t)]$ is non-increasing for $t\in(t_1-\epsilon,t_1+\epsilon)$. 

\begin{figure}[htbp]
	\begin{center}
            \includegraphics[height=5cm]{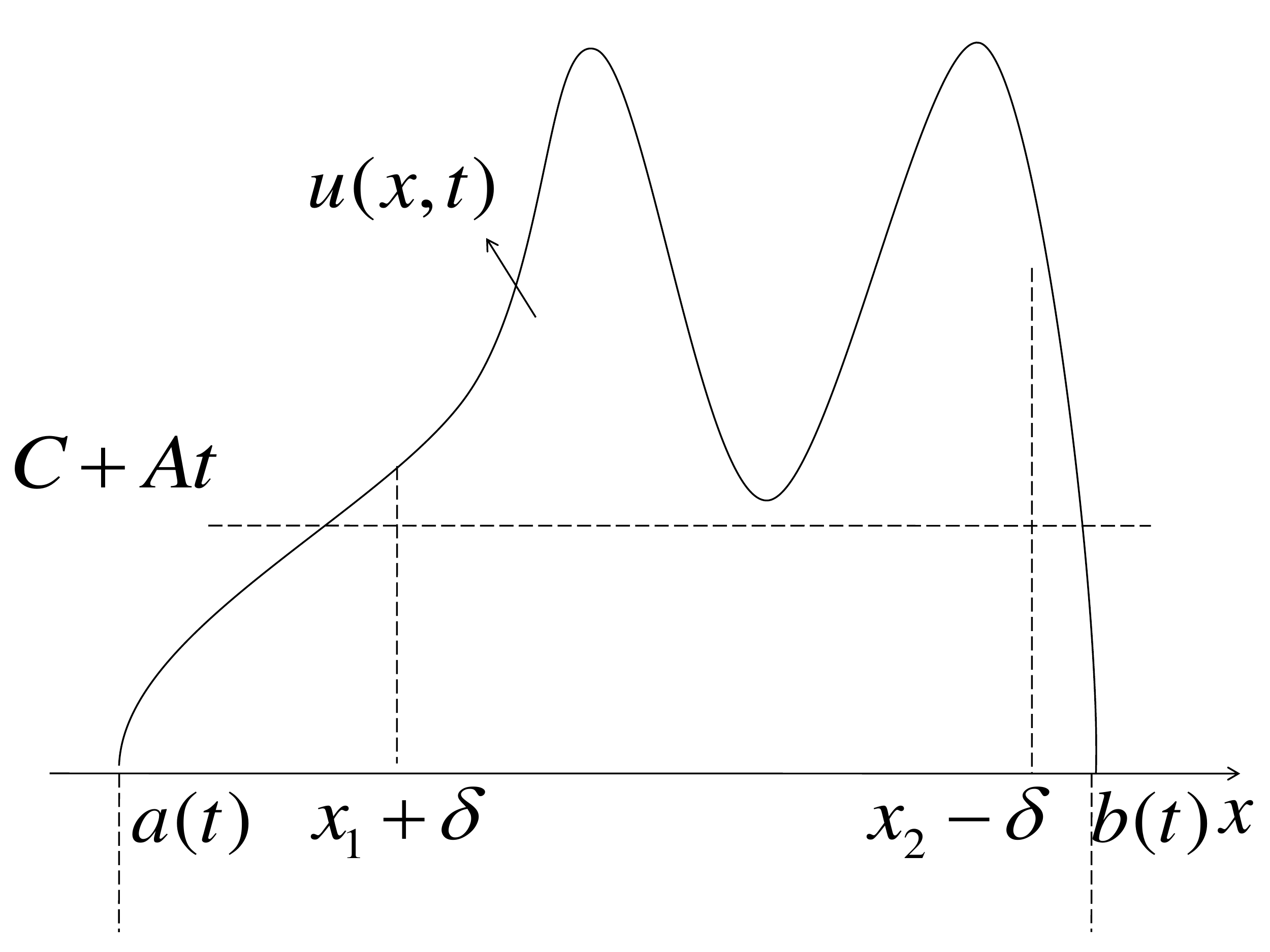}
		\vskip 0pt
		\caption{Proof of the case 3 in Lemma \ref{lem:in}}
        \label{fig:lemalphadom2}
	\end{center}
\end{figure}

\textbf{Case 3.} $C+At_1>0$.

In this case, there exists $\epsilon$ such that $C+A(t_1-\epsilon)>0$. Then there hold
 $$h(t)\equiv0$$
 and 
$$C+At>u(a(t),t)=u(b(t),t)=0,$$
$t\in(t_1-\epsilon,t_1+\epsilon)$. So by Theorem D in \cite{A1}, $\mathcal{Z}[\gamma_{13}(t),\gamma_2(t)]$ is non-increasing and finite for $t\in(t_1-\epsilon,t_1+\epsilon)$. Consequently, $\mathcal{Z}[\gamma_1(t),\gamma_2(t)]$ is finite and non-increasing in $t\in(t_1-\epsilon,t_1+\epsilon)$.

\end{proof}

\begin{proof}[Proof of Lemma \ref{lem:alphad2}]  Since $ U$ is an $\alpha$-domain, using Theorem \ref{thm:partialUmeancurvature}, $\partial D(t)=\{(x,y)\in\mathbb{R}^2\mid |y|=u(x,t), a(t)\leq x\leq b(t)\}$, where $(u,a,b)$ satisfies (\ref{eq:eq1}), (\ref{eq:eq2}), (\ref{eq:eq3}). Moreover, there exists a maximal time $T_U>0$ such that $\partial D(t)$ is smooth, $0<t<T_U$.

Since $U$ is not contained in the cylinder $\overline{C_{\alpha}}$, there exists a small ball $B_{\epsilon}(P)\subset U\setminus\overline{C_{\alpha}}$. By (1) in Theorem \ref{thm:order}, $D(t)$ contains the ball $B_{\epsilon(t)}(P)$ for $0<t<\delta_1$. Where $\epsilon(t)$ satisfies
\begin{equation}\label{eq:ball2}
\epsilon^{\prime}(t)=A-\frac{1}{\epsilon(t)},\ 0<t<\delta,
\end{equation}
with $\epsilon(0)=\epsilon$. Since $B_{\epsilon}(P)\cap\overline{C_{\alpha}}=\phi$, by (1b) in Theorem \ref{thm:conti}, there exists $t_1>0$, such that $B_{\epsilon(t)}(P)\cap \overline{C_{\alpha+At}}=\phi$, $0<t<t_1$.

For $0<\rho<\alpha+At_0$, $0<t_0\leq t_{U}^{\alpha}$, where $t_{U}^{\alpha}=\min\{\delta_1,t_1\}$, $y=\rho-At_0$ intersects $\gamma_1(0)$ exactly twice($\gamma_1(t)$ is constructed in Lemma \ref{lem:in}). Lemma \ref{lem:in} implies that $ y=\rho$ intersects $\gamma_1(t_0)$ at most twice. Consequently, $y=\rho$ intersects $y=u(x,t_0)$ at most twice, for $0<t_0< \min\{t_{U}^{\alpha},T_U\}$. On the other hand, there holds $B_{\epsilon(t_0)}(P)\subset D(t_0)\setminus\overline{C_{\alpha+At_0}}$, $0<t_0< \min\{t_{U}^{\alpha},T_U\}$, then $y=\rho$ intersects $y=u(x,t_0)$ at least twice, $0<t_0< \min\{t_{U}^{\alpha},T_U\}$. Therefore we have $\partial C_{\rho}$ intersects $\partial D(t_0)$ exactly twice, $0<t_0<\min\{t_{U}^{\alpha},T_U\}$. 

Choosing $t_U=\min\{t_U^{\alpha},T_U\}$, $D(t)$ is an $(\alpha+At)$-domain, $0<t<t_U$.  The proof is completed.
\begin{figure}[htbp]
 \begin{minipage}{0.49\hsize}
  \begin{center}
   \includegraphics[width=7cm]{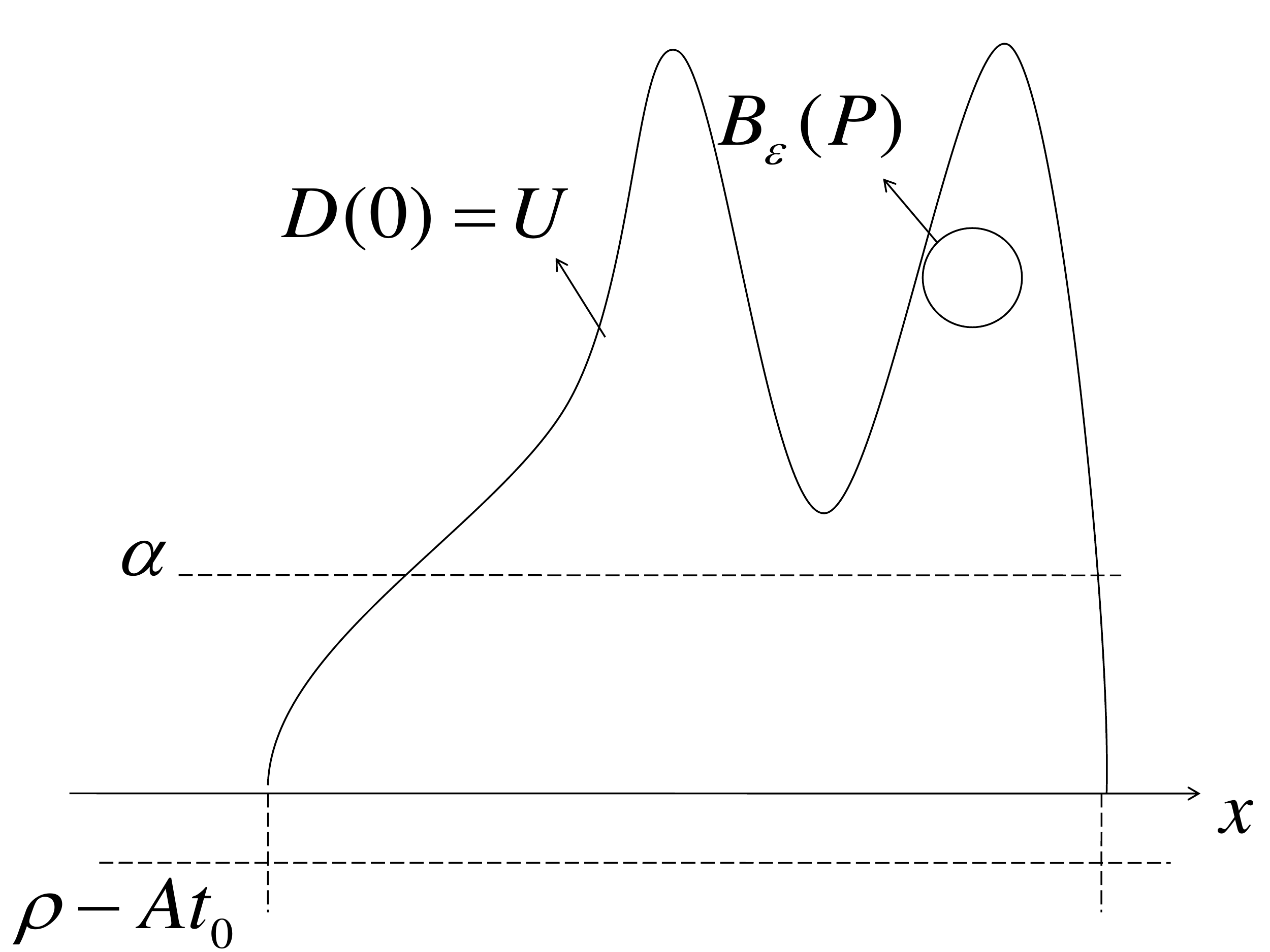}
  \end{center}
  \caption{Proof of Lemma \ref{lem:alphad2}}
  \label{fig:4711}
 \end{minipage}
 \begin{minipage}{0.49\hsize}
  \begin{center}
   \includegraphics[width=7cm]{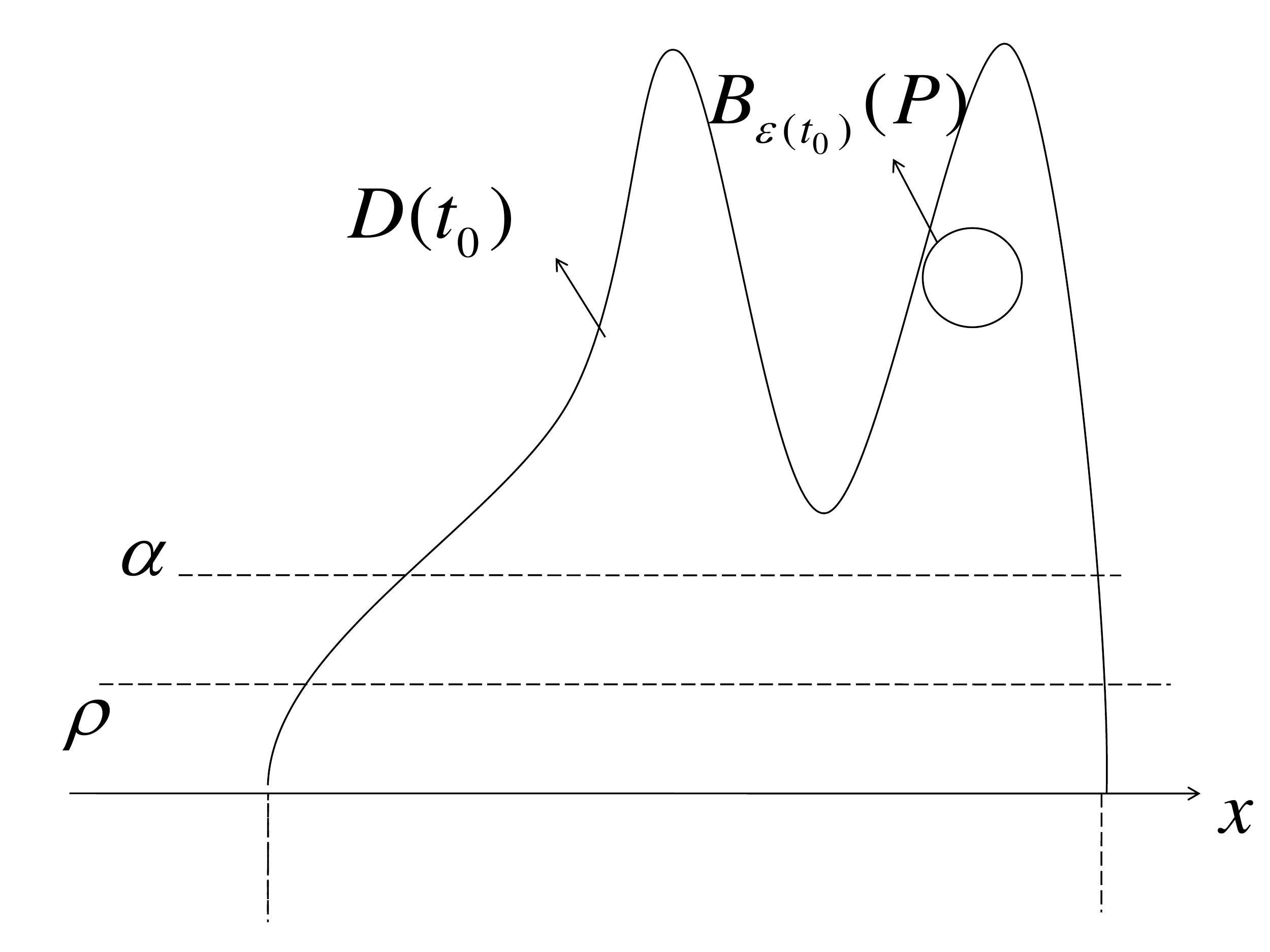}
  \end{center}
  \caption{Proof of Lemma \ref{lem:alphad2}}
  \label{fig:4712}
 \end{minipage}
\end{figure}
\end{proof}

\begin{prop}\label{pro:sin} For $t_U^{\alpha}$ and $T_U$ given in the proof Lemma \ref{lem:alphad2}, $t_U^{\alpha}\leq T_U$.
\end{prop}

To prove the previous proposition we need the following lemma.

\begin{lem}\label{lem:sing1} Assume that $D(t)=\{(x,y)\mid |y|<u(x,t), a(t)\leq x\leq b(t)\}$ is a $\rho$-domain, $0<t<T$. Let  $w_1<w_2$ such that
$$
C_\rho\cap\partial D(t)=\{(x,y)\mid x=w_1(y,t),\ x=w_2(y,t)\}.
$$
 Then
$$\lim\limits_{t\rightarrow T}w_1(y,t)=w_1(y,T)\ \ \ \ \ \textrm{and} \ \ \ \ \lim\limits_{t\rightarrow T}w_2(y,t)=w_2(y,T)$$
exist and these convergences are uniform for $|y|\leq\frac{\rho}{2}$. Moreover, $a(T)=:v_1(0,T)<v_1(r,T)$ and $b(T)=:v_2(0,T)>v_2(r,T)$, $0<r<\frac{\rho}{2}$,
where $v_1(r,t)=w_1(y,t)$ and $v_2(r,t)=w_2(y,t)$.
\end{lem}
\begin{proof} $w_1(y,t)$ and $w_2(y,t)$ satisfy the equation (\ref{eq:graph}), respectively for "$\mp$". We only prove for $w_1(y,t)$. Since $w_1$ is uniformly bounded, Corollary \ref{cor:es} and Remark \ref{rem:hes} imply that derivatives $\frac{\partial ^j}{\partial y^j} w_1$, $j=1,2$, are uniformly bounded for $0\leq|y|\leq\frac{\rho}{2}$, $\frac{T}{2}\leq t<T$. Consequently, $\frac{\partial w_{1}}{\partial t}$ is bounded for $0\leq|y|\leq\frac{\rho}{2}$, $\frac{T}{2}\leq t<T$. So there exists $w_1(y,T)$ such that $w_1(y,t)$ converges to $w_1(y,T)$ uniformly for $0\leq|y|\leq\frac{\rho}{2}$, $\frac{T}{2}\leq t<T$.

Note there hold 
$$
\frac{\partial v_1}{\partial r}(\frac{\rho}{2},t)>0,\ \frac{\partial v_1}{\partial r}(0,t)=0,\ 0<t<T
$$ 
and 
$$
\frac{\partial v_1}{\partial r}(r,0)>0,\ 0< r<\frac{\rho}{2}.
$$
 Since $p=\frac{\partial v_1}{\partial r}$ satisfies (\ref{eq:deriveeq}), maximum principle implies
 $$
\frac{\partial v_1}{\partial r}>0,\ 0<r<\frac{\rho}{2},\ 0<t\leq T.
$$ 
Therefore $v_1(0,T)<v_1(r,T)$, for $0<r<\frac{\rho}{2}$.
\end{proof}

\begin{proof}[Proof of Proposition \ref{pro:sin}.] If $T_U<t_U^{\alpha}$. By Lemma \ref{lem:alphad2}, there exists $\rho>0$ such that $D(t)$ is a $\rho$-domain, for $0<t<T_U$. 

We divide $\partial D(t)$ into two parts: $\partial D(t)=(\partial D(t)\cap\{r< \rho/2\})\cup(\partial D(t)\cap\{r\geq \rho/2\})$.
\\
{\bf Step 1.} $\partial D(t)\cap\{r< \rho/2\}$

 Since $\partial D(t)$ is a $\rho$-domain, there exist $w_1<w_2$ such that $\partial D(t)\cap\{r<\rho\}=\{(x,y)\mid x=w_1(y,t), |y|<\rho\}\cup\{(x,y)\mid x=w_2(y,t), |y|<\rho\}$. By the same argument as in Lemma \ref{lem:sing1}, $\frac{\partial ^j}{\partial y^j} w_i$, $j=1,2$, $i=1,2$, are uniformly bounded for $0\leq|y|\leq\frac{\rho}{2}$, $\frac{T_U}{2}\leq t<T_U$. Therefore, the mean curvature of $\partial D(t)\cap\{r< \rho/2\}$ is bounded for $\frac{T_U}{2}\leq t<T_U$. 
\\
{\bf Step 2.} $\partial D(t)\cap\{r\geq\rho/2\}$

Recalling $\partial D(t)=\{(x,y)\mid |y|=u(x,t),a(t)\leq x\leq b(t)\}$, by Lemma \ref{lem:sing1}, there hold $a(T_U)<v_1(\rho/2,T_U)$ and $b(T_U)>v_2(\rho/2,T_U)$. Then for any $\epsilon$ small enough, $t$ is close to $T_U$ such that
\begin{equation}\label{eq:subset1}
(v_1(\rho/2,t),v_2(\rho/2,t))\subset (a(T_U)+\epsilon,b(T_U)-\epsilon).
\end{equation}
 Corollary \ref{cor:es} and Remark \ref{rem:hes} imply that $u_x$ and $u_{xx}$ are uniformly bounded for $x\in(a(T_U)+\epsilon,b(T_U)-\epsilon)$, $t$ close to $T_U$. (\ref{eq:subset1}) implies that $u_x$ and $u_{xx}$ are uniformly bounded for $x\in(v_1(\rho/2,t),v_2(\rho/2,t))$, $t$ close to $T_U$. The curvature of $\partial D(t)\cap\{r\geq\rho/2\}$ is bounded for $t$ close to $T_U$.

Consequently, the curvature of $\partial D(t)$ is uniformly bounded as $t\uparrow T_U$. It contradicts to $\partial D(t)$ becoming singular at $T_U$.
\end{proof}

\begin{rem}\label{rem:time} In Lemma \ref{lem:alphad2}, $0<t<\min\{t_{U}^{\alpha},T_U\}$ can be replaced by $0<t<t_{U}^{\alpha}$. Seeing the choice of $t_{U}^{\alpha}$, if $U\subset W$, $t_{U}^{\alpha}\leq t_{W}^{\alpha}$.
\end{rem}

{\bf Intersection number principle} Lemma \ref{lem:inters} and Lemma \ref{lem:alphad2} show the possible intersection number between two curves evolving by $V=-\kappa+A$. Here we want to introduce a more general result about the intersection number. Consider the following problem which we call (Q):
\begin{equation}
\left\{
\begin{array}{lcl}
\dis{u_t=\frac{u_{xx}}{1+u_x^2}+A\sqrt{1+u_x^2}},\ x\in(a(t),b(t)),\ 0<t< T,\\
u(a(t),t)=0,\ u(b(t),t)=0,\ 0\leq t< T,\\
u_x(a(t),t)=\tan\theta_-(t),\ u_x(b(t),t)=-\tan\theta_+(t),\ 0\leq t< T,\\
u(x,0)=u_0(x),\ a(0)\leq x\leq b(0),
\end{array}
\right.\tag{Q}
\end{equation}
where $u_0\in C[a(0),b(0)]\cap C^1(a(0),b(0))$ and $\theta_{\pm}(t)$ are smooth functions with values in $[0,\pi/2]$. Let
$$
\gamma_1(t):=\left\{
\begin{array}{lcl}
\{(x,y)\mid y=\tan\theta_-(t)(x-a(t)),y<0\},\ \theta_-(t)<\pi/2\\
\{(x,y)\mid x=a(t),y<0\},\theta_-(t)=\pi/2,
\end{array}
\right.
$$
$$
\gamma_2(t):=\left\{
\begin{array}{lcl}
\{(x,y)\mid y=-\tan\theta_+(t)(x-b(t)),y<0\},\ \theta_+(t)<\pi/2\\
\{(x,y)\mid x=b(t),y<0\},\theta_+(t)=\pi/2,
\end{array}
\right.
$$
and
$$
\gamma_3(t):=\{(x,y)\mid y=u(x,t),a(t)\leq x\leq b(t)\}.
$$
The extension curve of $u(\cdot,t)$ is given by 
$$
\gamma(t):=\gamma_1(t)\cup\gamma_2(t)\cup\gamma_3(t).
$$
\begin{prop}\label{pro:intersection}
Let $u^{1}(x,t)$, $a^{1}(t)<x<b^{1}(t)$ be solution of (Q) for $\theta_{\pm}^1(t)\in[0,\pi/2)$, and $u^2(x,t)$, $a^{2}(t)<x<b^{2}(t)$ be solution of (Q) for $\theta_{\pm}^2(t)=\pi/2$, for $0\leq t<T$. Let $\gamma_i(t)$ be the extension curve of $u^i(x,t)$, respectively. Then $\mathcal{Z}[\gamma_1(t),\gamma_2(t)]$ is non-increasing in $t\in[0,T)$ and is finite for each $t\in[0,T)$. Moreover, $\mathcal{Z}[\gamma_1(t),\gamma_2(t)]$ will drop when $\gamma_1(t)$ intersects $\gamma_2(t)$ tangentially.
\end{prop}
For the proof of this proposition, it is similar as the proof of Lemma \ref{lem:in} above or Proposition 2.4 in \cite{GMSW}. Here we omit it.
\begin{rem}\label{rem:1matanointersection}
(1). Proposition 2.4 in \cite{GMSW} only give the results under $\theta_{\pm}^i\in(0,\pi/2)$, $i=1,2$.  

(2). For $\theta_{\pm}^i=\pi/2$, $i=1,2$, the results in Proposition \ref{pro:intersection} are not true. Indeed, this condition is same as in Remark \ref{rem:intersection1}
\\
(a). If $\mathcal{Z}(u^1(\cdot,0),u^2(\cdot,0))>0$, $\mathcal{Z}(u^1(\cdot,t),u^2(\cdot,t))$ will not increase for $0<t<T$ provided that $\mathcal{Z}(u^1(\cdot,t),u^2(\cdot,t))>0$, $0<t<T$.
\\
(b). If $\mathcal{Z}(u^1(\cdot,0),u^2(\cdot,0))=0$, $\mathcal{Z}(u^1(\cdot,t),u^2(\cdot,t))\leq 1$, $0<t<T$.
\end{rem}


\section{Proof of Theorem \ref{thm:exist} and Theorem \ref{thm:fattening1}}

Denote $U=\{(x,y)\in \mathbb{R}^2\mid |y|<u_0(x),-b_0\leq x\leq b_0\}$, where $u_0$ is given by (\ref{eq:ineven}).  By assumption of $u_0$ in Section 1, we know that $U\cap\{x\geq0\}$ is an $\alpha$-domain with smooth boundary, for some $\alpha>0$. 

We choose vector field $X\in C^1(\mathbb{R}^{2}\setminus\{(0,0)\}\rightarrow\mathbb{R}^{2})$ such that

(i) At any $P\in \partial U$ not on the $x$-axis has $\langle X,\textbf{n}(P)\rangle<0$, $\textbf{n}$ is inward unit normal vector at $P$.

(ii) Near the $(0,0)$, we set $X((x,y))=(0,-y/|y|)$ and set $X=(-1,0)$ near $(b,0)$, $X=(1,0)$ near $(-b,0)$. 
\\
We note that $X$ has no definition at $(0,0)$.

Since $X\neq0 $ on $\partial U\setminus \{(0,0)\}$, there exists a neighbourhood $V\supset\partial U$ such that $|X|\geq \delta>0$ for some $\delta>0$ in $V\setminus \{(0,0)\}$. 

\begin{prop}\label{pro:sigma2} For $\rho$ small enough, there exists a smooth curve $\Sigma\subset V\setminus\{(0,0)\}$ with

(i) $X(P)\notin T_P\Sigma$ at all $P\in\Sigma$,i.e., $\Sigma$ is transverse to the vector field $X$;

(ii) $\Sigma=\partial U$ in $\{(x,y)\mid|y|\geq2\rho\}$;

(iii) $\Sigma\cap\{(x,y)\mid|y|\leq\rho\}$ consists of discs $\Delta_{\pm c}=\{(\pm c,y)\mid|y|\leq\rho\}$ and pipe $B_d=\{(x,y)\mid-d \leq x\leq d,|y|=\rho\}$.
\end{prop}

\begin{figure}[htbp]
	\begin{center}
            \includegraphics[height=5cm]{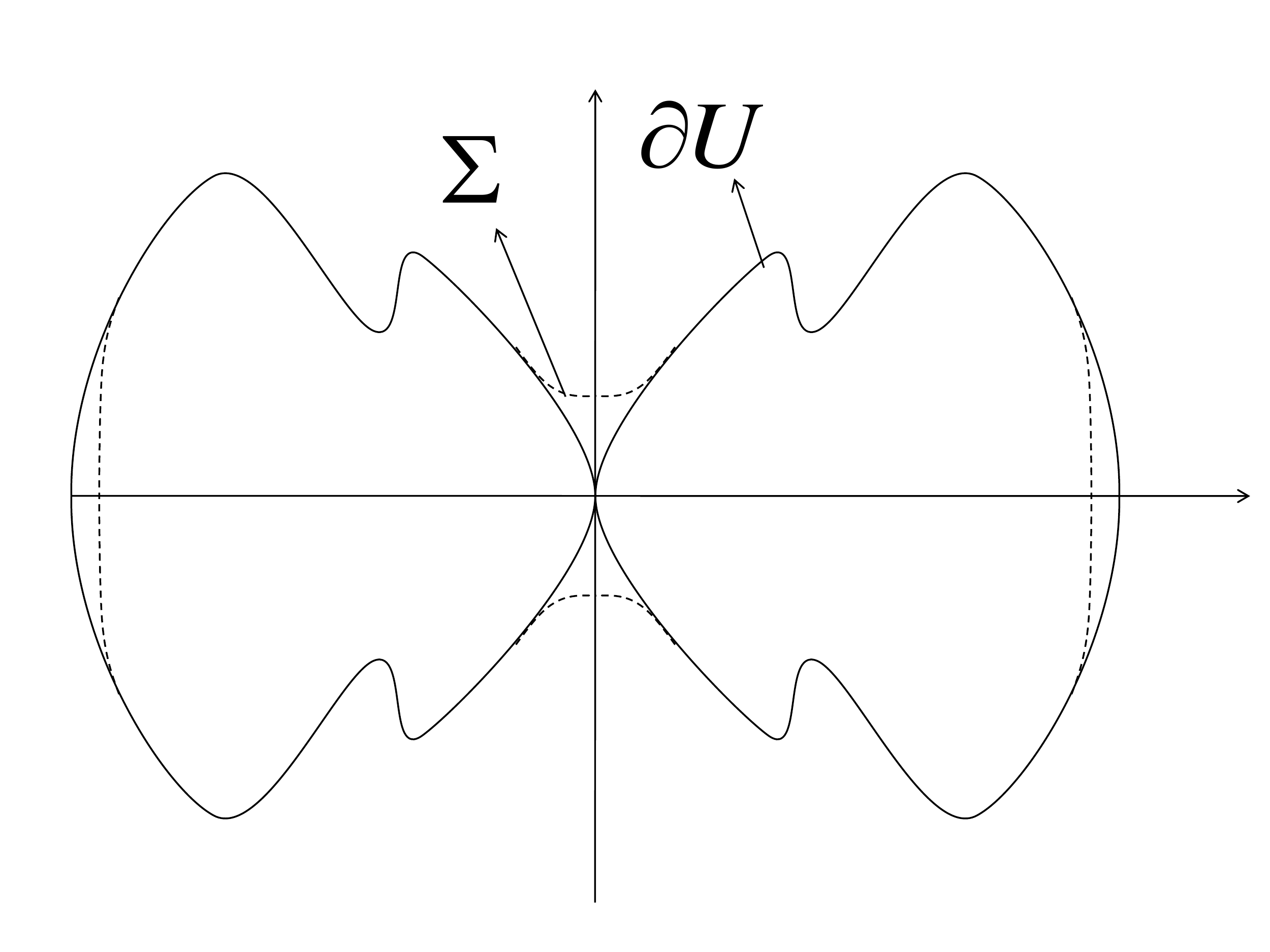}
		\vskip 0pt
		\caption{Proof of Proposition \ref{pro:sigma2}}
        \label{fig:sigma2}
	\end{center}
\end{figure}
\begin{proof} Because $U\cap\{x\geq0\}$ is an $\alpha$-domain, there exist $\delta_j$, $\gamma_j$ and $0<\delta_j<\gamma_j$ such that
$$
u_0(\delta_j)=u_0(\gamma_j)=u_0(-\delta_j)=u_0(-\gamma_j)=\frac{\alpha}{2^j}
$$
and
$$\partial U\cap C_{\alpha}=\{(x,y)\mid x=\pm v(y), |y|<\alpha\}\cup\{(x,y)\mid x=\pm w(y),|y|<\alpha\},$$
where $v,w\in C^{\infty}((-\alpha,\alpha))$ and $0<v(y)<w(y)$ for $|y|<\alpha$.

We let $w_j\in C^{\infty}((-\alpha/2^{j-1},\alpha/2^{j-1}))$ be defined as following
$$
w_j(y)=\left\{
\begin{array}{lcl}
\gamma_{j+2},\ 0\leq\dis{|y|<\frac{\alpha}{2^{j+1}}}\\
w(y),\ \dis{\frac{\alpha}{2^{j}}<|y|<\frac{\alpha}{2^{j-1}}},
\end{array}
\right..
$$

And $u_j\in C^{\infty}((-\delta_{j-1},\delta_{j-1}))$ is defined as following
$$
u_j(x)=\left\{
\begin{array}{lcl}
\dis{\frac{\alpha}{2^{j+1}}},\ x\in[0,\delta_{j+2}]\\
u_0(x),\  x\in[\delta_j,\delta_{j-1})
\end{array}
\right..
$$

Let $\Sigma_j$ consist of three parts: $\{(x,y)\mid |y|=u_j(x),\ x\in(-\delta_j,\delta_j)\}$, $\{(x,y)\mid x=\pm w_j(y), |y|<\alpha/2^j\}$ and $\partial U\cap \{|y|\geq \alpha/2^{j}\}$. It is easy to see that for $j$ sufficient large, $\Sigma_j\subset V\setminus \{(0,0)\}$ satisfies (i), (ii), (iii) for $c=\gamma_{j+2}$, $\rho=\alpha/2^{j+1}$ and $d=\delta_{j+2}$.
\end{proof}

Denote $\sigma(P,t):\Sigma\times(-\delta,\delta)\rightarrow V$($V$ is given at the begining of this section and $\Sigma$ is given by Proposition \ref{pro:sigma2}) the flow generated by vector field $X$ in $\mathbb{R}^{2}$. Precisely, $\sigma(P,t)$ is defined as following:
$$
\left\{
\begin{array}{lcl}
\dis{\frac{d\sigma(P,t)}{dt}=X(\sigma(P,t))},\ P\in \Sigma,\\
\sigma(P,0)=P,\ \ \ \ \ P\in \Sigma.
\end{array}
\right.
$$

Seeing (i) in Proposition \ref{pro:sigma2}, for any $C^{1}$ function $u:\Sigma\rightarrow\mathbb{R}$, ``the image of $u$ under $\sigma$''---$\{\sigma(P,u(P))\mid P\in \Sigma\}$ is a $C^1$ curve. Conversely, for any curve $\Gamma\subset V$ being $C^1$ close to $\Sigma$, there exists a unique $C^1$ function $u:\Sigma\rightarrow\mathbb{R}$ such that $\Gamma=\{\sigma(P,u(P))\mid P\in \Sigma\}$. In other words, the map $\sigma(\cdot,t)$ defines a new coordinate from $\Sigma$ to $V$. Therefore, if $\Gamma(t)\subset V$$(0<t<T)$ is a smooth family of smooth curves and $C^1$ close to $\Sigma$, there exists a unique function $u \in C^\infty(\Sigma\times(0,t))$ such that $\Gamma(t)=\{\sigma(P,u(P,t))\mid P\in\Sigma\}$. Let $z$ be the local coordinate on an open subset of $\Sigma$. If $\Gamma(t)$ evolves by $V=-\kappa+A$, in this coordinate $u$ satisfies the following equation
\begin{equation}\label{eq:para1}
\frac{\partial u}{\partial t}=a(z,u,u_z)\frac{\partial^2u}{\partial z^2}+b(z,u,u_z).
\end{equation}

Here $a$, $b$ are smooth functions of their arguments \cite{A2}(Section 3). $a$ is always positive so that (\ref{eq:para1}) is a parabolic equation. 

For example, $\sigma(\cdot,t)$ is the flow defined as above.  We can easily deduce that 
$$
\sigma(P,t)=\left\{
\begin{array}{lcl}
(x,\rho-t),\ P\in B_d,\\
(-c+t,y),\ P\in \Delta_{-c},\\
(c-t,y),\ P\in \Delta_{c},
\end{array}
\right.
$$
where we choose the local coordinates:
\\
(1). on $B_d$, $(x,\rho y)$ for $|y|=1$;
\\
(2). on $\Delta_{\pm c}$, $(\pm c,y)$. 

Since $y=\pm1$ on $B_d$, $u$ only depends on $x$. Therefore on $B_d$, $a(x,u,u_x)=1/(1+u_x^2)$ and $b(x,u,u_x)=-A\sqrt{1+u_x^2}$. Then $u$ satisfies
\begin{equation}\label{eq:2dimhorieq}
u_t=\frac{u_{xx}}{1+u_x^2}-A\sqrt{1+u_x^2}.
\end{equation}
On $\Delta_{\pm c}$, $u$ only depends on $y$. Then on $\Delta_{\pm c}$, $a(y,u,u_y)=1/(1+u_y^2)$, $b(y,u,u_y)=-A\sqrt{1+u_y^2}$. Therefore $u$ satisfies  (\ref{eq:graph}) for ``$-$'' and $n=1$.

\begin{rem}
In $\mathbb{R}^{n+1}$, $b$ obtained above may not be smooth.
For example,
$$
u_t=\frac{u_{xx}}{1+u_x^2}+\frac{n-1}{\rho-u}-A\sqrt{1+u_x^2},
$$
on $\{(x,y)\in\mathbb{R}\times\mathbb{R}^n\mid |y|=\rho, -d<x<d\}$. In this case, $b=\frac{n-1}{\rho-u}-A\sqrt{1+u_x^2}$. It is easy to see when $u=\rho$, $b$ is not smooth. This is the most different between 2-dimension and higher dimension.
\end{rem}

\begin{lem}\label{lem:max} For $v(x,t)$ being smooth function on $V\times(0,T)$, where $V$ is a compact set, we denote $m(t)$ as $$
m(t)=\max\{v(x,t)\mid x\in V\}.
$$
Then there exists $P_t\in V$ such that $v(P_t,t)=m(t)$ and $m^{\prime}(t)=v_t(P_t,t)$ for $t>0$.
\end{lem}
It is a well known result. For example, the result can be found in \cite{M}.

\begin{prop}\label{pro:uniq2} Let $\Gamma_1$, $\Gamma_2$ be two families of curves with $\sigma^{-1}(\Gamma_j)$ the graph of $u_j(\cdot,t)$ for certain $u_j\in C(\Sigma\times[0,T))$. Assume $u_j$ are smooth on $\Sigma\times(0,T)$ and smooth on $\Sigma\setminus(\Delta_{\pm c}\cup B_d)\times[0,T)$. If $\Gamma_1(0)=\Gamma_2(0)$, there holds $\Gamma_1(t)=\Gamma_2(t)$, $0\leq t<T$.

\end{prop}
\begin{proof} Consider $v(P,t)=u_1(P,t)-u_2(P,t)$. From our assumptions, we have $v\in C(\Sigma\times[0,T))$ and that $v$ is smooth on $\Sigma\setminus(\Delta_{\pm c}\cup B_d)\times[0,T)$, as well as on $\Sigma\times(0,T)$. Moreover $v(P,0)\equiv0$. Define $m(t)=\max\{v(P,t)\mid P\in \Sigma\}$ and for each $0\leq t<T$ with $m(t)>0$. We want to show that $m^{\prime}(t)\leq Cm(t)$ for some constant $C$. Choose $P_t$ as in Lemma \ref{lem:max} such that $m(t)=v(P_t,t)$ and $m^{\prime}(t)=v_t(P_t,t)$.

\textbf{Case 1.} $P_t\in B_d$, since $u_j$ satisfy the equation (\ref{eq:2dimhorieq}), $v$ satisfies a parabolic equation
$$
v_t=a_1(x,t)v_{xx}+b_1(x,t)v_x,
$$
where $a_1(x,t)$ and $b_1(x,t)$ is smooth, and $a_1(x,t)>0$. Since $v$ attains its maximum at $P_t$, $v_x(P_t,t)=0$ and $v_{xx}(P_t,t)\leq0$. Then $v_t(P_t,t)\leq0$. Considering Lemma \ref{lem:max}, $m^{\prime}(t)\leq0$.

\textbf{Case 2.} $P_t\in \Delta_{\pm c}$. We only consider $P_t\in \Delta_{-c}$. Then in the $y$-coordinates of $\Delta_{-c}$, $u_j$ satisfy the full graph equation, which is (\ref{eq:graph}) for ``$-$'' and $n=1$. Therefore $v=u_1-u_2$ satisfies a parabolic equation 
$$
v_t=a_2(y,t)v_{yy}+b_2(y,t)v_{y}.
$$
Seeing $v_y(P_t,t)=0$ and $v_{yy}(P_t,t)\leq 0$, $m^{\prime}(t)\leq0$.

\textbf{Case 3.} $P_t\in \Sigma\setminus(\Delta_{\pm c}\cup B_d)$. Then we can choose coordinate $z$ on some neighbourhood of $P_t$ on $\Sigma$ and $u_j$ satisfy (\ref{eq:para}). We may write this equation as $u_t=F(z,t,u, u_z, u_{zz})$. Then $v=u_1-u_2$ satisfies
$$
v_t=a_3(z,t)v_{zz}+b_3(z,t)v_{z}+c_3(z,t)v,
$$
where
$$
c_3(z,t)=\int_{0}^1F_u(z,t,u^{\theta}, u_z^{\theta},u_{zz}^{\theta})d\theta,
$$
where $u^{\theta}=(1-\theta)u_2+\theta u_1$.

By the assumption, outside of the disks $\Delta_{\pm c}$ and the pipe $B_d$, $u_i$ are smooth up to $t=0$, so the coefficient $c(z,t)$ is bounded, $0<t<T$, saying by $|c(z,t)|\leq M<\infty$. The constant $M$ may depend on the choice of local coordinate $z$. Noting $\Sigma$ is compact, by easy covering argument, we can choose $M$ independent of the choice of local coordinate. Since $v_z(P_t,t)=0$, $v_{zz}(P_t,t)\leq0$,
$$
v_t(P_t,t)\leq c(P_t,t)v(P_t,t)\leq Mv(P_t,t).
$$
Consequently, $m^{\prime}(t)\leq Mm(t)$.

Combining these three cases, we have $m^{\prime}(t)\leq C m(t)$, for some constant $C>0$. Considering $m(0)=0$, $m(t)\leq 0$. Conversely, we can prove $M(t)=\min\{v(P,t)\mid P\in \Sigma\}\geq0$. Therefore $u_1\equiv u_2$. We complete the proof.
\end{proof}

\begin{lem}\label{lem:closeas}  Then there exist $E_j$ closed such that $E_j^{\circ}$ are $\alpha/2^j$-domains and $E_j\downarrow \overline{U}$. Where $U$ is given at the beginning of the section and $E^{\circ}$ denotes the interior of the set $E$.
\end{lem}
\begin{proof} Since $U\cap \{x\geq0\}$ is an $\alpha$-domain, there exists unique $\delta_0>0$ such that 
$$
u_0(\delta_0)=\alpha
$$ 
and
 $$
u_0^{\prime}(x)>0,\ 0<x<\delta_0.
$$
For all $j\geq1$, there exists unique $\delta_j$, $0<\delta_j<\delta_0$ such that 
$$
u_0(\delta_j)=\alpha/2^j.
$$
 We can contruct $v_j\in C^{\infty}((-b_0,b_0))$ being even such that 
$$
v_j(x)=\left\{
\begin{array}{lcl}
\alpha/2^j,\ x\in (-\delta_j/2,\delta_j/2),\\
u_0(x), x\in [-b_0,-\delta_j]\cup[\delta_j,b_0],
\end{array}
\right.
$$
$v_j(x)\geq u_0$, $x\in[-b_0,b_0]$ and $v^{\prime}_j(x)>0$, $x\in(\delta_j/2,\delta_j)$. It is easy to see $v_j\downarrow u_0$ uniformly in $[-b_0,b_0]$.

Let $E_j=\{(x,y)\mid|y|\leq v_j(x),\ -b_0\leq x\leq b_0\}$. Since $v_j\downarrow u_0$ uniformly in $[-b_0,b_0]$, $E_j\downarrow \overline{U}$. It is easy to check $E_j^{\circ}$ are $\alpha/2^j$-domain.
\end{proof}

\begin{lem}\label{lem:closebou}  Let the same assumption in Theorem \ref{thm:exist} be given. Then there exists $t_1>0$ such that, for all $t_2$ satisfying $0<t_2<t_1$, the second fundamental forms and derivatives of $\partial E_j(t)$ are uniformly bounded for $t_2\leq t\leq t_1$,  where $E_j(t)$ denote the closed evolution of $V=-\kappa+A$ with $E_j(0)=E_j$.
\end{lem}
\begin{proof} 
Let $E_j(t)=\{(x,y)\mid|y|\leq v_j(x,t)\}$.

{\bf Step 1.} For all $t_2$ satisfying $0<t_2<\delta$($\delta$ given by Theorem \ref{thm:exist}), there exists a constant $c>0$ such that 
$$
v_j(0,t)>c,\ t_2/2<t<\delta.
$$

Let $U^{+}(t)$ denote the open evolution with $U^+(0)=U\cap\{x\geq0\}$. Using Theorem \ref{thm:partialUmeancurvature}, $U^{+}(t)$ is the domain surrounded by $\Lambda(t)$($\Lambda(t)$ is defined in Section 1).
By (3) in Theorem \ref{thm:order} and $U\cap\{x\geq0\}\subset E_j$, there holds $U^{+}(t)\subset E_j(t)$. By our assumption that $a_*(t)<0$, for $0<t\leq\delta$, there holds $(0,0)\in U^{+}(t)\subset E_j(t)$, $0<t<\delta$. For all $t_2$ satisfying $0<t_2<\delta$, there exists $c>0$ such that $v_j(0,t)> c$, $t_2/2\leq t\leq \delta$.

{\bf Step 2.} Construction of four auxiliary balls.

Since $U\cap\{x\geq0\}$ is an $\alpha$-domain, there exist $\beta_2>\beta_1>0$ such that $u_0(\pm\beta_1)=u_0(\pm\beta_2)=\alpha$ and $u_0^{\prime}(x)<0$ for $x>\beta_2$, $u_0^{\prime}(x)>0$ for $0<x<\beta_1$. There exist $p>\beta_1$ and $0<q<\beta_2$ such that $\dis{u_0(\pm q)=u_0(\pm p)=\frac{\alpha}{2}}$. we consider the points
$$Q=(-p,0),\ \ \ \ \ \ \ \ \ \ \ \ \ \ \ \ \ \ P=(p,0),$$
$$Q^{\prime}=(-p,\alpha),\ \ \ \ \ \ \ \ \ \ \ \ \ \ \ \ \ P^{\prime}=(p,\alpha).$$

\begin{figure}[htbp]
	\begin{center}
            \includegraphics[height=8cm]{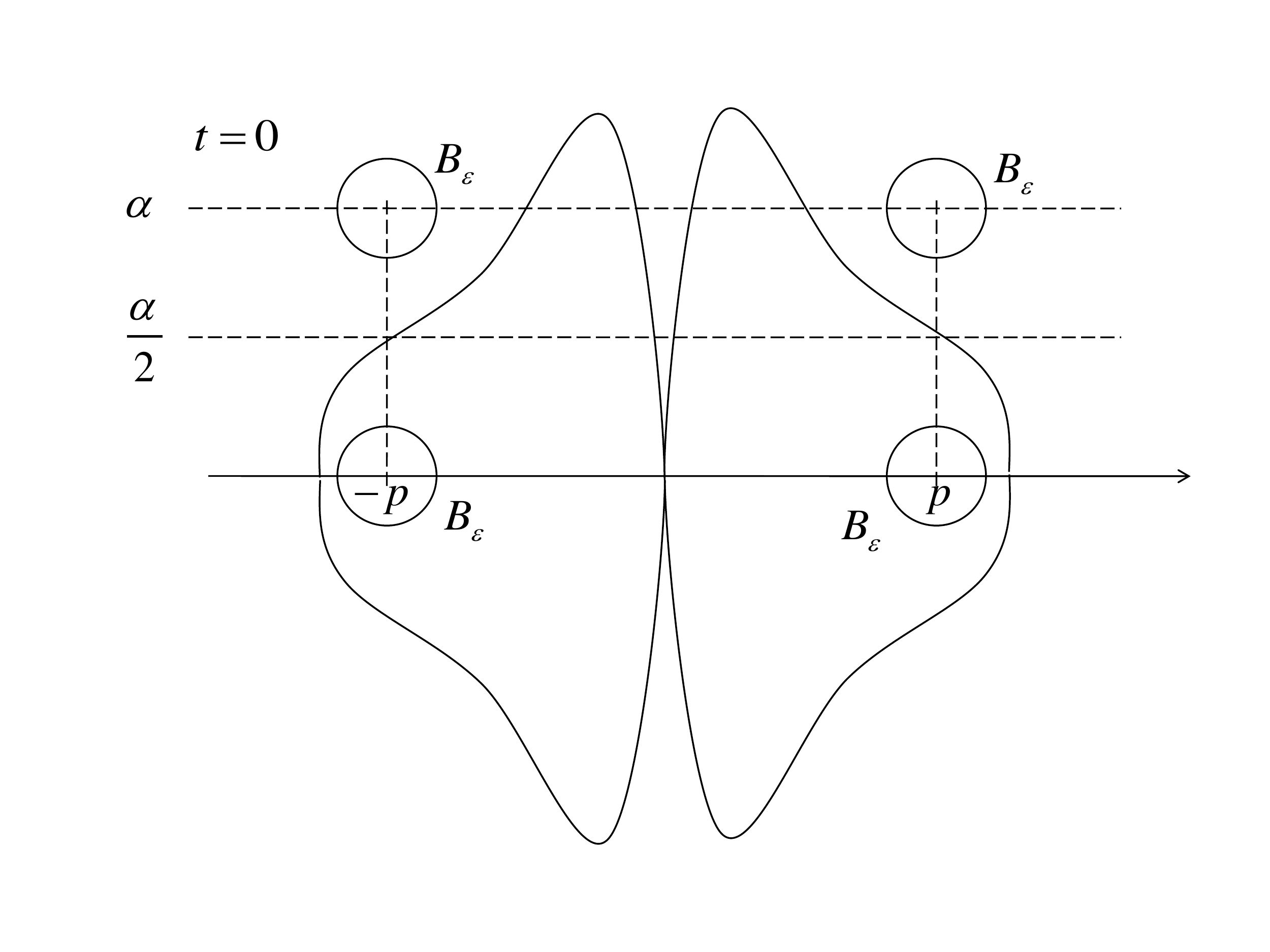}
		\vskip 0pt
		\caption{Proof of Lemma \ref{lem:closebou}}
        \label{fig:uniformb}
	\end{center}
\end{figure}

Since $P\in U$ and $P^{\prime}\in \overline{U}^{c}$, there exists $\epsilon$ such that $\overline{B_{\epsilon}(P)}\subset U$ and $\overline{B_{\epsilon}(P^{\prime})}\subset\overline{U}^c$. Consequently, $\overline{B_{\epsilon}(P)}\cup \overline{B_{\epsilon}(Q)}\subset E^{\circ}$ and $\overline{B_{\epsilon}(P^{\prime})}\cup \overline{B_{\epsilon}(Q^{\prime})}\subset E^{c}$. Then for $j$ large enough, $\overline{B_{\epsilon}(P)}\cup \overline{B_{\epsilon}(Q)}\subset E_j^{\circ}$ and $\overline{B_{\epsilon}(P^{\prime})}\cup \overline{B_{\epsilon}(Q^{\prime})}\subset E_j^{c}$. By (4) in Theorem \ref{thm:order},
\begin{equation}\label{eq:q1}
\overline{B_{\epsilon(t)}(P)}\cup \overline{B_{\epsilon(t)}(Q)}\subset E_j(t)^{\circ},
\end{equation}
$0<t<\delta_2$. By (1b) in Theorem \ref{thm:conti}, there exists $\delta_3>0$ such that
\begin{equation}\label{eq:q2}
\overline{B_{\epsilon(t)}(P^{\prime})}\cup \overline{B_{\epsilon(t)}(Q^{\prime})}\subset E_j(t)^c,
\end{equation}
for $0< t<\delta_3$. Where $\epsilon(t)$ is the solution of (\ref{eq:ball2}) with $\epsilon(0)=\epsilon$, $0<t<\delta_1$. Choose $\delta_2$ independent on $j$ such that $\epsilon(t)>\epsilon/2$, $0<t<\delta_2$. 
\begin{figure}[htbp]
	\begin{center}
            \includegraphics[height=7cm]{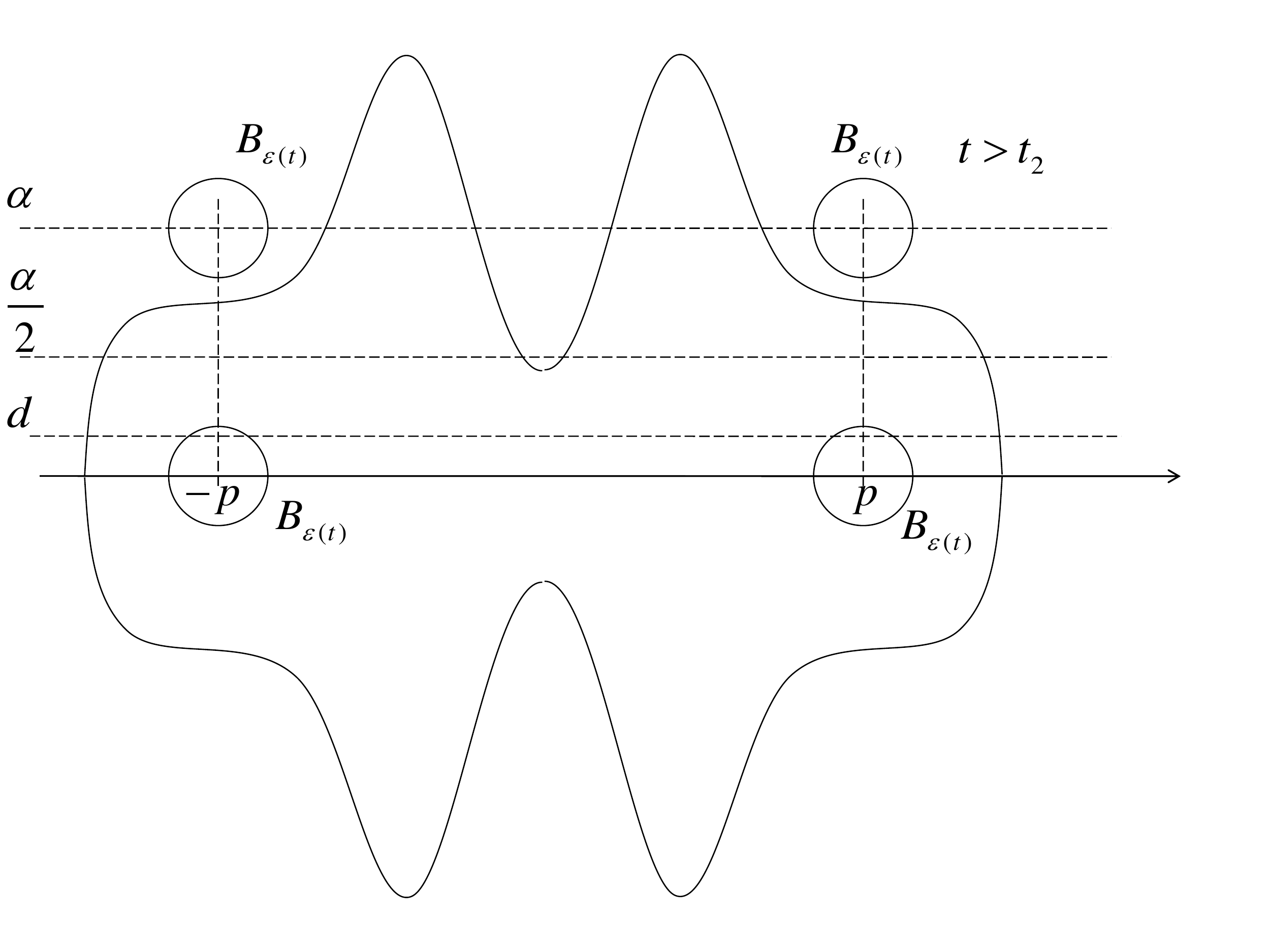}
		\vskip 0pt
		\caption{Proof of Lemma \ref{lem:closebou}}
        \label{fig:uniformb2}
	\end{center}
\end{figure}

{\bf Step 3.} Divide $\partial E_j(t)$ into two parts by auxiliary balls. 

Since for all $\rho<\alpha$, $C_{\rho}$ intersects $\partial E_j$ at most forth, by Proposition \ref{pro:intersection}, there exists $t_0>0$ such that $C_{\rho}$ intersects $\partial E_j(t)$ at most forth, $0<t<t_0$. By continuity, we can deduce that there exists $\delta_4$ such that for all $\rho<\alpha$, the equation $v_j(x,t)=\rho$ has most one root for $x>p$, for all $t<\delta_4$. By symmetry, it is so for $x<-p$. 

Choosing $t_1=\min\{t_0,\delta_2,\delta_3,\delta_4\}$, Step 1 implies that $E_j(t)^{\circ}$ are all $c$-domains, $t_2/2<t<t_1$.  Let $d<\min\{c,\epsilon/4\}$. By (\ref{eq:q1}) in Step 2, we have $v_j(x,t)>d$, for $t_2/2<t<t_1$, $|x-p|<\sqrt{\epsilon^2(t)-d^2}$ or $|x+p|<\sqrt{\epsilon^2(t)-d^2}$. Seeing $\epsilon(t)>\epsilon/2$, there holds
$$
v_j(x,t)\geq d,\ \text{in}\ \Omega=(-p-\frac{\sqrt{3}}{4}\epsilon,p+\frac{\sqrt{3}}{4}\epsilon)\times(t_2/2,t_1).  
$$

For $x\leq-p$, by (\ref{eq:q2}) in Step 2,
$$
v_j(x,t)<\alpha/2-\epsilon(t)<\alpha/2-\epsilon/2,\ x\leq-p,\ 0\leq t<t_1.
$$
This is also true for $x\geq p$.

{\bf Step 4.} The derivatives and second fundamental forms of $\partial E_j(t)$ are bounded in $\Omega^{\prime}=[-p,p]\times(t_2,t_1)$.

Since $v_j(x,t)\geq d$ in $\Omega=(-p-\frac{\sqrt{3}}{4}\epsilon,p+\frac{\sqrt{3}}{4}\epsilon)\times(t_2/2,t_1)$, Theorem \ref{thm:grad} implies that $v_{jx}$ are uniformly bounded in $\Omega$. By Remark \ref{rem:hes}, $v_{jxx}$ are uniformly bounded in $\Omega^{\prime}$. 

{\bf Step 5.} The derivatives and second fundamental forms of $\partial E_j(t)$ are bounded for $x\leq -p$ and $x\geq p$, $t_2<t<t_1$.

We only consider for $x\leq-p$. For $0<t<t_1$ the part of $\partial E_j(t)$ on $x\leq-p$ can be represent by $x=w_j(y,t)$, for $|y|<\alpha/2$, $t\in (0,t_1)$. And $w_j$ satisfy the equation (\ref{eq:graph}) in the condition ``$-$'' and $n=1$. Then Corollary \ref{cor:es} and Remark \ref{rem:hes} imply that all $\frac{\partial ^k}{\partial y^k}w_j(y,t)$, $k=1,2$, are uniformly bounded for $|y|\leq\alpha/2-\epsilon/2$, $t_2<t<t_1$ for any $t_2>0$. Then the derivatives and second fundamental forms of $\partial E_j(t)$ are uniformly bounded when $x\leq-p$, $t_2<t<t_1$. 

The proof of this lemma is completed.
\end{proof}

\begin{lem}\label{lem:opensy} There exist $U_j$ being open and $U_j\cap\{x\geq0\}$ being an $\alpha$-domain such that $U_j\uparrow U$.
\end{lem}

\begin{proof}
Since $U\cap\{x>0\}$ being $\alpha$-domain, for $j\geq1$, there exist $\delta_j$ satisfying $0<\delta_j<\delta_0$ such that $u_0(\delta_j)=\alpha/2^j$, where $\delta_0$ satisfies $u_0(\pm \delta_0)=\alpha$ and $u_0^{\prime}(x)>0$ for $0<x<\delta_0$. We set $u_j\in C^{\infty}((-b_0,b_0))$ and even satisfying
$$
u_j(x)=\left\{
\begin{array}{lcl}
0,\ x=0,\\
u_0(x),\ x\in[-b_0,-\delta_j]\cup[\delta_j,b_0],
\end{array}
\right.
$$
and $u_j(x)\leq u_0$ for $x\in[-b_0,b_0]$, $u_j^{\prime}(x)>0$ for $x\in(0,\delta_j)$.

Let $U_j=\{(x,y)\mid |y|<u_j(x)\}$. Obviously $u_j\uparrow u_0$, then $U_j\uparrow U$. It is easy to check $U_j\cap\{x>0\}$ are $\alpha$-domain.
\end{proof}
\begin{lem}\label{lem:openbou} Let the same assumption in Theorem \ref{thm:exist} be given. Then there exists $t_1>0$ such that for all $t_2$ satisfying $0<t_2<t_1$, the second fundamental forms and derivatives of $\partial U_j(t)$ is uniform bounded, $t_2<t<t_1$, where $U_j(t)$ is the open evolution of $V=-\kappa+A$ with $U_j(0)=U_j$.
\end{lem}
\begin{proof} Let $U(t)$ and $U^{+}(t)$ be the open evolution with $U(0)=U$ and $U^{+}(0)=U\cap\{x>0\}$. Seeing appendix, $\Lambda(t)=\partial U^{+}(t)$($\Lambda(t)$ is given in Section 1). Since $a_*(t)<0$, for $0<t<\delta$, $(0,0)\in\partial U^{+}(t)$. 

By (1) in Theorem \ref{thm:order} and $U\cap\{x>0\}\subset U$, we have $U^{+}(t)\subset U(t)$. Consequently, $(0,0)\in U(t)$, $0<t<\delta$. Then for $j$ large enough, $(0,0)\in U_j(t)$. The following parts can be proved similar as in Lemma \ref{lem:closebou}.
\end{proof}

\begin{proof}[Proof of Theorem \ref{thm:exist}] Seeing Lemma \ref{lem:closebou} and \ref{lem:openbou},  $\partial U(t)$, $\partial E(t)$ are smooth curves and homeomorphic to the curve $\Sigma$ given by Proposition \ref{pro:sigma2}. Consequently, $\partial U(t)$, $\partial E(t)$ satisfy the assumption of Proposition \ref{pro:uniq2}, $0\leq t<T_1$, for some $T_1$ satisfying $0<T_1<t_1$. Where $t_1$ is given by Lemma \ref{lem:closebou} and \ref{lem:openbou}. Then there holds $\partial U(t)=\partial E(t)$, $0<t<T_1$. If we let $\Gamma(t)=\partial E(t)\cap\{x\geq0\}$, $\Gamma(t)$ will be the unique solution of (\ref{eq:cur}), (\ref{eq:Neum1}) and (\ref{eq:initial1}). The proof of Theorem \ref{thm:exist} is completed.
\end{proof}
\begin{rem}
Indeed, $\partial E(t)$ and $\partial U(t)$ are smooth on $\Sigma\setminus B_d\times[0,T_1)$ and $\Sigma\times(0,T_1)$. If we remove the assumption that $\Gamma_0$ is smooth at end point $(-b_0,0)$ and $(b_0,0)$, $\partial E(t)$ and $\partial U(t)$ will be smooth on $\Sigma\setminus( \Delta_{\pm c}\cup B_d)\times[0,T_1)$ and $\Sigma\times(0,T_1)$. Therefore the result is also true even if removing smoothness at end points.
\end{rem}
\begin{proof}[Proof of Theorem \ref{thm:fattening1}] It is sufficient to show that there is a ball $B$ such that $B\subset E(t)\setminus U(t)$, for some $t$. 

{\bf Closed evolution $E(t)$.} Since $E_j^{\circ}$(given by Lemma \ref{lem:closeas}) are $\alpha/2^j$-domain with smooth boundary, by Lemma \ref{lem:alphad2}, there exists a positive time $t_1$, $t_1<\delta$($\delta$ is given in Theorem \ref{thm:fattening1}) such that $E_j(t)^{\circ}$ are $(At+\alpha/2^j)$-domain for $0<t<t_1$. Combining $E_j(t)\downarrow E(t)$, we have $E(t)^{\circ}$ is an $At$-domain, $0<t<t_1$. Therefore $E(t)^{\circ}$ is an $At_1/2$-domain, $t_1/2<t<t_1$.

{\bf Open evolution $U(t)$.} Denote $U^{\pm}(t)$ being the open evolutions with $U^{\pm}(0)=U\cap\{\pm x\geq 0\}$. Seeing appendix $\partial U^{+}(t)=\Lambda(t)$, $\partial U^{-}(t)=\{(-x,y)\mid(x,y)\in \Lambda(t)\}$, where $\Lambda(t)$ is given in Section 1. Thus the left end point of $U^+(t)$ and the right end point of $U^-(t)$ are ($a_*(t)$,0) and ($-a_*(t)$,0), respectively. By the assumption in this theorem $a_*(t)\geq 0$, $0\leq t<\delta$, it means that $-a_*(t)\leq a_*(t)$, $0\leq t<\delta$. Therefore, $U^+(t)\cap U^-(t)=\emptyset$, $0\leq t<\delta$. From Lemma \ref{lem:sep}, the inner evolution $U(t)$ satisfies $U(t)=U^{+}(t)\cup U^{-}(t)$, for $0\leq t<\delta$.

By (2a) in Theorem \ref{thm:conti}(the boundary of open evolution evolves continuously) and $a(t)\geq0$, there exists $\dis{\delta_1<\frac{At_1}{4}}$ such that 
$$
\dis{B_{\delta_1}((0,\frac{At_1}{4}))}\cap U(t)=\emptyset,\ \dis{\frac{t_1}{2}<t<t_1}
$$
and
$$
B_{\delta_1}((0,\frac{At_1}{4}))\subset E(t),\ \dis{\frac{t_1}{2}<t<t_1}.
$$
 Where $\dis{B_{\delta_1}((0,\frac{At_1}{4}))}$ is a ball centered at $(0,At_1/4)$ with radius $\delta_1$. Then $\dis{B_{\delta_1}((0,\frac{At_1}{4}))}\subset \Gamma(t)=E(t)\setminus U(t)$, for $\dis{\frac{t_1}{2}<t<t_1}$.
\end{proof}

\section{Formation of sigularity}
In this section, we want to identify the singular formation of $\Gamma(t)$, when $\Gamma(t)$ becomes singular at $t=T<\infty$. 
Where
\begin{equation}\label{eq:singularT}
T=\sup\{t>0\mid \Gamma(s)\ \text{are}\ \text{smooth},0<s<t\}
\end{equation}
and $\Gamma(t)$ is given by Theorem \ref{thm:exist}. For convenience, we still consider $\Gamma(t)$ extended evenly. By Theorem \ref{thm:openevolutionmeancurvature} and Theorem \ref{thm:gu}, $\Gamma(t)=\{(x,y)\in\mathbb{R}^2\mid|y|=u(x,t),-b(t)\leq x\leq b(t)\}$ and $(u,b)$ is the solution of the following free boundary problem
$$
\left\{
\begin{array}{lcl}
\dis{u_t=\frac{u_{xx}}{1+u_x^2}+A\sqrt{1+u_x^2}}, \ x\in(-b(t),b(t)), \ 0<t< T,\\
u(-b(t),t)=0$, $u(b(t),t)=0$,\ $0\leq t< T,\\
u_x(-b(t),t)=\infty,\ u_x(b(t),t)=-\infty,\ 0\leq t< T,\\
u(x,0)=u_0(x),\ -b_0\leq x\leq b_0.
\end{array}
\right.
$$

Noting the choice of initial curve, $\Gamma(t)$ is not convex for $t$ near $0$. Therefore, $\Gamma(t)$ possible intersects itself at $y$-axis. Therefore, it is necessary to study the local minima of $u(\cdot,t)$.

As showed in section 4, the numbers of local maxima and local minima are a finite nonincreasing function of time. It follows that, after a while, the numbers of local maxima and local minima are constants. After discarding an initial section of the solution, we may even assume that $x\mapsto u(x,t)$ has $m$ local minima and $m+1$ local maxima. Let these minima and maxima be located at $\{\xi_j(t)\}_{1\leq j\leq m}$ and $\{\eta_j(t)\}_{0\leq j\leq m}$, respectively. And order the $\xi_j(t)$ and $\eta_j(t)$ so that
\begin{equation}\label{eq:ordermami}
-b(t)<\eta_0(t)<\xi_1(t)<\eta_1(t)<\cdots<\xi_m(t)<\eta_m(t)<b(t).
\end{equation}
Since the number of critical points of $u(\cdot,t)$ drops whenever $u(\cdot,t)$ has degenerate critical point, the minima and maxima of $u(\cdot,t)$ are all nondegenerate. By the implicit function theorem the $\xi_j(t)$ and $\eta_j(t)$ are therefore smooth functions of time.

\begin{lem}\label{lem:convergeendmimax}
The limits
$$
\lim\limits_{t\rightarrow T}b(t)=b(T)
$$
and
$$
\lim\limits_{t\rightarrow T}\xi_j(t)=\xi_j(T),\ \lim\limits_{t\rightarrow T}\eta_j(t)=\eta_j(T)
$$
exist.
\end{lem}

\begin{proof}
We prove this lemma by the method from \cite{AAG}, first developed by \cite{CM}. But in our proof, there is a little  difference, since the intersection number between two flows evolving by $V=-\kappa+A$ may increase. Therefore, the method in \cite{AAG} should be modified.

First, we prove $\lim\limits_{t\rightarrow T}b(t)$ exists. By the vertical equation
$$
w_t=\frac{w_{rr}}{1+w_r^2}+A\sqrt{1+w_r^2},
$$
we can derive $b^{\prime}(t)=w_{rr}+A\leq A$ because of $w_{rr}(0)\leq0$. Then $b(t)-At$ is non-increasing. It is easy to see $b(t)-At$ is bounded for $t<T$. Therefore $\lim\limits_{t\rightarrow T}(b(t)-At)$ exists. Consequently,  $\lim\limits_{t\rightarrow T}b(t)$ exists.

Next, we prove $\lim\limits_{t\rightarrow T}\xi_j(t)$ exists. We assume
$$\limsup\limits_{t\rightarrow T}\xi_j(t)>\liminf\limits_{t\rightarrow T}\xi_j(t).$$ We can choose $x_0\in(\liminf\limits_{t\rightarrow T}\xi_j(t),\limsup\limits_{t\rightarrow T}\xi_j(t))$ and $x_0\neq 0$. Without loss of generality, we assume $-b(T)<x_0<0<b(T)$. Since $\xi_j(t)$ is continuous in $t$, there exists a sequence $t_m\rightarrow T$ such that
$$
\xi_j(t_m)=x_0 \ \textrm{and}\ u_x(x_0,t_m)=0.
$$
We let $\widetilde{\Gamma}(t)$ be the reflection from $\Gamma(t)$ about $x=x_0$. Consequently, $\widetilde{a}(t):=2x_0-a(t)$ and $\widetilde{b}(t):=2x_0-b(t)$ are the end points of $\widetilde{\Gamma}(t)$. Obviously, $\widetilde{\Gamma}(t)$ evolves by $V=-\kappa+A$ and $\widetilde{a}(T)<-b(T)<x_0<\widetilde{b}(T)<b(T)$. For $t$ being sufficiently close to $T$, $\widetilde{a}(t)<a(t)<x_0<\widetilde{b}(t)<b(t)$, i.e., the order of $\widetilde{a}(t)$, $\widetilde{b}(t)$, $a(t)$, $b(t)$ dose not change. Using Theorem \ref{thm:sl}, since $\widetilde{\Gamma}(t_m)$ intersects $\Gamma(t_m)$ at $x_0$ tangentially, the intersection number between $\widetilde{\Gamma}(t)$ and $\Gamma(t)$ will drop infinite times, for $t$ close to $T$. But Theorem \ref{thm:sl} shows that the intersection number between $\Gamma(t)$ and $\widetilde{\Gamma}(t)$ is finite(The choice of $x_0$ implies $\Gamma(t)$ is not identity to $\widetilde{\Gamma}(t)$). This yields a contradiction.
\end{proof}

\begin{lem}\label{lem:singlepoint pinch}
If $\xi_{j}(T)<\eta_j(T)$, then for any compact interval $[c,d]\subset(\xi_j(T),\eta_j(T))$, there exists $t_1$ and $\delta>0$ such that $u(x,t)\geq\delta$ for $x\in[c,d]$, $t\in[t_1,T)$.
(Similarly for $\eta_{j-1}(T)<\xi_j(T)$, $-b(T)<\eta_0(T)$, $\eta_m(T)<b(T)$).
\end{lem}

\begin{proof}
Let $[a,b]\subset(\xi_j(T),\eta_j(T))$ be any compact interval, then there exists $t_1<T$ such that $[a,b]\subset(\xi_{j}(t),\eta_j(t))$ and $u_x(x,t)>0$, $x\in[a,b]$, $t\in(t_1,T)$. Letting $\theta=\arctan u_x$, $\theta$ satisfies
$$
\theta_t=\cos^2\theta\theta_{xx}+A\sin\theta\theta_x.
$$
Since $u_x>0$, $x\in[a,b]$, $t\in(t_1,T)$, there holds
$$
\theta_t-\cos^2\theta\theta_{xx}-A\sin\theta\theta_x=0.
$$

On the other hand, we let $\varphi(x,t)=\epsilon e^{-ct}\sin(\lambda(x-a))$, where $\lambda=\pi/(b-a)$, $c>A\lambda\pi+\lambda^2$, $0<\epsilon<\pi$. Since $\varphi_{xx}\leq0$, $x\in[a,b]$ and
seeing
$$
\left|-A\lambda\frac{\sin(\epsilon e^{-ct}\sin(\lambda(x-a)))}{\sin(\lambda(x-a))}\cos(\lambda(x-a))\right|\leq A\lambda\pi,
$$
there holds
\begin{eqnarray*}
\varphi_t&-&\cos^2\varphi\varphi_{xx}-A\sin\varphi\varphi_x\leq\varphi_t-\varphi_{xx}-A\sin\varphi\varphi_x\\
&=&
\epsilon e^{-ct}\sin(\lambda(x-a))\left(-c+\lambda^2-A\lambda\frac{\sin(\epsilon e^{-ct}\sin(\lambda(x-a)))}{\sin(\lambda(x-a))}\cos(\lambda(x-a))\right)\\&\leq&\epsilon e^{-ct}\sin(\lambda(x-a))(-A\lambda\pi-\lambda^2+\lambda^2+A\lambda\pi)=0,
\end{eqnarray*}
for $x\in[a,b],\ t\in(t_1,T).$
Since $u_x(x,t_1)$ is bounded from below for some positive constant in $[a,b]$, we can choose $\epsilon>0$ small enough such that $\varphi(x,t_1)\leq\theta(x,t_1)$. Seeing 
$$
\varphi(a,t)=0<\theta(a,t),\ \varphi(b,t)=0<\theta(b,t),\ t\in(t_1,T).
$$
By maximum principle, 
$$
\theta(x,t)\geq\varphi(x,t),\ a<x<b,\ t_1<t<T.
$$
Consequently,
$$
u_x\geq\arctan u_x=\theta\geq\epsilon e^{-ct}\sin(\lambda(x-a)),\ x\in[a,b],\ t\in(t_1,T).
$$
$$
u\geq\epsilon\frac{e^{-ct}}{\lambda}(1-\cos(\lambda(x-a))),\ x\in[a,b],\ t\in(t_1,T).
$$
Then for all $[c,d]\subset(a,b)$, $u$ is uniformly bounded from below for $x\in[c,d],\ t\in[t_1,T)$.
\end{proof}

\begin{lem}\label{lem:limitsurface}
$\lim\limits_{t\rightarrow T}u(x,t)=u(x,T)$ exists, and $u(x,t)$ converges uniformly to $u(x,T)$, for $x\in \mathbb{R}$, as $t\rightarrow T$. The function $u$ is smooth at $(x,t)\in \mathbb{R}\times(0,T]$ provided that $u(x,t)>0$. We interpret that $u(x,t)=0$ outside $(a(t),b(t))$.
\end{lem}
\begin{proof}
By Lemma \ref{lem:singlepoint pinch}, for all $[c,d]\subset(\xi_{j-1}(T),\xi_{j}(T))$, $u(x,t)\geq \delta$, $x\in[c,d]$, $t\in[t_1,T)$. By Theorem \ref{thm:grad}, $u_x$ is uniformly bounded on $[c,d]\times[t_1,T)$, which implies $\frac{\partial^i}{\partial x^i}u(x,t)$, $i=1,2$ are bounded on any compact subinterval of $(c,d)$. On the other hand, from equation, $u_t(x,t)$ is uniformly bounded on such interval, so that $u(\cdot,t)$ converges uniformly on any such interval.

The same idea can be applied to the conditions in the intervals $(-b(T),\xi_1(T))$ and $(\xi_m(T),b(T))$. Siince outside of $[-b(T),b(T)]$,$u(x,T)$ is considered to be $0$, the result is true.

Except at $-b(T)$, $b(T)$ and $\xi_j(T)$s, $u(x,t)$ converges pointwise for every $x$ not equaling $-b(T)$, $b(T)$, $\xi_j(T)$, as $t\rightarrow T$. The convergence is uniform on any interval that does not contain any of the points.

Next we want to prove the functions $u(\cdot,t)$ are equicontinuous for $T/2<t<T$.

Assuming $x_1<x_2$, if $x_1$, $x_2$ are both not in the interval $(-b(T),b(T))$, the conclusion is obvious. Assume $x_1\in(-b(T),b(T))$.

Suppose that $|u(x_1,t)-u(x_2,t)|\geq\epsilon$. Then either $u(x_1,t)\geq\epsilon/2$ or $u(x_2,t)\geq\epsilon/2$ or both; we assume the first one. From Theorem \ref{thm:grad}, $|u_x|<\sigma(\epsilon/2,T/2)$ whenever $u(x,t)\geq\epsilon/2$, $T/2<t<T$. Thus, if $u(x,t)\geq\epsilon/2$ on $(x_1,x_2)$, 
$$
x_2-x_1\geq\frac{|u(x_1,t)-u(x_2,t)|}{\sigma(\epsilon/2,T/2)}\geq\frac{\epsilon}{\sigma(\epsilon/2,T/2)}.
$$
If $u(x,t)<\epsilon/2$ some where in the interval $(x_1,x_2)$, then there is a smallest $x_3$ satisfying $x_1<x_3$ at which $u(x_3,t)=\epsilon/2$. On the interval $(x_1,x_3)$, $u(x,t)\geq\epsilon/2$. Then
$$
x_2-x_1\geq x_3-x_1\geq\frac{u(x_1,t)-u(x_3,t)}{\sigma(\epsilon/2,T/2)}\geq\frac{\epsilon}{2\sigma(\epsilon/2,T/2)}.
$$
So for every $\epsilon>0$, choose $\delta=\epsilon/(2\sigma(\epsilon/2,T/2))$ so that  
$$
|u(x_1,t)-u(x_2,t)|<\epsilon,
$$
$|x_1-x_2|<\delta$, for $T/2<t<T$.

Thus $u(x,t)$ is equicontinuous. Noting that $u(x,t)$ converges to $u(x,T)$ in $\mathbb{R}\setminus\{\xi_i(T),-b(T),b(T)\}$ and $\mathbb{R}\setminus\{\xi_i(T),-b(T),b(T)\}$ is dense in $\mathbb{R}$, the proof is completed.
\end{proof}

\begin{lem}\label{lem:seprate}
Suppose that $u(\eta_0(T),T)>0$, then $-b(T)<\eta_0(T)$.
\end{lem}

\begin{proof}
Since $u(\eta_0(T),T)>0$, there exists $\delta>0$ such that $\delta=\inf\limits_{0\leq t\leq T} u(\eta_0(t),t)$. We consider
$$
x=v(|y|,t),
$$
being the inverse function of $|y|=u(x,t)$ for $x\in(a(t),\eta_0(t))$ and let $w(y,t)=v(|y|,t)$. $w(y,t)$ satisfies the equation (\ref{eq:graph}) for the condition "$-$" and ``$n=1$'', $|y|<\delta$, $0<t<T$. Clearly $w$ is uniformly bounded, so Corollary \ref{cor:es} and Remark \ref{rem:hes} imply that $\frac{\partial^k w}{\partial y^k}(y,t)$, $k=1,2$ are bounded for $|y|\leq\delta/2$, $T/2\leq t<T$. So the limit function $w(y,T)$ obtained by Lemma \ref{lem:limitsurface} is smooth for $|y|\leq\delta/2$.

As the proof of Lemma \ref{lem:sing1}, using maximum principle, $v_r(r,T)>0$, $0<t<\delta/2$. Then $-b(T)=v(0,T)<v(\delta/2,T)<\eta_0(T)$.
\end{proof}

Lemma \ref{lem:singlepoint pinch} and Lemma \ref{lem:seprate} imply that ``width'' and ``height'' become zero at same time. Therefore, $\Gamma(t)$ can not pinch at $y$-axis before shrinking. We prove the detail in the following theorem and corollary. 

\begin{thm}\label{thm:formationofsingular}(Formation of singular)

1. If $m=0$, $u(\eta_0(T),T)=0$ and $b(T)=0$. This implies that $\Gamma(t)$ shrinks to the origin $O$, as $t\rightarrow T$.

2. If $m\geq1$, there is $j$ such that $u(\xi_j(T),T)=0$, $1\leq j\leq m$.
\end{thm}
\begin{proof}
1. First, we prove for $m=0$, i.e., $u(x,t)$ only has one maximum without local minimum. We prove this by contradiction.

{\bf Case 1.} If $u(\eta_0(T),T)>0$, from Lemma \ref{lem:seprate}, $-b(T)<\eta_0(T)<b(T)$. $\Gamma(t)$ can be divided into three parts $\Delta_1(t)$, $\Delta_2(t)$ and $\Delta_3(t)$, for $t$ being very close to $T$, where $\Delta_1(t)$ and $\Delta_2(t)$ are the left and right caps of $\Gamma(t)$, $\Delta_3(t)$ is the middle part of $\Gamma(t)$ away form $x$-axis.
It is easy to show the derivatives and second fundamental formations of $\Delta_1$, $\Delta_2$ and $\Delta_3$ are uniformly smooth for $t\rightarrow T$(We can similarly prove as Lemma \ref{lem:sing1}), which contradicts to $\Gamma(t)$ becoming singular at $T$.

{\bf Case 2.} If $b(T)> 0$, there holds $-b(T)<\eta_0(T)$ or $\eta_0(T)<b(T)$, assuming $-b(T)<\eta_0(T)$. By Lemma \ref{lem:singlepoint pinch}, for every $[c,d]\subset(-b(T),\eta_0(T))$, $u(x,t)\geq \delta>0$ in $[c,d]\times[t_1,T)$. Then $u(\eta_0(t),t)\geq\delta$, $t_1\leq t<T$. Consequently, $u(\eta_0(T),T)\geq\delta$. By the same argument as in Case 1, we get a contradiction. Here we complete the proof under the condition $m=0$.

2. For $m\geq1$, if $u(\xi_j(T),T)>0$, for any $1\leq j\leq m$, we can divide $\Gamma(t)$ into three parts as above for $t$ being  close to $T$. Then we can get contradiction similarly as in the condition $m=0$. So there is $j$ such that $u(\xi_j(T),T)=0$.
\end{proof}

\begin{cor}\label{cor:2dsingular}
There is $t_1$ satisfying $0<t_1<T$ such that $u(x,t)$ loses all its local minima for $t\in[t_1,T)$. Moreover, $\Gamma(t)$ shrinks to a point, as $t\rightarrow T$.
\end{cor}
\begin{proof} Denote $h(t)=\max\limits_{a(t)<x<b(t)}u(x,t)$. By Lemma \ref{lem:in}, we can deduce that, for $t$ satisfying $t_2<t<T$ given, when $\rho<\min\{At_2,h(t)\}$, $y=\rho$ intersects $y=u(x,t)$ only twice.

If $u(x,t)$ does not lose its all local minima, the number of minima will not change denoted by $m\geq 1$. From Theorem \ref{thm:formationofsingular}, there exists $j$, $1\leq j\leq m$ such that $u(\xi_j(T),T)=0$. So we can choose $t_0$ satisfying $t_2<t_0<T$, there exists $\xi_j(t_0)$ such that $u(\xi_j(t_0),t_0)<At_2$. Obviously, $u(\xi_j(t_0),t_0)< h(t_0)$, then $u(\xi_j(t_0),t_0)<\min\{At_2,h(t_0)\}$. Consequently, $y=\rho=u(\xi_j(t_0),t_0)$ intersects $y=u(x,t_0)$ three times. Contradiction.

Therefore, there is $t_1$ such that $u(x,t)$ will lose its all local minima for $t\in[t_1,T)$. Seeing the proof in Theorem \ref{thm:formationofsingular} for $m=0$, $u(\eta_0(T),T)=0$ and $b(T)=0$. It means that $\Gamma(t)$ shrinks to a point, as $t\rightarrow T$.
\end{proof}

\begin{rem}
We note that all the proof in this section, we do not use the condition that $u(\cdot,t)$ is even. Therefore, the argument in this section can be used in any $x$-axisymmetric curve.
\end{rem}
\section{Asymptotic behaviors }
In this section, we will prove Theorem \ref{thm:threecondition} and Theorem \ref{thm:asym}. For convenience, we still extend $\Gamma(t)$ by even, still denoted by $\Gamma(t)$ and let
$$
h(t)=\max\limits_{-b(t)\leq x\leq b(t)}u(x,t),\ \ l(t)=2b(t).
$$
Denote $U(t)$ being the open set surrounded by $\Gamma(t)$. 

All the proofs in this section are proved by intersection number principle introduced in section 4. For the proof of asymptotic behavior, there have so far been many methods. The intersection argument in proving asymptotic behavior is developed by Professor Hiroshi Matano. Saying roughly, if two functions $u(x,t)$ and $v(x)$ are satisfying the same parabolic equation, moreover, $u(x,t)$ intersects $v(x)$ at some fixed point tangentially, for any large $t$. Then there holds $u(x,t)\equiv v(x)$. Another important method in studying asymptotic behavior is by using Lyapunov function to prove $u(x,t)$ is independent on $t$. By the intersection argument, we can prove $u(x,t)$ is independent on $t$ without Lyapunov function.

The following lemma says that $l(t_0)$ being large enough deduce $h(t_0)$ being large. We prove it by using Proposition \ref{pro:intersection}. Although the proof of  Lemma \ref{lem:expanding} is similar as in \cite{GMSW}, for the reader's convenience, we still give the proof for detail.
\begin{lem}\label{lem:expanding}
For any $\tau\in(0,T)$ and $M\in(0,A\tau/2)$, there exists $l_{M,\tau}>0$ such that, when $l(t_0)>l_{M,\tau}$ for some $t_0\in[\tau,T)$, it holds $h(t_0)>M$.
\end{lem}

\begin{proof}
For given $\tau\in(0,T)$ and $M\in(0,A\tau/2)$, we choose $R_0$ such that
$$
R_0\geq\frac{2}{A}.
$$
Let $R(t)$ be the solution of (\ref{eq:ball2}) with $R(0)=R_0$. Since $R_0>1/A$, $R(t)$ is increased in $t$. Therefore $R^{\prime}(t)\geq A-1/R_0\geq A/2$. Integrating the inequality, there holds
$$
R(\tau)\geq R_0+A\tau/2\geq R_0+M.
$$
So there exists $\tau_1\in(0,\tau]$ such that
$$
R(\tau_1)=R_0+M.
$$

Now we let
$$
W(x,t):=\sqrt{R(t)^2-x^2}-R_0,\ \ x\in[\sigma_-(t),\sigma_+(t)],\ t\in(0,\tau_1],
$$
where $\sigma_{-}(t)=-\sqrt{R(t)^2-R^2_0}$ and $\sigma_+(t)=\sqrt{R(t)^2-R^2_0}$. And we denote
$$
\theta_{\pm}(t)= \arctan\frac{\sqrt{R(t)^2-R^2_0}}{R_0}.
$$
Obviously, $\pi/2>\theta_{\pm}(t)>0$.

We choose $l_{M,\tau}:=\sigma_+(\tau_1)-\sigma_-(\tau_1)=2\sqrt{R(\tau_1)^2-R^2_0}=2\sqrt{M^2+2R_0M}.$ We let $\gamma_1(t)$ and $\gamma_2(t)$ be the extension of $u(x,t)$ and $W(x,t)$ as in Proposition \ref{pro:intersection}. Obviously, $(W(x,t),\sigma_{\pm}(t))$ is the solution of (Q) with $\theta_{\pm}(t)$(Proposition \ref{pro:intersection}), so by Proposition \ref{pro:intersection}, we can deduce
$$
\mathcal{Z}(\gamma_1(t_0),\gamma_2(\tau_1))\leq \mathcal{Z}(\gamma_1(t_0-s),\gamma_2(\tau_1-s)),\ \textrm{for}\ s\in[0,\tau_1).
$$
Since the extended curve $\gamma_2(\tau_1-s)$ converges to the $x$-axis, as $s\rightarrow\tau_1$, the right-hand side of the above inequality equals 2 for $s$ sufficiently close to $\tau_1$. Consequently,
$$
Z[\gamma_1(t_0),\gamma_2(\tau_1)]\leq2.
$$

Assuming $l(t_0)>l_{M,\tau}$, for some $t_0\in[\tau,T)$, then $\sigma_{\pm}(\tau_1)$ satisfy
$$
-b(t_0)<\sigma_1(\tau_-)<\sigma_+(\tau_1)<b(t_0).
$$
Hence $\gamma_1(t_0)$ intersects $\gamma_2(\tau_1)$ twice below the $x$-axis. So $u(x,t_0)>W(x,\tau_1)$ on the interval $[\sigma_-(\tau_1),\sigma_+(\tau_1)]$. Consequently, $h(t_0)>M$.
\end{proof}

The following corollary gives that as long as $l(t)$ is unbounded, $\Gamma(t)$ will be expanding. 

\begin{cor}\label{cor:expanding2}
Assume $T=\infty$ and there exists a sequence $s_m\rightarrow\infty$ such that $l(s_m)\rightarrow\infty$, as $m\rightarrow\infty$. Then
$l(t)\rightarrow\infty$ and $h(t)\rightarrow\infty$, as $t\rightarrow\infty$.
\end{cor}
\begin{proof}
We can use the same argument as in Lemma \ref{lem:expanding}, there exist $C>1/A$ and $m_0$ such that $u(x,s_{m_0})>\sqrt{(C+R_0)^2-x^2}-R_0$. Obviously, $\sqrt{(C+R_0)^2-x^2}-R_0>\sqrt{C^2-x^2}$, $-C\leq x\leq C$. Therefore $u(x,s_{m_0})\geq \sqrt{C^2-x^2}$, $-C\leq x\leq C$. By (b) in Remark \ref{rem:intersection1}, $u(x,s_{m_0}+t)\geq\sqrt{C(t)^2-x^2}$, $-C(t)\leq x\leq C(t)$, where $C(t)$ is the solution of (\ref{eq:ball2}) with $C(0)=C$. Seeing the choice of $C$, we can deduce $C(t)\rightarrow\infty$, as $t\rightarrow\infty$. Then $h(t+s_{m_0})>C(t)\rightarrow\infty$ and $l(t+s_{m_0})>2C(t)\rightarrow\infty$, as $t\rightarrow\infty$.
\end{proof}

The following lemma gives that as long as $h(t)$ is unbounded, $\Gamma(t)$ will be expanding. 

\begin{lem}\label{lem:1expanding}
Assume $T=\infty$ and there exists a sequence $s_m\rightarrow\infty$ such that $h(s_m)\rightarrow\infty$, as $m\rightarrow\infty$. Then
$l(t)\rightarrow\infty$ and $h(t)\rightarrow\infty$, as $t\rightarrow\infty$.
\end{lem}
\begin{proof}
If $l(t)$ is unbounded, by Corollary \ref{cor:expanding2}, $h(t)\rightarrow\infty$ and $l(t)\rightarrow\infty$, $t\rightarrow\infty$. The result is true. Next we prove $l(t)$ is unbounded by contradiction.

Assume $l(t)$ is bounded.

{\bf Step 1.} we are going to prove that $\lim\limits_{t\rightarrow\infty}b(t)$ exists. If $\liminf\limits_{t\rightarrow\infty}b(t)<\limsup\limits_{t\rightarrow\infty}b(t)$, we can choose $x_0$ such that
$$
\liminf\limits_{t\rightarrow\infty}b(t)<x_0<\limsup\limits_{t\rightarrow\infty}b(t).
$$
We consider the function $u_1(x)=\sqrt{1/A^2-(x+1/A-x_0)^2}$. Obviously, $x_0$ is the right endpoint of $u_1(x)$ and $(u_1(x),x_0-2/A,x_0)$ is the solution of the problem (Q) with $\theta_{\pm}=\pi/2$. So $b(t)-x_0$ changes sign infinite many times as $t$ varying over $[0,\infty)$. There exists a sequence $p_m\rightarrow\infty$ such that $u(x,p_m)$ intersects $u_1(x)$ tangentially at $x_0$. Arguing as in Lemma \ref{lem:inters}, the intersection number between $u(x,t)$ with $u_1(x)$ drops at $b(p_m)=x_0$. Therefore, the intersection number between $u(x,t)$ and $u_1(x)$ drops infinite many times. This yields a contradiction. Then we let $\nu:=\lim\limits_{t\rightarrow\infty}b(t)$.

{\bf Step 2.} We deduce the contradiction.

Since $h(s_m)\rightarrow\infty$, Lemma \ref{lem:alphad2} implies that for $t_1=4/A^2$, there holds $y=\rho$ intersects $y=u(x,t)$ only twice, $\rho<At_1$, $t>t_1$. Here we choose $\rho_0=2/A$. Then there exists $w(y,t)>0$ such that 
$$
C_{\rho_0}\cap\Gamma(t)=\{(x,y)\mid x=w(y,t)\ \text{or}\ x=-w(y,t)\}.
$$
$w(y,t)$ satisfies (\ref{eq:graph}) in the condition "$+$" and $n=1$, for $\{y\mid |y|<\rho_0\}\times(t_1,\infty)$. Since $w(0,t)=b(t)$ is bounded for $t>0$, by Corollary \ref{cor:es} and Remark \ref{rem:hes}, $\frac{\partial^kw}{\partial y^k}(y,t)$, $k=1,2,3$ are uniformly bounded for $|y|\leq\rho_0/2$, $t>t_1+\epsilon^2$.  From equation, $\frac{\partial^kw}{\partial t^k}w(y,t)$, $k=1,2$ are also bounded for $|y|<\rho_0/2$, $t>t_1+\epsilon^2$. So there exists $w_1(y,t)$, for any sequence satisfying $t_m\rightarrow \infty$ such that $w(\cdot,\cdot+t_m)$ converges to $w_1$ in $C^{2,1}([-\rho_0/2,\rho_0/2]\times[t_1+\epsilon^2,\infty))$ locally in time, as $m\rightarrow\infty$. Hence $w_1(y,t)$ also satisfies (\ref{eq:graph}) with the condition ``$+$'' and $n=1$. Moreover, $w_1(0,t)=\nu$ and $\frac{\partial}{\partial y}w_1(0,t)=0$, $t>t_1+\epsilon^2$.

Next, we consider the function $w_2(y)=\nu-1/A+\sqrt{1/A^2-y^2}$. $w_2(y)$ satisfies (\ref{eq:graph}) with the condition ``$+$'' and $n=1$. Moreover, $w_2(0)=\nu$ and $\frac{\partial}{\partial y}w_2(0)=0$.
So $w_1(y,t)$ intersects $w_2(y)$ at $y=0$ tangentially for all $t>t_1+\epsilon^2$. By the same argument as in Lemma \ref{lem:inters}, there holds $w_1(y,t)\equiv w_2(y)$, $|y|\leq \rho_0/2$. Noting $\frac{\partial w_2}{\partial y}(1/A)=\infty$, $\frac{\partial w_1}{\partial y}(1/A,t)=\infty$. But $w_y(1/A,t)$ is bounded, as $t\rightarrow\infty$, by gradient interior estimate. This is a contradiction.(Indeed, $w(y,t)$ has definition for $y\in(-2/A,2/A)$, but the limit function $w_2(y)$ has definition only in $[-1/A,1/A]$.)

We complete the proof.
\end{proof}

\begin{proof}[Proof of Theorem \ref{thm:threecondition}]
For $T<\infty$, seeing Corollary \ref{cor:2dsingular}, we get the conclusion.

For $T=\infty$, $h(t)$ is bounded or unbounded, as $t\rightarrow\infty$.

(1). $h(t)$ is unbounded. Lemma \ref{lem:1expanding} yields that $l(t)\rightarrow\infty$, $h(t)\rightarrow\infty$, as $t\rightarrow\infty$.

(2). $h(t)$ is bounded. 

By Corollary \ref{cor:expanding2}, $l(t)$ is also bounded. Next we want to prove $h(t)$ and $l(t)$ are bounded from below.

{\bf Step 1.} We prove if there exists a sequence $s_m\rightarrow\infty$, as $m\rightarrow\infty$ such that $h(s_m)\rightarrow0$, as $m\rightarrow\infty$, then $l(s_m)\rightarrow0$, as $m\rightarrow\infty$.

By Lemma \ref{lem:expanding} with $M=h(s_m)$,
$$
l(s_m)\leq l_{M,\tau}=2\sqrt{M^2+2R_0M}=2\sqrt{h(s_m)^2+2R_0h(s_m)}
$$
Then we have $l(s_m)\rightarrow0$.

{\bf Step 2.} $h(t)$ is bounded from below.

If there exists another sequence $t_m\rightarrow\infty$ such that $h(t_m)\rightarrow 0$, as $t_m\rightarrow\infty$, by Step 1, $l(t_m)\rightarrow 0$, as $t_m\rightarrow\infty$. Then there exists $t_{m_0}$ and $r<1/A$ such that $U(t_{m_0})\subset B_r((0,0))$, recalling $U(t)$ being the domain surrounded by $\Gamma(t)$. Then by comparison principle, we have $U(t+t_{m_0})\subset B_{r(t)}((0,0))$, where $r(t)$ is the solution of (\ref{eq:ball2}) with $r(0)=r$. Obviously, $B_{r(t)}((0,0))$ shrinks to origin in finite time. Then it is also for $U(t)$. This contradicts to $T=\infty$. 

Hence $h(t)$ is bounded from blew. 

{\bf Step 3.} Prove the result by contradiction. Assume there exists a sequence $s_m\rightarrow\infty$ such that $l(s_m)\rightarrow0$. 

Since $h(t)$ is bounded from below, by Lemma \ref{lem:alphad2}, there exist $\rho_0$ and $t_1>0$ such that for all $\rho<\rho_0$, $y=\rho$ intersects $y=u(x,t)$ only twice for $t>t_1$. Then we let $w(y,t)>0$ such that 
$$
C_{\rho_0}\cap\Gamma(t)=\{(x,y)\mid x=w(y,t)\ \text{or}\ x=-w(y,t)\}.
$$
 Arguing as the proof of Lemma \ref{lem:1expanding}, $\nu=\lim\limits_{t\rightarrow\infty}b(t)=0$, by $l(s_m)\rightarrow0$. And $w(\cdot,t)\rightarrow w_1$ in $C^{2,1}([0,\rho_0/2]\times[t_1+\epsilon^2,\infty))$ locally in time, as $m\rightarrow\infty$ and $w_1(y)=-1/A+\sqrt{1/A^2-y^2}\leq0$. But seeing $w(y,t)>0$, there holds $w_1(y)\geq0$ for $|y|<\rho_0/2$. Consequently, $w_1(y)\equiv0$, for $|y|<\rho_0/2$. Contradiction.

Therefore $h(t)$ and $l(t)$ are bounded from below.
\end{proof}

The conclusion of the Shrinking case in Theorem \ref{thm:asym} is obvious. We only need prove the case expanding and bounded.

\begin{proof}[Proof of the Expanding case in Theorem \ref{thm:asym}]
In this case, since $h(t)$ and $l(t)$ tend to infinity, using the same argument as in the proof of Corollary \ref{cor:expanding2}, there exist $t_0$ and $C>1/A$ such that $B_C((0,0))\subset U(t_0)$. By comparison principle, $B_{C(t)}((0,0))\subset U(t_0+t)$. Therefore, $B_{C(t-t_0)}((0,0))\subset U(t)$, $t\geq t_0$, where $C(t)$ satisfies (\ref{eq:ball2}) with $C(0)=C$.

On the other hand, seeing $U(0)$ being bounded, there exists $R>1/A$ such that $U(0)\subset B_R((0,0))$. Then $U(t)\subset B_{R(t)}((0,0))$, where $R(t)$ also satisfies  (\ref{eq:ball2}) with $R(0)=R$.

Denoting $R_1(t)=C(t-t_0)$ and $R_2(t)=R(t)$, $B_{R_1(t)}((0,0))\subset U(t)\subset B_{R_2(t)}((0,0))$, $t>t_0$. By the theory of ordinary equation, we can easily deduce that $\lim\limits_{t\rightarrow\infty}R_1(t)/t=\lim\limits_{t\rightarrow\infty}R_2(t)/t=A$. We complete the proof.
\end{proof}

\begin{proof}[Proof of the Bounded case in Theorem \ref{thm:asym}]

Since $h(t)$ is bounded from below, by Lemma \ref{lem:alphad2}, as before, there exist $t_1$, $\rho_0$ such that
$$
C_{\rho_0}\cap\Gamma(t)=\{(x,y)\mid x=w(y,t)\ \text{or}\ x=-w(y,t)\},\ t>t_1.
$$

{\bf Step 1.} Asymptotic behavior of $C_{\rho_0/2}\cap\Gamma(t)$. 

Arguing as the proof of Lemma \ref{lem:1expanding}, there exist $\nu$ and $w_2(y)$ such that $\nu=\lim\limits_{t\rightarrow\infty}b(t)$ and $w(\cdot,t)\rightarrow w_2$ in $C^{2,1}([-\rho/2,\rho_0/2])$, as $t\rightarrow\infty$, where $w_2(y)=\nu-1/A+\sqrt{1/A^2-y^2}$. 

{\bf Step 2.} Asymptotic behavior of $\{(x,y)\mid|y|\geq\rho_0/2\}\cap\Gamma(t)$.

Noting that $w_2(\rho_0/4)<\nu=\lim\limits_{t\rightarrow\infty}b(t)$, then for all $\epsilon>0$, there is $t_2$ such that $(-w_2(\rho_0/4)-\epsilon,w_2(\rho_0/4)+\epsilon)\subset(-b(t),b(t))$, $t>t_2$. We consider $u(x,t)$ in the following set $[-w_2(\rho_0/4),w_2(\rho_0/4)]\times(t_2,\infty)$. Because $u(x,t)$ satisfies (\ref{eq:graph}) under the condition $n=1$ and "$+$", $\frac{\partial ^k }{\partial x^k}u$, $k=1,2,3$, are uniformly bounded in $(-w_2(\rho_0/4)-\epsilon/2,w_2(\rho_0/4)+\epsilon/2)\times(t_2+\epsilon^2,\infty)$. Therefore $\frac{\partial^i}{\partial t^i}u$, $i=1,2$, $\frac{\partial ^k }{\partial x^k}u$, $k=1,2,3$, are uniformly bounded in $[-w_2(\rho_0/4),w_2(\rho_0/4)]\times(t_2+\epsilon^2,\infty)$

Next we want to show $\lim\limits_{t\rightarrow\infty}u(0,t)$ exists. If $\limsup\limits_{t\rightarrow\infty}u(0,t)>\liminf\limits_{t\rightarrow\infty}u(0,t)$, we can choose $y_0$ such that $\limsup\limits_{t\rightarrow\infty}u(0,t)>y_0>\liminf\limits_{t\rightarrow\infty}u(0,t)$. We consider the function $u_2(x)=y_0-1/A+\sqrt{1/A^2-x^2}$. By the same argument in the proof of Lemma \ref{lem:1expanding}, we get the contradiction. Denote $\mu:=\lim\limits_{t\rightarrow\infty}u(0,t)$. We can show, as in the proof of Lemma \ref{lem:1expanding}, $u(\cdot,t)\rightarrow u_3$ in $C^{2,1}([-w_2(\rho_0/4),w_2(\rho_0/4)])$, as $t\rightarrow\infty$, where $u_3(x)=\mu-1/A+\sqrt{1/A^2-x^2}$. 

{\bf Step 3.} Identify $\nu$ and $\mu$.

Since the graph of $y=u_2(x)$ and $x=w_2(y)$ are identical with each other, for $\rho/4<y<\rho/2$, then $\nu=\mu=1/A$. Consequently, $u(x,t)$ converges to $\varphi(x)=\sqrt{1/A^2-x^2}$, $x\in \mathbb{R}$, as $t\rightarrow\infty$. Here we consider $u(x,t)$ and $\varphi(x)$ as 0 outside the domains of definition.(Indeed, seeing the proof, $\lim\limits_{t\rightarrow \infty}d_H(\Gamma(t),\partial B_{1/A}((0,0))=0$). We complete the proof.
\end{proof}
\section{Appendix }
In this section, we want to prove there exists unique smooth family of smooth hypersurfaces $\Gamma(t)$ satisfying
\begin{equation}\label{eq:hcur}
V=-\kappa+A,\ \text{on}\ \Gamma(t)\subset \mathbb{R}^{n+1},
\end{equation}
where $\Gamma(0)=\partial U$ with $U$ is an $\alpha$-domain.

Seeing $\partial U$ is not necessary smooth, we also use the level set method and prove the interface evolution is not fattening.

\begin{defn}\label{def:alphad} We say a domain being an $\alpha$-domain in $\mathbb{R}^{n+1}$ if

(1) Let $U\subset \mathbb{R}^{n+1}$ be an open set of the form
$$
U=\{(x,y)\in\mathbb{R}\times\mathbb{R}^n\mid r<u(x)\}.
$$

(2) $I=\{x\in\mathbb{R}\mid u(x)>0 \}$ is a bounded, connected interval.

(3) $u$ is smooth on $I$;

(4) $\partial U$ intersects each cylinder $\partial C_{\rho}$ with $0<\rho\leq\alpha$ twice and these intersections are transverse, where $C_{\rho}=\{(x,y)\in\mathbb{R}\times\mathbb{R}^n\mid r<\rho\}$.
\end{defn}

For $U$ being $\alpha$-domain, we choose smooth vector field $X:\mathbb{R}^{n+1}\rightarrow\mathbb{R}^{n+1}$ such that
\\
(i) At any point $P\in\partial U$ not on the $x$-axis has $\langle X(P), \textbf{n}(P)\rangle<0$, $\textbf{n}$ is inward unit normal vector at $P$.
\\
(ii) Near the two end points of $\partial U$, $X$ is constant vector with $X\equiv\pm e_0=(\pm1,0,\cdots,0).$

Since $X\neq0$ on the compact $\partial U$, there is an open neighbourhood $V\supset\partial U$ on which $|X|\geq\delta>0$ for some $\delta>0$.

\begin{figure}[htbp]
	\begin{center}
            \includegraphics[height=4cm]{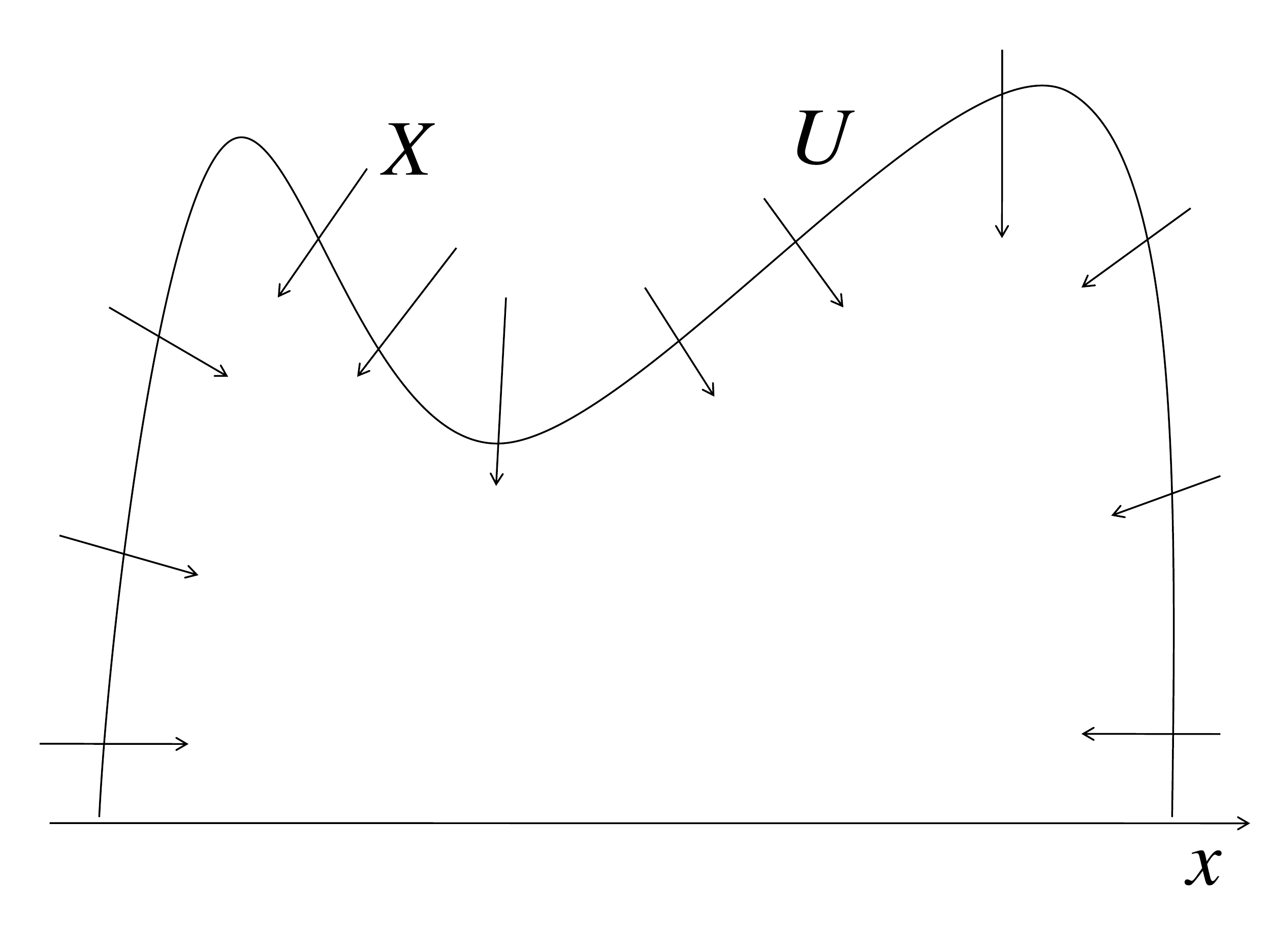}
		\vskip 0pt
		\caption{Vector field $X$}
        \label{fig:vectorX}
	\end{center}
\end{figure}

\begin{prop}\label{pro:sigma1}For small enough $\rho>0$ there exists a smooth hypersurface $\Sigma\subset V$ with\\
(i) $X(P)\notin T_P\Sigma$ at all $P\in\Sigma$, i.e., $\Sigma$ is transverse to the vector field $X$.\\
(ii) $\Sigma=\partial U$ in $\{(x,y)\in\mathbb{R}\times\mathbb{R}^{n}\mid |y|\geq2\rho\}$.\\
(iii) $\Sigma\cap\{(x,y)\in\mathbb{R}\times\mathbb{R}^{n}\mid|y|\leq\rho\}$ consists of two flat disks $\Delta_a=\{(a,y)\in\mathbb{R}\times\mathbb{R}^{n}\mid|y|\leq\rho\}$ and $\Delta_b=\{(b,y)\in\mathbb{R}\times\mathbb{R}^{n}\mid|y|\leq\rho\}$ for some $a<b$.
\end{prop}

Seeing Figure \ref{fig:sigma1}, this proposition can be proved as in Proposition \ref{pro:sigma2}.

\begin{figure}[htbp]
	\begin{center}
            \includegraphics[height=4cm]{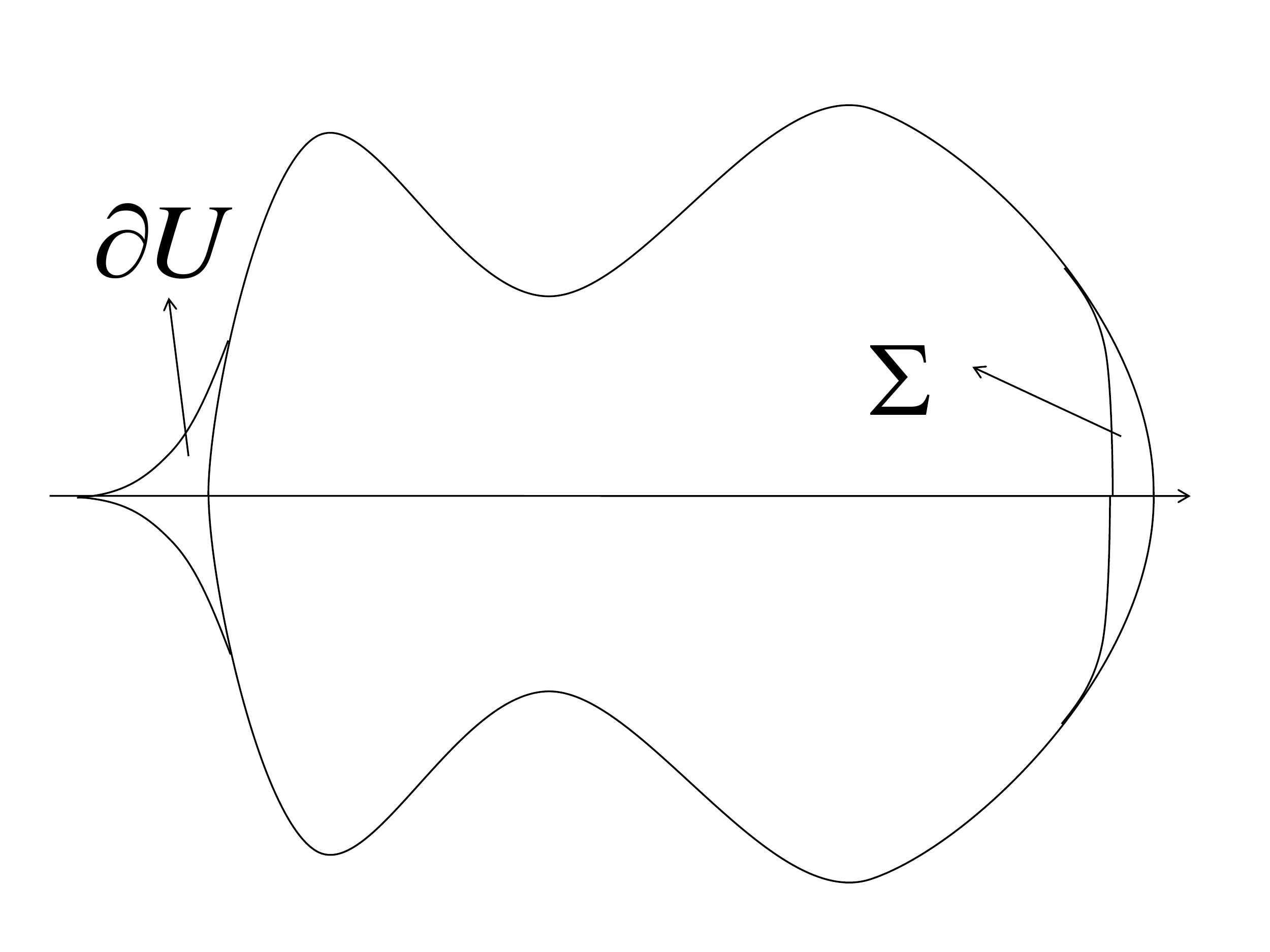}
		\vskip 0pt
		\caption{Proof of Proposition \ref{pro:sigma1}}
        \label{fig:sigma1}
	\end{center}
\end{figure}

Let $\phi^{t}:\mathbb{R}^{n+1}\rightarrow\mathbb{R}^{n+1}(t\in\mathbb{R})$, $t\in(-\delta,\delta)$ be the flow generated by vector field $X$ on $\mathbb{R}^{n+1}$ determined by
$$
\left\{
\begin{array}{lcl}
\dis{\frac{d\phi^t(P)}{dt}=X(\phi^t)},\ P\in \Sigma,\\
\phi^0(P)=P,\ \ \ \ \ P\in \Sigma.
\end{array}
\right.
$$

We denote $\sigma(P,s):=\phi^{s}(P)$. As in Section 5, suppose $\Gamma(t)\subset V$$(0<t<T)$ are smooth hypersurfaces with $\sigma^{-1}(\Gamma(t))$ being the graph $u(\cdot,t)$ for $u:\Sigma\times[0,T)\rightarrow\mathbb{R}$. Let $z_1,z_2,\cdots,z_n$ be local coordinates on an open subset of $\Sigma$. If $\Gamma(t)$ evolving by $V=-\kappa+A$, then in these coordinates $u$ satisfies the following parabolic equation
\begin{equation}\label{eq:para}
\frac{\partial u}{\partial t}=a_{ij}(z,u,\nabla u)\frac{\partial^2u}{\partial z_i\partial z_j}+b(z,u,\nabla u).
\end{equation}

\begin{figure}[htbp]
	\begin{center}
            \includegraphics[height=4cm]{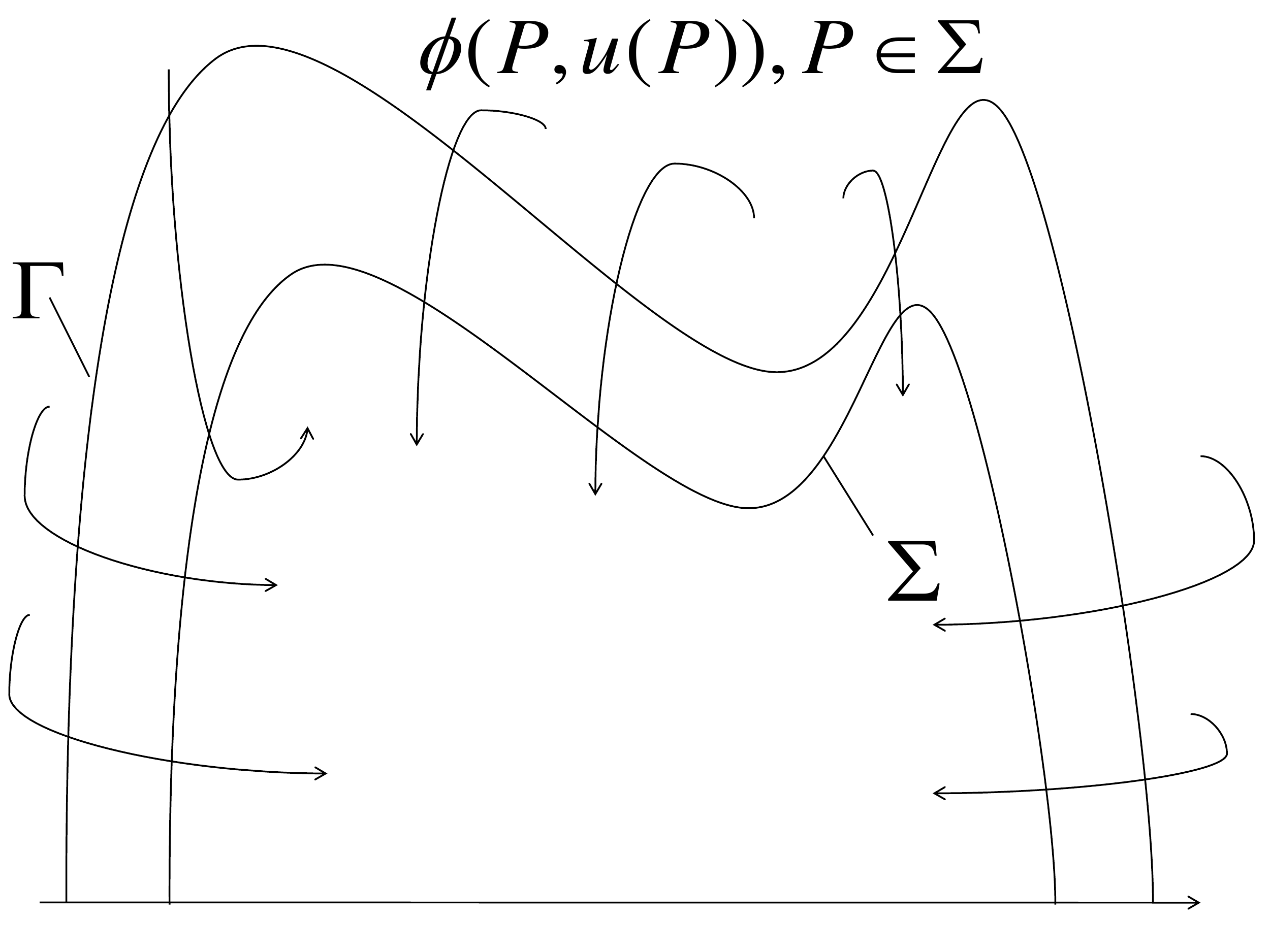}
		\vskip 0pt
		\caption{The transportation from $\Sigma$ to $\Gamma$}
        \label{fig:charact}
	\end{center}
\end{figure}
For example, on $\Delta_a$, by calculation, $\sigma(y_1,,y_2,\cdots,y_n,s)=(a-s,y_1,y_2,\cdots,y_n)$. Then $u$ satisfies the "$-$" condition of (\ref{eq:graph}).
\begin{prop}\label{pro:uniq} For $n\geq1$, let $\Gamma_1(t)$, $\Gamma_2(t)(0\leq t<T)$ be two families of hypersurface smooth and $\sigma^{-1}(\Gamma_j(t))$ be the graph of $u_j(\cdot,t)$ for certain $u_j\in C(\Sigma\times[0,T))$. Assume that the $u_j$ are smooth on $\Sigma\times(0,T)$ as well as on $\Sigma\setminus(\Delta_a\cup\Delta_b)\times[0,T)$. Then if the $\Gamma_j(t)$ evolve by $V=-\kappa+A$ and if $\Gamma_1(0)=\Gamma_2(0)$, then there holds $\Gamma_1(t)=\Gamma_2(t)$ for $0<t<T$.
\end{prop}
We use the same method in \cite{AAG}. 
The proof is similar as in Proposition \ref{pro:uniq2}. Here we omit it.

\begin{thm}\label{thm:partialUmeancurvature}
If $U$ is an $\alpha$-domain with smooth boundary, let $D(t)$ and $E(t)$ be the open and closed evolutions of $V=-\kappa+A$ with $D(0)=U$ and $E(0)=\overline{U}$. Then there exists $T>0$ such that $\partial D(t)$ and $\partial E(t)$ are smooth hypersurfaces for $0<t\leq T$ and $\partial D(t)=\partial E(t)$. Moreover, denoting $\Sigma(t)=\partial D(t)=\partial E(t)$, $\Sigma(t)$ can be written into $\Sigma(t)=\{(x,y)\in\mathbb{R}\times\mathbb{R}^n\mid |y|=u(x,t), a(t)\leq x\leq b(t)\}$ and $(u,a,b)$ is the solution of (Q) with $\theta_{\pm}=\pi/2$.
\end{thm}
\begin{proof}
We only give the sketch of the proof. By approximate argument similarly in Lemma 6.2 and Lemma 6.4, $\partial D(t)$ and $\partial E(t)$ are smooth hypersurfaces and can be represented by $\sigma(P,u_j(P))$, for some $u_j$, $j=1,2$. Then we can use Proposition \ref{pro:uniq} to prove $\partial D(t)=\partial E(t)$. Therefore $\Gamma(t)=\partial E(t)$ can be represented by $\Gamma(t)=\{(x,y)\in\mathbb{R}\times\mathbb{R}^n\mid |y|=u(x,t), a(t)\leq x\leq b(t)\}$. Using Theorem \ref{thm:openevolutionmeancurvature}, $(u,a,b)$ is the solution of (Q) with $\theta_{\pm}=\pi/2$.
\end{proof}

{\bf Acknowledgment}

The author expresses his hearty thanks to Professor Hiroshi Matano and Professor Yoshikazu Giga for their stimulating suggestions. The author learned the content about extended intersection number principle from Professor Hiroshi Matano. The author learned the techniques about viscosity solutions and formation of singularity in Section 6 contained in \cite{A2} from Professor Yoshikazu Giga. The author is grateful to the anonymous referee for valuable suggestion to improve the presentation of this paper.

\end{document}